\documentclass[10pt,a4paper]{article}
\pdfoutput=1

\usepackage[latin1]{inputenc}
\usepackage[T1]{fontenc}
\usepackage{microtype}

\usepackage{enumerate}

\usepackage{xcolor}

\usepackage{amsmath}
\usepackage{amsfonts}
\usepackage{amssymb}
\usepackage{amsthm}
\usepackage[mathic]{mathtools}

\usepackage[linesnumbered,ruled,vlined,boxed]{algorithm2e}
\newlength{\commentWidth}
\newcommand{\atcp}[1]{\tcp*[r]{\makebox[\commentWidth]{\textit{#1}\hfill}}}
\let\oldnl\nl
\newcommand{\nonl}{\renewcommand{\nl}{\let\nl\oldnl}}

\usepackage[ocgcolorlinks]{hyperref}
\colorlet{DarkRed}{red!50!black}
\colorlet{DarkGreen}{green!50!black}
\colorlet{DarkBlue}{blue!50!black}
\hypersetup{
	linkcolor = DarkRed,
	citecolor = DarkGreen,
	urlcolor = DarkBlue,
	bookmarksnumbered = true,
	linktocpage = true
}

\usepackage{tikz}
\usetikzlibrary{decorations.pathreplacing}
\usetikzlibrary{plotmarks}
\usetikzlibrary{positioning,automata,arrows}
\usetikzlibrary{shapes.geometric}

\PassOptionsToPackage{usenames,dvipsnames,svgnames}{xcolor}
\definecolor{orange}{RGB}{235,90,0}
\definecolor{darkorange}{RGB}{175,30,0}
\definecolor{turkis}{RGB}{131,182,182}
\definecolor{darkturkis}{RGB}{31,82,82}
\definecolor{green}{RGB}{102,180,0}
\definecolor{darkgreen}{RGB}{51,90,0}
\definecolor{myblue}{RGB}{0,0,213}
\definecolor{mydarkblue}{RGB}{0,0,100}
\definecolor{mybrightblue}{HTML}{74B0E4}
\definecolor{mybrighterblue}{HTML}{B3EAFA}
\definecolor{lila}{RGB}{102,0,102}
\definecolor{darkred}{RGB}{139,0,0}
\definecolor{darkyellow}{RGB}{188,135,2}
\definecolor{brightgray}{RGB}{200,200,200}
\definecolor{darkgray}{RGB}{50,50,50}
\definecolor{amaranth}{rgb}{0.9, 0.17, 0.31}
\definecolor{alizarin}{rgb}{0.82, 0.1, 0.26}
\definecolor{amber}{rgb}{1.0, 0.75, 0.0}
\definecolor{green(ryb)}{rgb}{0.4, 0.69, 0.2}
\definecolor{hanblue}{rgb}{0.27, 0.42, 0.81}
\definecolor{grannysmithapple}{rgb}{0.66, 0.89, 0.63}

\newtheorem*{theorem*}{Theorem}
\newtheorem*{lemma*}{Lemma}
\newtheorem*{corollary*}{Corollary}

\newtheorem{theorem}{Theorem}
\newtheorem{definition}{Definition}

\newtheorem{lemma}{Lemma}
\newtheorem{lemma+definition}{Lemma and Definition}

\newtheorem{corollary}{Corollary}

\newcommand{\set}[1]{\{#1\}}

\DeclareMathOperator{\LOG}{\ensuremath{\operatorname{log\mathllap{\raisebox{1.35ex}{\rule{0.95em}{0.08ex}}\hspace{0.025em}}}}}

\newcommand{\calC}{\mathcal{C}}

\newcommand{\wt}[1]{\widetilde{#1}}
\newcommand{\as}{{\mathrel{\,:=\,}}} 

\newcommand{\T}{\mathbf{T_*}}

\DeclareMathOperator{\mm}{max}
\newcommand{\MAX}{\ensuremath{\mm_1}}

\renewcommand{\P}{\ensuremath{\mathcal{P}}}

\DeclareMathOperator{\argmin}{argmin}
\DeclareMathOperator{\argmax}{argmax}

\DeclareMathOperator{\Mea}{Mea}
\DeclareMathOperator{\Disc}{Disc}

\DeclareMathOperator{\anc}{anc}
\newcommand{\ancstar}{\anc^*}

\let\eps\varepsilon

\newcommand{\BBB}{\mathcal{B}}

\newcommand{\MMM}{\mathcal{Z}}
\newcommand{\TTT}{\mathcal{T}}

\newcommand{\ZZ}{\ensuremath{\mathbb{Z}}}

\newcommand{\RR}{\ensuremath{\mathbb{R}}}
\newcommand{\CC}{\ensuremath{\mathbb{C}}}

\newcommand{\cisolate}{\ensuremath{\mathbb{C}\textsc{Isolate}}}

	\newcommand{\beql}[1]{\begin{equation}\label{eq:#1}}
	\newcommand{\eeql}{\end{equation}}
	\newcommand{\bitem}{\begin{itemize}}
	\newcommand{\eitem}{\end{itemize}}
	\newcommand{\benum}{\begin{enumerate}}
	\newcommand{\eenum}{\end{enumerate}}

\newcommand*\samethanks[1][\value{footnote}]{\footnotemark[#1]}

\begin{document}

\title{
	A Near-Optimal Subdivision Algorithm for Complex Root Isolation based on
	the Pellet Test and Newton Iteration
}

\author{
	Ruben Becker 
	\thanks{
		MPI for Informatics, Saarland Informatics Campus, Saarbr\"{u}cken, Germany.
	} 
	\thanks{
		MPC-VCC and Saarbr\"ucken Graduate School of Computer Science.
	} 
	\\ \small{\texttt{ruben@mpi-inf.mpg.de}}
	\and Michael Sagraloff 
	\samethanks[1]
	\\ \small{\texttt{msagralo@mpi-inf.mpg.de}}
	\and Vikram Sharma 
	\thanks{
		Institute of Mathematical Sciences Chennai, India.
	} 
	\\ \small{\texttt{vikram@imsc.res.in}}
	\and Chee Yap 
	\thanks{
		Courant Institute of Mathematical Sciences, New York University, New York, USA.
	}
	\\ \small{\texttt{yap@cs.nyu.edu}}
}
\date{}

\maketitle

\begin{abstract}
	We describe a subdivision algorithm for isolating the
	complex roots of a polynomial $F\in\mathbb{C}[x]$.
	Given an oracle that provides
	approximations of each of the coefficients of $F$ to any absolute error bound and given an arbitrary square $\mathcal{B}$ in the complex plane
	containing only simple roots of $F$, our algorithm
	returns disjoint isolating disks for the roots of $F$ in $\mathcal{B}$.

	Our complexity analysis bounds the absolute error to which the coefficients of $F$ have to be provided,
	the total number of iterations, and the overall bit complexity.
	It further shows that the complexity of our algorithm
	is controlled by the geometry of the roots in
	a near neighborhood of the input square $\mathcal{B}$, namely, the number of
	roots, their absolute values and pairwise distances.
	The number of subdivision steps is near-optimal.
	For the \emph{benchmark problem}, namely, to isolate all the roots of
	a polynomial of degree $n$ with integer
	coefficients of bit size less than $\tau$,
	our algorithm needs $\tilde{O}(n^3+n^2\tau)$ bit operations, which
	is comparable to the record bound of Pan (2002).
	It is the first time that such a bound has been achieved using subdivision
	methods, and independent of divide-and-conquer techniques
	such as Sch\"onhage's splitting circle technique.

	Our algorithm uses the
	quadtree construction of Weyl (1924) with two key ingredients:
	using Pellet's Theorem (1881) combined with Graeffe iteration,
	we derive a "soft-test" to count the number of roots in a disk.
	Using Schr\"oder's modified Newton operator combined with bisection,
	in a form inspired by the quadratic interval method from Abbot (2006),
	we achieve quadratic convergence towards root clusters.
	Relative to the divide-conquer algorithms,
	our algorithm is quite simple with the potential of being practical.
	This paper is self-contained: we provide pseudo-code
	for all subroutines used by our algorithm.
\end{abstract}

\section{Introduction}

The computation of the roots of a univariate polynomial is one of the best studied problems in the areas of
computer algebra and numerical analysis, nevertheless there are still a number of novel algorithms presented each year; see~\cite{McNamee2012239,McNamee:2002,McNamee2007,McNamee-Pan2013,Pan:history} for an extensive overview. One reason for this development is undoubtedly the great importance of the problem, which results from the fact that solutions for many problems from mathematics,
engineering, computer science, or the natural sciences make critical use of univariate root
solving.
Another reason for the steady research is that, despite the huge existing literature, there is still a large discrepancy between methods
that are considered to be efficient in practice and those that achieve good theoretical bounds. For instance,
for computing all complex roots of a polynomial, practitioners typically use Aberth's, Weierstrass-Durand-Kerner's and QR
algorithms. These iterative methods are relatively simple as, in each step, we only need to evaluate the given polynomial (and its derivative) at certain points.
They have been integrated in popular packages such as \textsc{MPSolve}~\cite{Bini-Fiorentino,DBLP:journals/jcam/BiniR14}
or \texttt{eigensolve}~\cite{DBLP:journals/jsc/Fortune02},
regardless
of the fact that their excellent empirical behavior has
not been entirely verified in theory. In contrast, there exist algorithms~\cite{Pan:survey,MSW-rootfinding2013,Pan:alg} that achieve near-optimal bounds
with respect to asymptotic complexity; however, implementations of these methods do not exist. The main reason
for this situation
is that these algorithms are quite involved and that they use a
series of asymptotically fast
subroutines (see \cite[p.~702]{Pan:alg}).
In most cases, this rules out a self-contained presentation, which makes it difficult to access such methods, not only for practitioners but also for researchers working
in the same area. In addition, for an efficient implementation, it would be necessary to incorporate a
sophisticated precision management and many implementation tricks. Even then, there might still be a considerable overhead due to the extensive
use of asymptotically fast subroutines, which does not show up in the asymptotic complexity bounds but is
critical for input sizes that can be handled on modern computers.

In this paper, we aim to resolve the above described discrepancy by presenting a subdivision algorithm for complex root isolation, which we denote by {\cisolate}. For our method, we mainly combine simple and well-known techniques such as the classical quad-tree construction by Weyl~\cite{Weyl}, Pellet's Theorem~\cite{rahman2002analytic}, Graeffe iteration~\cite{Graeffe49,householder-graeffe:59}, and Schr\"oder's modified Newton operator~\cite{Schroder1870}. In addition, we derive bounds on its theoretical worst-case complexity matching the best bounds currently known for this problem; see Section~\ref{subsec:main results} for more details. Hence, we hope that our contribution will finally bring together theory and practice in the area of complex root finding.
In this context, it is
remarkable that, for the complexity results, we do not require any asymptotically fast subroutines except the classical fast algorithms for polynomial multiplication and Taylor shift computation.
Our presentation is self contained and we provide pseudo-code for all subroutines. Compared to existing asymptotically fast algorithms, our method is relatively simple and has the potential of being practical.

In theory, the currently
best algorithm for complex root finding goes back to Sch\"onhage's splitting circle method~\cite{schonhage:fundamental}, which
has been considerably refined by Pan~\cite{Pan:alg} and others~\cite{Kirrinnis1998378,Neff199681}. In~\cite{Pan:alg}, Pan gives an algorithm for approximate polynomial
factorization with near-optimal arithmetic and bit complexity.\footnote{Pan considers a similar model of computation, where it is assumed that the coefficients of the input polynomial are
	complex numbers that can be accessed to an arbitrary precision. Then, for
	a polynomial $F$ with roots $z_1,\ldots,z_n$ contained in the unit disk and an integer $L\ge n\log n$, Pan's
	algorithm computes approximations $\tilde{z}_i$ of $z_i$ with
	$\|F-\operatorname{lcf}(F)\cdot \prod_{i=1}^n (x-\tilde{z}_i)\|_1<2^{-L}\cdot\|F\|_1$ using only $\tilde{O}(n\log L)$
	arithmetic operations with a precision of $O(L)$. For a lower bound on the bit
	complexity of the approximate polynomial factorization, Pan considers a polynomial whose coefficients must be approximated with a precision of $\Omega(L)$
	as, otherwise, the above inequality is not fulfilled. This shows that already the cost for reading sufficiently good
	approximations of the input polynomial is comparable to the cost for running the entire algorithm. Hence, near-optimality of his algorithm follows. In the considered computational model, Pan's algorithm
	also performs near-optimal with respect to the Boolean complexity of the problem of approximating all roots.
	However, we remark that this \emph{does not} imply near-optimality of his method for the benchmark problem of isolating the complex roots of an integer polynomial. Namely, Pan's argument for the lower bound is based on a lower bound on the precision to which the coefficients have to be approximated. In the case of integer polynomials, the coefficients are given exactly, hence the cost for reading an arbitrary good approximation of the polynomial never exceeds the cost for reading the integer coefficients.}
From an approximate factorization, one can derive isolating disks for all complex roots. A
corresponding algorithm for complex root isolation, which uses Pan's method as a subroutine, has been presented
and analyzed in~\cite{MSW-rootfinding2013}. Its cost can be expressed in terms of (accessible) parameters that directly depend on the input such as the degree of $F$ and the size of its coefficients, but
also in terms of (hidden) geometric parameters such as the pairwise distances between the roots.
A special case, namely the so-called \emph{(complex) benchmark problem} of isolating all complex roots of a polynomial
$F$ with integer coefficients of bit size at most $\tau$, has attracted a lot of interest in the literature.
Using Pan's method~\cite{Pan:survey,MSW-rootfinding2013}, the latter problem can be solved with $\tilde{O}(n^2\tau)$ operations\footnote{With $\tilde O(\cdot)$, we indicate that poly-logarithmic factors are omitted, i.e., for a function $p$, we denote with $\tilde O(p)$ the set of functions in $O(p \log^c p)$, where $c$ is a constant.}, which constitutes
the current record bound for this problem.\footnote{So far, the bound $\tilde{O}(n^2\tau)$ can only be achieved by running Pan's factorization algorithm with an $L$
	of size $\Omega(n(\tau+\log n))$, which means that $\tilde{\Theta}(n^2\tau)$ bit operations are needed for any input polynomial; see~\cite[Theorem~3.1]{Pan:survey} for details. The adaptive algorithm from~\cite{MSW-rootfinding2013} needs $\tilde{O}(n^3+n^2\tau)$ bit operations, however its cost crucially depends on the hardness of the input polynomial (e.g.,~the separations of its roots), hence the actual cost is typically much lower.} So far, there exists no other method for complex root isolation that achieves a comparable bound.
For the \emph{real benchmark problem}, that is the isolation of the \emph{real} roots of a polynomial of degree $n$ with integer coefficients of bit size at most $\tau$, recent work~\cite{Sagraloff2015} describes a practical subdivision algorithm based on the Descartes method and Newton Iteration with bit complexity $\tilde{O}(n^3+n^2\tau)$.
An implementation of this method~\cite{DBLP:conf/issac/KobelRS16} is competitive with the fastest existing implementations~\cite{rouillier-zimmermann:roots:04} for real root isolation, and it shows superior performance for hard instances, where roots appear in clusters.
Our contribution is in the same line with~\cite{Sagraloff2015}, that is, both methods combine a subdivision approach, a simple predicate to test for roots, and Newton iteration to speed up convergence. The main difference is that we treat the more general problem of isolating all complex roots, whereas the algorithm from~\cite{Sagraloff2015} can only be used to compute the real roots, due to the use of Descartes' Rule of Signs to test for roots.

We further remark that, in comparison to global approaches such as \textsc{MPSolve}~\cite{Bini-Fiorentino,DBLP:journals/jcam/BiniR14},  which compute all complex roots in parallel, our algorithm can also be used for a local search for only the roots contained in some given square. In this case, the number of iterations as well as the cost of the algorithm adapt to geometric parameters that only depend on the roots located in some neighborhood of the given square.

\subsection{Overview of the Algorithm and Main Results}\label{subsec:main results}

We consider a polynomial
\begin{align}\label{def:polyF}
	F(x)=\sum_{i=0}^n a_i x^i\in\mathbb{C}[x],\quad\text{with }n\ge 2\text{ and }\frac{1}{4}<|a_n|\le 1.
\end{align}
Notice that, after multiplication with a suitable power of two, we can always ensure that the above requirement on the leading coefficient is fulfilled, without changing the roots of the given polynomial. It is assumed that the coefficients of $F$ are given by means of a coefficient oracle. That is, for an arbitrary $L$, the oracle provides a dyadic approximation $\tilde{a}_i$ of each coefficient $a_i$ that coincides with $a_i$ to $L$ bits after the binary point. We call an approximation $\tilde{F}$ obtained in this way an \emph{(absolute) $L$-bit approximation of $F$} and assume that the cost for asking the oracle
for an $L$-bit approximation of $F$ is the cost of reading such an approximation;\footnote{Notice that we only require approximations of the coefficients, hence our method also applies to polynomials with algebraic, or even transcendental coefficients. In any case, the given bounds for the cost of isolating the roots of such a polynomial do not encounter the cost for computing sufficiently good $L$-bit approximations of the coefficients. Depending on the type of the coefficients, this cost might be considerably larger than the cost for just reading such approximations.} see Section~\ref{sec:definitions} for more details.
Let us denote by $z_1$ to $z_n$ the roots of $F$, where each root occurs as often as determined by its multiplicity.
Now, given a closed, axis-aligned square $\mathcal{B}$ in the complex plane, our goal is to compute isolating disks for all roots of $F$ contained in $\mathcal{B}$. Since we can only ask for approximations of the coefficients, we need to further require that $\mathcal{B}$ contains only simple roots of $F$ as, otherwise, a multiple root of multiplicity $k$ cannot be distinguished from a cluster of $k$ nearby roots, and thus the problem becomes ill-posed. If the latter requirement is fulfilled, then our algorithm {$\cisolate$}
computes isolating disks for all roots contained in $\mathcal{B}$.\footnote{If the requirement is not fulfilled, our algorithm does not terminate. However, using an additional stopping criteria, it can be used to compute arbitrarily good approximations of all (multiple) roots; see the remark at the end of Section~\ref{subsec:algorithm} for more details.} However, it may also return
isolating disks for some of the roots contained in $2\mathcal{B}$, the square centered at $\mathcal{B}$
and of twice the size as $\mathcal{B}$. Our approach is based on Weyl's quad tree construction, that is, we
recursively subdivide $\mathcal{B}$ into smaller sub-squares and discard squares for which we can show that they do not contain a root of $F$. The remaining squares are clustered into maximal connected components, which are tested for being isolating for a single root.

As exclusion and inclusion predicate, we propose a test based on Pellet's theorem and Graeffe
iteration. We briefly outline our approach and refer to Section~\ref{sec:rootcounting} for more details.
Let $\Delta:=\Delta(m,r)\subset\CC$ be the disk centered at $m$ with radius $r$, and define $\lambda\cdot\Delta(m,r):=\Delta(m,\lambda\cdot r)$ for arbitrary $\lambda\in\RR^+$.
According to Pellet's theorem~\cite{rahman2002analytic}, the number of roots contained in $\Delta$
equals $k$ if the absolute value of the $k$-th coefficient of
$F_{\Delta}(x):=F(m+rx)$ dominates the sum of the absolute values of all other coefficients.
For $k=0$ and $k=1$, it has been known~\cite{Yap:2011:SBE:1993886.1993938,Yakoubsohn2000} that Pellet's theorem applies if the smaller disk $n^{-e_1}\cdot\Delta$ contains $k$ roots and the larger disk $n^{e_2}\cdot\Delta$ contains no further root, where $e_1$ and $e_2$ are suitable positive constants. In the paper at hand, we derive constants $e_1$ and $e_2$ such that the latter result stays true for all $k$. As a consequence, using only $O(\log\log n)$ Graeffe iteration for
iteratively squaring the roots of $F_{\Delta}$, we can replace the factors $n^{e_1}$ and $n^{e_2}$ by  the constants $\rho_1:=\frac{2\sqrt{2}}{3}\approx 0.94$ and $\rho_2:=\frac{4}{3}$. More precisely, we derive a test that allows us to exactly count the number of
roots contained in a disk $\Delta$, provided that the disks $\rho_2\cdot \Delta$ and $\rho_1\cdot\Delta$ contain the same
number of roots. If the latter requirement is not fulfilled, the test might return the value $-1$, in which case we have no information on the number of roots in $\Delta$.
Since, in general, the latter test requires exact arithmetic and since we can only ask for approximations of the coefficients of $F$, there might be cases, where we either cannot decide the outcome of our test or where an unnecessarily high precision is needed. Based on the idea of so-called soft-predicates~\cite{DBLP:conf/cie/YapS013},
we formulate a variant of the above test, which we denote by $\mathbf{T_*}$, that uses only approximate arithmetic and runs with a precision demand that is directly related to the maximal absolute value that $F$ takes on the disk $\Delta$.

In the subdivision process, we inscribe each square in a corresponding disk and run the $\mathbf{T_*}$-test
on this disk. Squares, for which the test $\mathbf{T_*}$ yields $0$, do not contain a root and can thus be discarded. The remaining
squares are clustered into maximal connected components, which we also inscribe in corresponding disks. If the
$\mathbf{T_*}$-test yields $1$ for such a disk, we discard the cluster and store the disk as an isolating disk.
Otherwise, we keep on subdividing each square into four equally sized sub-squares and proceed.
This approach on its own already yields a reasonably efficient algorithm,
however, only linear convergence against the roots can be achieved. As a
consequence, there might exist
long paths in the subdivision tree with no branching (there are at most
$n-1$ branching nodes).
For instance, when considering the benchmark problem, there exist polynomials (e.g.~so called Mignotte polynomials having two roots with a very small distance to each other) for which the length of such a sequence is lower bounded by $\Omega(n\tau)$.
We show how to traverse such sequences in a much faster manner, that is reducing their length to $O(\log(n\tau))$ in the worst-case, via a regula falsi method, which combines Newton
iteration and square quartering. Our approach is inspired by the so-called quadratic interval refinement (QIR for
short) method proposed by Abbott~\cite{abbott-quadratic}. It combines the secant
method and interval bisection in order to further refine an interval that is already known to be isolating for a root. In~\cite{Sagraloff:2012:NMD:2442829.2442872,Sagraloff:2014:NAC:2608628.2608632,Sagraloff2015}, the QIR approach has been considerably refined by replacing the
secant method by Newton iteration (i.e. Schr\"oder's modified Newton operator for multiple roots). Compared to Abbott's original variant, this yields a method with quadratic convergence against clusters of roots during the isolation process. Our approach is similar to the one from~\cite{Sagraloff2015}, however, we use the $\mathbf{T}_*$-test instead of Descartes' Rule of Signs, which only applies to real intervals. Furthermore, the approach from~\cite{Sagraloff2015} uses fast approximate multipoint evaluation~\cite{Kirrinnis1998378,DBLP:journals/corr/abs-1304.8069} in order to determine subdivision points whose distance to the roots of $F$ is not too small. This is needed to avoid an unnecessarily large precision when using Descartes' Rule of Signs.
For our algorithm {\cisolate}, there is no need for (fast) approximate multipoint evaluation. We now state our first main theoretical result, which shows that our algorithm performs near-optimal with respect to the number of produced squares:

\begin{theorem*}
	Let $F$ be polynomial as in (\ref{def:polyF}) and suppose that $F$ is square-free. For isolating all complex roots of $F$, the algorithm $\mathbb{C}\textsc{Isolate}$ produces a number of squares bounded by
	\[
		\tilde{O}\left(n\cdot\log(n)\cdot\log\left(n\cdot\Gamma_F\cdot\LOG(\sigma_F^{-1})\right)\right),
	\]
	where we define $\LOG(x):=\max(1,\log |x|)$ for arbitrary $x\in\CC$, $\Gamma_F:=\LOG(\max_{i=1}^n|z_i|)$ the \emph{logarithmic root bound} and $\sigma_F:=\min_{(i,j):i\neq j}|z_i-z_j|$ the \emph{separation of $F$}.
\end{theorem*}

For the benchmark problem, the above bound simplifies to $O(n\log(n)\log(n\tau))$.
When running our algorithm on an arbitrary axis-aligned square $\mathcal{B}$, we obtain refined bounds showing that our algorithm is also adaptive with respect to the number of roots contained in some neighborhood of $\mathcal{B}$ as well as with respect to their geometric location. Namely, suppose that the enlarged square $2\mathcal{B}$ contains only simple roots of $F$, then we may replace $n$, $\Gamma_F$, and $\sigma_F$ in the bound in the above theorem by the number of roots contained in the enlarged square $2\mathcal{B}$, the logarithm of the width of $\mathcal{B}$, and the minimal separation of the roots of $F$ contained in $2\mathcal{B}$, respectively; see also Theorem~\ref{mapping}.

Finally, we give bounds on the the bit complexity of our approach as well as on the precision to which the coefficients of $F$ have to be provided:

\begin{theorem*}
	Let $F$ be a polynomial as in (\ref{def:polyF}) and suppose that $F$ is square-free. For isolating all complex roots of $F$, the algorithm {\cisolate} uses a number of bit operations bounded by
	\begin{align}
		\nonumber
		  & \tilde{O}\left(\sum\nolimits_{i=1}^n n\cdot(\tau_F+n\cdot \LOG (z_i)+\LOG (\sigma_F(z_i)^{-1})+\LOG (F'(z_i)^{-1}))\right)= \\
		  & \tilde{O}(n(n^2 + n \LOG(\Mea_F) + \LOG (\Disc_F^{-1}))),\nonumber
	\end{align}
	where we define $\tau_F:=\lceil\LOG\|F\|_{\infty} \rceil$, $\sigma_F(z_i):=\min_{j\neq i}|z_i-z_j|$ the \emph{separation of $z_i$}, $\Mea_F:=|a_n|\cdot\prod_{i=1}^n\max(1,|z_i|)$ the \emph{Mahler Measure}, and $\Disc_F$ the \emph{discriminant of $F$}.
	As input, the algorithm requires an $L$-bit approximation of $F$ with
	\begin{align*}
		L & =\tilde{O}\left(\sum\nolimits_{i=1}^n (\tau_F+n\cdot \LOG (z_i)+\LOG (\sigma_F(z_i)^{-1})+\LOG (F'(z_i)^{-1}))\right) \\
		  & =\tilde{O}(n^2 + n \LOG(\Mea_F) + \LOG (\Disc_F^{-1})).
	\end{align*}
\end{theorem*}

Again, we also give refined complexity bounds for the problem of isolating all roots of $F$ contained in some square $\mathcal{B}$, which show that the cost and the precision demand of our algorithm adapt to the hardness of the roots contained in a close neighborhood of the square.
For the benchmark problem, the above bound simplifies to $\tilde{O}(n^3+n^2\tau)$. It is interesting that our bounds on the bit complexity for isolating all complex roots as achieved by {\cisolate} exactly match the corresponding bounds for the complex root isolation algorithm from~\cite{MSW-rootfinding2013}, which uses Pan's method for approximate polynomial factorization.

\subsection{Related Work}

As already mentioned at the beginning, there exists a huge literature on computing the roots of a univariate polynomial. This makes it simply impossible to give a comprehensive overview without going beyond the scope of a research paper, hence we suggest the interested reader to consult some of the
excellent surveys~\cite{McNamee2012239,McNamee:2002,McNamee2007,McNamee-Pan2013,Pan:history}. Here, we mainly focus on a comparison of
our method with other existing subdivision methods for real and complex root finding.

For real root computation, subdivision algorithms have become extremely popular due to their simplicity, ease of implementation, and  practical
efficiency. They have found their way into the most popular computer algebra systems, where they constitute the default
routine for real root computation.
Prominent examples of subdivision methods are the Descartes method~\cite{Collins-Akritas,eigenwillig-phd,Eigenwillig2005,rouillier-zimmermann:roots:04,Sagraloff:2012:NMD:2442829.2442872,Sagraloff2015,Sagraloff2014DSC,Sharma:2015:NOS:2755996.2756656}, the Bolzano method\footnote{The Bolzano method is based on Pellet's theorem (with $k=0$).
	It is used to test an interval $I$ for roots of the input polynomial $F$ and its derivative $F'$. $I$
	contains no root if Pellet's theorem applies to $F$. If it applies to $F'$, the function $F$ is monotone on
	$I$, and thus $I$ is either isolating for a root or it contains no root depending on whether there is a
	sign change of $F$ at the endpoints of $I$ or not.}~\cite{Becker12,DBLP:journals/jsc/BurrK12,Yap:2011:SBE:1993886.1993938}, the Sturm method~\cite{davenport:85,du-sharma-yap:sturm:07}, and the continued fraction method~\cite{akritas-strzebonski:comparison:05,sharma,tsigaridas13,te-cf:08}. From a high-level point of view, all of the above mentioned
methods essentially follow the same approach:
Starting from a given interval $I_0$, they recursively subdivide $I_0$ to search for the roots contained in $I_0$. Intervals that are shown to contain no root are discarded, and
intervals that are shown to be isolating for a simple root are returned. The two main differences between these algorithms are the choice of the exclusion
predicate and the way how the intervals are subdivided. For the \emph{real benchmark problem} of isolating all
real roots of a polynomial of degree $n$ with integer coefficients of bit size $\tau$ or less, most of the above
methods need $\tilde{O}(n\tau)$ subdivision steps and their worst-case bit complexity is bounded by
$\tilde{O}(n^4\tau^2)$. The bound on the number of subdivision steps stems from the fact that the product
of the separation of all roots is lower bounded by $2^{-\tilde{O}(n\tau)}$ and that only linear convergence
to the roots is achieved. By considering special polynomials (e.g., Mignotte polynomials) that have
roots with separation $2^{-\Omega(n\tau)}$, one can further show that the bound $\tilde{O}(n\tau)$ is even
tight up to logarithmic factors; see~\cite{Collins,Eigenwillig2005}. When using exact
arithmetic, the cost for each subdivision step is bounded by $\tilde{O}(n^3\tau)$ bit operations, which is
due to the fact that $n$ arithmetic operations with a precision of $\tilde{O}(n^2\tau)$ are performed.
In~\cite{Sagraloff2014DSC,Sagraloff2015}, it has been shown for the Descartes method that it suffices to work with a precision of size $\tilde{O}(n\tau)$ in order to isolate all real roots, a fact that has already been empirically verified in~\cite{rouillier-zimmermann:roots:04}.
This yields a worst-case bit complexity
bound of size $\tilde{O}(n^3\tau^2)$ for a modified Descartes method, which uses approximate instead of exact arithmetic.
For a corresponding modified variant of the Bolzano method~\cite{Becker12}, a similar argument yields the
same bound. Recent work~\cite{Sagraloff:2012:NMD:2442829.2442872,Sagraloff2015,Sharma:2015:NOS:2755996.2756656} combines the Descartes
method and Newton iteration, which yields algorithms with quadratic convergence in almost all iterations. They use only $O(n\log(n\tau))$ subdivision steps, which is near optimal. The methods
from~\cite{Sagraloff:2012:NMD:2442829.2442872,Sharma:2015:NOS:2755996.2756656} work for integer polynomials only and each computation is carried out with exact arithmetic. An amortized analysis of their cost yields the bound $\tilde{O}(n^3\tau)$ for the bit
complexity. \cite{Sagraloff2015} introduces an algorithm that improves upon the methods from~\cite{Sagraloff:2012:NMD:2442829.2442872,Sharma:2015:NOS:2755996.2756656} in two points. First, it can be used to compute the real roots of a polynomial
with arbitrary real coefficients. Second, due to the use of approximate arithmetic, its precision demand is
considerably smaller. For the real benchmark problem, it achieves the bit complexity bound $\tilde{O}(n^3+n^2\tau)$. More precisely, it needs $\tilde{O}(n\log(n\tau))$ iterations, and, in each iteration, $\tilde{O}(n)$ arithmetic operations are carried out with an average precision of size $\tilde{O}(n+\tau)$. This essentially matches the bounds achieved by our algorithm {\cisolate} for complex root isolation.
{\cisolate} shares common elements with the method from~\cite{Sagraloff2015}, however we had to develop novel tools to accommodate the fact that our search area
is now the entire complex plane and not the real axis. In particular, we replaced Descartes' Rule of
Signs, which serves as the test for real roots in~\cite{Sagraloff2015}, by our novel test $\mathbf{T_*}$ for counting the number of complex roots in a disk.

For computing the complex roots, there also exist a series of subdivision methods (e.g.~\cite{Collins:1992:EAI:143242.143308,mt-mega-2009,Mourrain2002612,Pan2000213,Pinkert:1976,DBLP:journals/jc/Renegar87,Yap:2011:SBE:1993886.1993938,Wilf:1978,Yakoubsohn2005652}); however, only a few algorithms have been analyzed in a way that allows a direct comparison with our method.
The earliest algorithm most relevant to our work is Weyl's~\cite{Weyl}.
He proposed a subdivision based algorithm
for computing a $2^{-b}$-relative approximation to all the roots of a polynomial, which is a slightly different
problem then root isolation. The
inclusion and exclusion tests are based on
estimating the distance to a nearest root from the center of a box, or what are called
proximity tests in the literature. The {\em arithmetic complexity} of the algorithm
is ${O}(n^3b\log n)$, when not using asymptotically fast polynomial arithmetic.
The problem with Weyl's approach, indeed with any approach based on subdivision,
is the linear convergence to the roots. The convergence factor was improved by Renegar
\cite{DBLP:journals/jc/Renegar87} and Pan \cite{Pan2000213} by considering a combination
of subdivision with Newton iteration.
Renegar~\cite{DBLP:journals/jc/Renegar87} uses the Schur-Cohn
algorithm~\cite[Section~6.8]{henrici1974} as an exclusion test (rather than the proximity tests of Weyl).
In addition, he introduces a subroutine for approximating the winding
number of a polynomial around the perimeter of some disk, and thus
a method for counting the number of roots of the polynomial in a disk. Once the number $k$ of
roots in a disk is known, a fixed number (depending on the degree and the radius of
the disk) of Newton steps are applied to the $(k-1)$-th derivative of the polynomial,
which guarantees quadratic convergence to
a cluster containing $k$ roots. The arithmetic complexity of
Renegar's algorithm for the problem of approximating the roots is $O(n^2 \log b + n^3\log n)$
without using asymptotically fast polynomial arithmetic.
The improvement over Weyl's result is basically due to the 
quadratic convergence obtained by the use of Newton iteration.

Pan \cite{Pan2000213} describes another modification of Weyl's approach that has arithmetic cost
$O((n^2\log n) \log (bn))$, which is an
improvement over Renegar's algorithm since the dependence
on the degree is a quadratic factor in $n$. The exclusion test is based on a combination
of Turan's proximity test \cite{turan1984new} and Graeffe iteration. Note that the asymptotic complexity
of these tests is $\wt{O}(n)$, whereas a straightforward implementation of the
Schur-Cohn test takes $O(n^2)$ arithmetic operations; the difference in the cost of these
exclusion tests is the reason behind the the improvement in the complexity estimate of Pan's algorithm
compared to Renegar's.
The algorithm in \cite{Pan2000213} recursively interchanges
Schr\"oder's iteration (a modification of Newton's iteration
to handle multiple roots) and Weyl's subdivision process.
As in the case of Renegar, the former is needed to approximate a cluster of roots, and if that fails to happen,
the subdivision is used to break up the set of roots into smaller subsets, and continue
recursively.
The transition between the iteration phase and the subdivision process is based
on estimating the root radii \cite[Section~14]{Schoenhage82}, and is perhaps more adaptive
than Renegar's approach. To estimate the number of roots
inside a disk (which is needed to estimate the size of a cluster), Pan uses a combination of the
winding number algorithm along with Graeffe iteration to ensure that there are no roots close
to the boundary of the disk; as suggested by Pan, one can alternatively use the root radii algorithm
without affecting the complexity significantly.
The analysis of the algorithm has two steps. First, is to bound the number
of boxes computed in the subdivision phase. This is done by considering the connected components
of the boxes and bounding the number of boxes in each component in terms of the number of roots
inside a slight scaling of the smallest disk containing the component; in our case, the bound
on the number of boxes is obtained by mapping the components
to appropriate roots (see Theorem~\ref{mapping}); the resulting bound
is comparable in both cases (see (\cite[Prop.~8.3]{Pan2000213} in Pan
and Theorem~\ref{thm:treesize} below).
The second step of the analysis shows that for
certain well separated clusters Newton iteration gives us quadratic convergence to the cluster
\cite[Lem.~10.6]{Pan2000213};
an analogous result is also derived by Renegar \cite[Cor.~4.5]{DBLP:journals/jc/Renegar87},
and by us (Lemma~\ref{newtonsucceeds}).
Some of the key differences between the approach in this paper and Pan's \cite{Pan2000213}
are the following:
we use Pellet's test combined with Graeffe iteration for both the exclusion test and detecting a cluster;
we use a modification of the QIR method \cite{abbott-quadratic} for multiple roots, which is more
adaptive in transitioning between the quadratic convergence and subdivision phases.
In terms of the results derived, perhaps the most important difference is that
we bound the {\em bit complexity} of our algorithm. In comparison, neither Renegar
nor Pan analyze the precision demand or the Boolean complexity of
their algorithms.

Similar to our method, Yakoubsohn~\cite{Yakoubsohn2005652} combines Weyl's quad tree approach and a test for roots based on Pellet's theorem. However, since only an exclusion predicate (based on Pellet's theorem with $k=0$) is considered but no additional test to verify that a region is isolating,
his method does not directly compute isolating regions but arbitrary good approximations of the complex roots. In~\cite{Yap:2011:SBE:1993886.1993938}, we introduced a variant of Yakoubsohn's method, denoted by \textsc{Ceval}, that computes isolating disks for the complex roots of an integer polynomial. There, an additional inclusion test (based on Pellet's theorem with $k=1$) has been used to show that a disk is isolating for a root. The methods from~\cite{Yap:2011:SBE:1993886.1993938,Yakoubsohn2005652} only consider square-quartering, and thus nothing better than linear convergence can be achieved. For the benchmark problem, the algorithm from~\cite{Yap:2011:SBE:1993886.1993938} needs $\tilde{O}(n^2\tau)$ subdivision steps and its cost is bounded by $\tilde{O}(n^4\tau^2)$ bit operations.
Yakoubsohn further mentions how to improve upon his method by combining the exclusion predicate with Graeffe iterations, which yields an improvement by a factor of size $n$ with respect to the total number of produced squares. In~\cite{Giusti2005}, an extension of Pellet's theorem for analytic functions has been considered and thoroughly analyzed. The authors also derive further criteria to detect clusters of roots of such functions, and to determine their multiplicities and diameters. This allows for the computation of suitable starting points for which Schr\"oder's modified Newton operator yields quadratic convergence to the cluster.
In contrast, we follow the approach of combining Pellet's theorem and Graeffe iteration to derive a simple test for detecting clusters of roots. However, we do neither compute the diameter of such a cluster nor do we consider any additional computations to check whether quadratic convergence to the cluster can be achieved. Instead, we rely on a trial-error approach that performs Schr\"oder's modified Newton operator by default and then checks for success.
 We show that this can be done in a certified manner such that quadratic convergence to clusters is guaranteed for all but only a small number of iterations, where our method falls back to bisection. Our approach works well with polynomials whose coefficients can only be approximated and we derive precise bounds on the precision demand in the worst-case.

In our previous work~\cite{DBLP:conf/cie/YapS013},
we provided the first complete algorithm for computing $\epsilon$-clusters
of roots of analytic functions. Like the present work, it is a subdivision
approach based on the $T_k$-test of Pellet; but unlike this paper,
it does not have quadratic convergence nor complexity analysis.
In~\cite{DBLP:conf/cie/YapS013},
we assumed that an analytic function is given when we also
have interval evaluation of its derivatives of any desired order;
this natural assumption is clearly satisfied by most common analytic functions.
The algorithm from~\cite{DBLP:conf/cie/YapS013} does not compute isolating
disks but arbitrary small regions containing clusters of roots, hence being also
applicable to functions with multiple roots and for which separation bounds are
not known.

\subsection{Structure of the Paper and Reading Guide}

In Section~\ref{sec:definitions}, we summarize the most important definitions and notations, which we will use throughout the paper. We suggest the reader to print a copy of this section in order to quickly refer to the definitions.
We introduce our novel test $\T$ for counting the roots in a disk in Section~\ref{sec:rootcounting}. The reader who is willing to skip all details of this section and who wants to proceed directly with the main algorithm should only consider the summary given at the beginning of Section~\ref{sec:rootcounting}, where we give the main properties of the $\T$-test.
The algorithm {\cisolate} is given in Section~\ref{sec:algorithm}. Its analysis is split into two parts. In Section~\ref{subsec:treesize}, we derive bounds on the number of iterations needed by our algorithms, whereas, in Section~\ref{subsec:bitcomplexity}, we estimate its bit complexity. Some of the (rather technical) proofs are outsourced to an appendix, and we recommend to skip these proofs in a first reading of the paper. In Section~\ref{sec:conclusion}, we summarize and hint to some future research.


\section{Definitions and a Root Bound}\label{sec:definitions}

Let $F$ be a polynomial as defined in (\ref{def:polyF}) with complex roots $z_1,\ldots,z_n$.
We fix the following definitions and denotations:
\begin{itemize}

	\item As mentioned in the introduction, we assume the existence of an oracle that provides arbitrary good approximations of the coefficients. More precisely, for an arbitrary non-negative integer $L$, we may ask the oracle for dyadic approximations $\tilde{a}_i=\frac{m_i}{2^{L+1}}$ of the coefficients $a_i$ such that $m_i\in\ZZ+i\cdot\ZZ\in\mathbb{C}$ are Gaussian integers and $|a_k-\tilde{a}_k|<2^{-L}$ for all $k=0,\ldots,n$. We also say that $\tilde{a}_k$ approximates $a_k$ to $L$ bits after the binary point, and a corresponding polynomial $\tilde{F}=\tilde{a}_0+\cdots+\tilde{a}_n\cdot x^n$ with coefficients fulfilling the latter properties is called an \emph{(absolute) $L$-bit approximation of $F$}. It is assumed that the cost for asking the oracle for such an approximation is the cost for reading the approximations.

	\item For any non-negative integer $k$, we denote by $[k]$
	      the set $\set{1\ldots k}$ of size $k$.
	      For any set $S$ and any non-negative integer $k$, we write $\binom{S}{k}$
	      for the set of all subsets of $S$ of size $k$.
	\item $\MAX(x_1,\ldots,x_k):=\max(1,|x_1|,\ldots,|x_k|)$ for arbitrary $x_1,\ldots,x_k\in\mathbb{C}$, $\log \coloneqq \log_2$ the binary logarithm, and
	      \[
	      	\LOG(x_1,\ldots,x_k) \coloneqq \lceil \MAX(\log \MAX(x_1,\ldots,x_k)) \rceil.
	      \]
	      Notice that, if $|z| \le 2$ for some $z\in\mathbb{C}$, then $\LOG(z)$ is $1$. Otherwise, $\LOG(z)$ equals $\log |z|$ rounded up to the next integer.
	\item $\|F\|_{\infty}=\max\{|a_k|: k=0,\ldots, n\}$ denotes the \emph{infinity-norm of $F$}. We further define $\tau_F:=\LOG (\|F\|_{\infty})$, which bounds the number of bits before the binary point in the binary representation of any coefficient of $F$.
	\item $\Gamma_F:=\LOG (\max_{i=1}^n |z_i|)$ is defined as the \emph{logarithmic root bound of $F$}.
	\item $\operatorname{Mea}_F:=|a_n|\cdot\prod_{i=1}^n \MAX(z_i)$ is defined as the \emph{Mahler measure of $F$}.
	\item $\sigma_F(z_i):=\min_{j\neq i} |z_i-z_j|$ is defined as the \emph{separation of the root $z_i$}
	      and $\sigma_F:=\min_{i=1}^n \sigma_F(z_i)$ as the \emph{separation of $F$}.
	\item For an arbitrary region $\mathcal{R}\subset\mathbb{C}$ in the complex space, we define $\sigma_{F}(\mathcal{R}):=\min_{i:z_i\in\mathcal{R}}\sigma_F(z_i)$, which we call the \emph{separation of $F$ restricted to $\mathcal{R}$}. We further denote by
	      $\mathcal{Z}(\mathcal{R})$ the set of all roots of $F$ that are contained in $\mathcal{R}$, and by
	      $\Mea_F(\mathcal{R}):=|a_n|\cdot\prod_{z_i\in \mathcal{Z}(\mathcal{R})} \MAX(z_i)$ the
	      \emph{Mahler measure of $F$ restricted to $\mathcal{R}$.}
	\item We denote the interior of a disk in the complex plane with center $m\in\mathbb{C}$ and radius $r\in\mathbb{R}^+$ by $\Delta=\Delta(m,r)$. For short, we also write $\lambda\cdot\Delta$ to denote the disk $\Delta(m,\lambda\cdot r)$ that is centered at $m$ and scaled by a factor $\lambda\in\mathbb{R}^+$. We further use $F_{\Delta}(x)$ to denote the shifted and scaled polynomial $F(m+r\cdot x)$, that is,  $F_{\Delta}(x):=F(m+r\cdot x)$.
	\item A disk \emph{$\Delta$ is isolating for a root $z_i$} of $F$ if it contains $z_i$ but no other root of $F$. For a set $S$ of roots of $F$ and positive real values $\rho_1$ and $\rho_2$ with $\rho_1\le 1\le \rho_2$, we further say that a disk \emph{$\Delta$ is $(\rho_1,\rho_2)$-isolating for $S$} if $\rho_1\cdot\Delta$ contains exactly the roots contained in $S$ and $\rho_2\cdot\Delta\setminus \rho_1\cdot\Delta$ contains no root of $F$.
	\item Throughout the paper, we only consider squares
	      \[
	      	B=\{z=x+i\cdot y\in\mathbb{C}: x\in [x_{\min},x_{\max}]\text{ and } y\in [y_{\min},y_{\max}]\}
	      \]
	      in the complex space that are
	      \emph{closed, axis-aligned, and of width $w(B)=2^\ell$ for some $\ell\in\mathbb{Z}$} (i.e.,
	      $|x_{\max}-x_{\min}|=|y_{\max}-y_{\min}|=2^{\ell}$), hence, for brevity, these properties are not peculiarly mentioned. Similar as for disks, for a $\lambda\in\mathbb{R}^+$, $\lambda\cdot B$ denotes the scaled square of size $\lambda\cdot 2^{\ell}$ centered at $B$.
\end{itemize}

According to Cauchy's root bound (e.g.~see~\cite{yap-fundamental}), we have $|z_i|\le 1+\max_{i=0}^n \frac{|a_i|}{|a_n|}< 1+4\cdot 2^{\tau_F}$, and thus $\Gamma_F=O(\tau_F)$. In addition, it holds that $$\tau_F\le \LOG (2^n\cdot \Mea_F)\le n(1+\Gamma_F)\le 2n\Gamma_F.$$
Following~\cite[Theorem 1]{MSW-rootfinding2013} (or~\cite[Section 6.1]{Sagraloff2014DSC}), we can compute an integer approximation $\tilde{\Gamma}_{F}\in\mathbb{N}$ of $\Gamma_{F}$ with
$\Gamma_{F}+1\le \tilde{\Gamma}_{F}\le\Gamma_{F}+8\log n+1
$
using $\tilde{O}(n^{2}\Gamma_{F})$ many bit operations. For this, the coefficients of $F$ need to be approximated to $\tilde{O}(n\Gamma_F)$ bits after the binary point. From $\tilde{\Gamma}_{F}$, we then immediately derive an integer $\Gamma=2^{\gamma}$, with $\gamma:=\lceil \log \tilde{\Gamma}_{F}\rceil\in\mathbb{N}_{\ge 1}$, such that
\begin{align}\label{def:Gamma}
	\Gamma_{F}+1\le \tilde{\Gamma}_{F}\le \Gamma\le 2\cdot\tilde{\Gamma}_{F} \le 2\cdot(\Gamma_{F}+8\log n+1).
\end{align}
It follows that $2^{\Gamma}=2^{O(\Gamma_{F}+\log n)}$ is an upper bound for the modulus of all roots of $F$, and thus once can always restrict the search for roots to the set of all complex numbers of absolute value of at most $2^{\Gamma}$.


\section{Counting Roots in a Disk}\label{sec:rootcounting}


In this section, we introduce the
$\T(\Delta)$-test, which constitute our main ingredient to count the numbers of roots of $F$ in a given disk $\Delta$.  Here, we briefly summarize the main properties of the $\T(\Delta)$-test. The reader willing to focus on the algorithmic details of the root isolation algorithm is invited to read the following summary and skip the remainder of this section on a first read.

\begin{itemize}
	\item For a given polynomial $F$ as in (\ref{def:polyF}) and a disk $\Delta$, the $\T(\Delta)$-test always returns an integer $k\in \{-1,0,1,\ldots,n\}$. If $k\ge 0$, then $\Delta$ contains exactly $k$ roots of $F$. If $k=-1$, no further information on the number of roots in $\Delta$ can be derived; see Lemma~\ref{softtest:success2}, part~(\ref{softtest:secondpart}).
	\item If $\Delta$ is $(\rho_1,\rho_2)$-isolating for a set of $k$ roots of $F$, where $\rho_1=\frac{2\sqrt{2}}{3}\approx 0.94$ and $\rho_2=\frac{4}{3}$, then $\T(\Delta)$ returns $k$, see Lemma~\ref{softtest:success2}, part~(\ref{softtest:firstpart}).
	      In particular, $\T(\Delta)$ returns $0$ if $\frac{4}{3}\cdot\Delta$ contains no root.
	\item The cost for the $\T(\Delta)$-test is bounded by
	      \begin{align*}
	      	  & \tilde{O}(n(\tau_F+n\LOG(m,r)+\LOG(\|F_{\Delta}\|_{\infty}^{-1}))) \\
	      	  & \hspace{3cm}=
	      	\tilde{O}(n(\tau_F+n\LOG(m,r)+\LOG((\max_{z\in \Delta}|F(z)|)^{-1})))
	      \end{align*}
	      bit operations, and thus
	      directly related to the size of $\Delta$ and the maximum absolute value that $F$ takes on the disk $\Delta$. For this, the test requires an $L$-bit approximation of $F$, with
	      \begin{align*}
	      	L & =\tilde{O}(\tau_F+n\LOG(m,r)+\LOG(\|F_{\Delta}\|_{\infty}^{-1}))      \\
	      	  & =	\tilde{O}(\tau_F+n\LOG(m,r)+\LOG((\max_{z\in \Delta}|F(z)|)^{-1})),
	      \end{align*}
	      see Lemma~\ref{allkcost}.
	      Here, we used that $\max_{z\in\Delta}|F(z)|\le (n+1)\cdot \|F_{\Delta}\|_{\infty}$ as shown in~\eqref{formula:FinftyDelta} in the proof of Theorem~\ref{graeffe thm}.
\end{itemize}

\subsection{Pellet's Theorem and the \texorpdfstring{$T_k$}{Tk}-Test}\label{sec:pellet}

In what follows, let $k$ be an integer with $0\le k\le n=\deg F$, and let
$K\in\RR$ with $K\ge 1$. We consider the following test, which allows us to compute the size of a cluster of roots contained in a disk $\Delta(m,r)$:

\begin{definition}[The $T_k$-Test]
	\emph{ For a polynomial $F\in\CC[x]$, the $T_k$-test on a disk $\Delta:=\Delta(m,r)$ with parameter $K$ holds} if
	\begin{align}
		T_k(m,r,K,F):\quad
		\left| \frac{F^{(k)}(m) r^k }{k!} \right|
		>
		K\cdot \sum_{i\neq k} \left| \frac{F^{(i)}(m) r^i }{i!} \right| \label{Tktest}
	\end{align}
	or, equivalently, if $F^{(k)}(m)\neq 0$ and
	\begin{align}
		T_k(m,r,K,F):\quad
		\sum_{i<k} \left| \frac{F^{(i)}(m) r^{i-k} k!}{F^{(k)}(m) i!} \right|
		+
		\sum_{i> k} \left| \frac{F^{(i)}(m) r^{i-k} k!}{F^{(k)}(m) i!} \right|
		< \frac{1}{K}.\label{Tktest:splitsum}
	\end{align}
\end{definition}

Mostly, we will write $T_k(\Delta,K,F)$ for $T_k(m,r,K,F)$, or simply $T_k(\Delta,K)$ if it is clear from the context which polynomial $F$ is considered.
Notice that if the $T_k$-test succeeds for some parameter $K=K_0$, then it also succeeds for any $K$ with $K\le K_0$. Clearly, $T_k(m,r,K,F)$ is equivalent to $T_k(0,1,K,F_{\Delta})$, with $F_{\Delta}(x):=F(m+r\cdot x)$.

The following result is a direct consequence of Pellet's theorem, and, in our algorithm, it will turn out to be crucial in order to compute the size of a cluster of roots of $F$; see~\cite[Section 9.2]{rahman2002analytic} or \cite{DBLP:conf/cie/YapS013} for a proof.

\begin{theorem}\label{thm:pellet}
	If $T_k(m,r,K,F)$ holds for some $K\in\mathbb{R}$ with $K\ge 1$ and some $k\in\{0,\ldots,n\}$, then $\Delta(m,r)$ contains exactly $k$ roots of $F$ counted with multiplicities.
\end{theorem}

We derive criteria on the locations of the roots $z_1,\ldots,z_n$ of $F$ under which the $T_k$-test is guaranteed to succeed:

\begin{theorem}
	\label{tktest thm}
	Let $k$ be an integer with $0\le k\le n=\deg(F)$, let $K\in\RR$ with $K\ge 1$, and let $c_1$ and $c_2$ be arbitrary real values fulfilling
	\begin{align}\label{def:constants}
		c_2 \cdot n\cdot \ln\left(\frac{1+2K}{2K}\right)\ge c_1 \cdot n\ge \frac{\MAX(k)}{\ln(1+\frac{1}{8K})}.
	\end{align}
	For a disk $\Delta=\Delta(m,r)$, suppose that there exists a real $\lambda$ with

	\[
		\lambda\ge\max(4c_2\cdot \MAX(k)\cdot n^3,16K\cdot\MAX(k)^2\cdot n)
	\]
	such that $\Delta$ is $(1,\lambda)$-isolating for the roots $z_1,\ldots, z_k$ of $F$, then $T_k(c_1n\cdot\Delta,K,F)$ holds.
\end{theorem}

In our algorithm, we will only make use of Corollary~\ref{tktest cor2}, which is actually a consequence of Theorem~\ref{tktest thm} with the specific values $K:=\frac{3}{2}$, $c_1:=16$, $c_2:=64$, $\lambda=256n^5$, and thus $\frac{\MAX(k)}{\ln(1+\frac{1}{8K})}\approx 12.49\cdot\MAX(k)$ and $\ln\left(\frac{1+2K}{2K}\right)\approx 0.29$.

\begin{corollary}\label{tktest cor2}
	Let $\Delta$ be a disk in the complex space that is $(\frac{1}{16n},16n^4)$-isolating for a set of $k$ roots (counted with multiplicity) of $F$. Then, $T_k(\Delta,\frac{3}{2},F)$ holds.
\end{corollary}

The proof of Theorem~\ref{tktest thm} is given in the appendix.
In the proof, we separately bound the two sums in~\eqref{Tktest:splitsum}. We also derive a bound on the minimal distance between a root of the $k$-th derivative $F^{(k)}$ of $F$ and a cluster of $k$ roots of $F$. Pawlowski~\cite{pawlowski1999} provides a similar but more general bound, which implies a bound on the first sum in~\eqref{Tktest:splitsum}. However, compared to~\cite{pawlowski1999}, our proof is significantly shorter and uses only simple arguments, hence we decided to integrate it in the appendix of this paper for the sake of a self-contained presentation.


\subsection{The \texorpdfstring{$T_k^G$}{TkG}-Test: Using Graeffe Iteration}\label{sec:Graeffe}

Corollary~\ref{tktest cor2} guarantees success of the $T_k(\Delta, 3/2, F)$-test, with $k=|\mathcal{Z}(\Delta)|$, if the disk $\Delta$ is $(\frac{1}{16n},16n^4)$-isolating for a set of $k$ roots. In this section, we use a well-known approach for squaring the roots of a polynomial, called Graeffe iteration~\cite{Graeffe49}, in order to improve upon the $T_k$-test. More specifically, we derive a variant of the $T_k$-test, which we denote $T_k^G$-test\footnote{The superscript ``$G$'' indicates the use of Graeffe iteration.}, that allows us to exactly count the roots contained in some disk $\Delta$ if $\Delta$ is $(\rho_1,\rho_2)$-isolating for a set of $k$ roots, with constants $\rho_1$ and $\rho_2$ of size $\rho_1\approx 0.947$ and $\rho_2=\frac{4}{3}$.

\begin{definition}[Graeffe Iteration]
	For a polynomial $F(x)=\sum_{i=0}^n a_{i}x^i \in\mathbb{C}[x]$, write $F(x)= F_e(x^2) + x\cdot F_o(x^2)$, with
	\begin{align*}
		  & F_e(x):=
		a_{2 \lfloor \frac{n}{2}\rfloor}x^{\lfloor \frac{n}{2} \rfloor }
		+ a_{2\lfloor \frac{n}{2}\rfloor -2} x^{\lfloor \frac{n}{2} \rfloor -1}
		+ \ldots
		+ a_{2} x
		+ a_0, \quad \text{ and }\\
		  & F_o(x):=
		a_{2 \lfloor \frac{n-1}{2}\rfloor +1}x^{\lfloor \frac{n-1}{2} \rfloor }
		+ a_{2 \lfloor \frac{n-1}{2}\rfloor -1} x^{\lfloor \frac{n-1}{2} \rfloor -1}
		+ \ldots
		+ a_{3} x
		+ a_{1}.
	\end{align*}
	Then, the \emph{first Graeffe iterate $F^{[1]}$ of $F$} is defined as:
	\[
		F^{[1]}(x):=(-1)^n [F_e(x)^2 - x\cdot F_o(x)^2].
	\]
\end{definition}

The first part of the following theorem is well-known (e.g. see~\cite{Graeffe49}), and we give its proof only for the sake of a self-contained presentation. For the second part, we have not been able to find a corresponding result in the literature. Despite the fact that we consider the result to be of independent interest, we will need it in the analysis of our approach.
\begin{theorem}
	\label{graeffe thm}
	Denote the roots of $F$ by $z_1,\ldots,z_n$, then it holds that $F^{[1]}(x)=\sum_{i=0}^n a_i^{[1]}x^i = a_n^2\cdot \prod_{i=1}^n(x-z_i^2)$. In particular, the roots of the first Graeffe iterate $F^{[1]}$ are the squares of the roots of $F$. In addition, we have
	\[
		n^2\cdot \MAX(\|F\|_\infty)^2\ge\|F^{[1]}\|_\infty\ge \|F\|_\infty^2 \cdot 2^{-4n}.
	\]
\end{theorem}

\begin{proof}
	See Appendix~\ref{appendix:graeffe}.
\end{proof}

\begin{algorithm}[t]\label{algo:graeffe}
	\DontPrintSemicolon
	\SetKwData{Left}{left}\SetKwData{This}{this}\SetKwData{Up}{up}
	\SetKwFunction{Union}{Union}\SetKwFunction{FindCompress}{FindCompress}
	\SetKwInOut{Input}{Input}\SetKwInOut{Output}{Output}
	\Input{Polynomial $F(x)=\sum_{i=0}^n a_i x^i$, and a non-negative integer $N$.}
	\Output{Polynomial $F^{[N]}(x)=\sum_{i=0}^n a^{[N]}_i x^i$. If $F$ has roots $z_1,\ldots,z_n$, then $F^{[N]}$ has roots $z_1^{2^N},\ldots,z_n^{2^N}$, and $a_n^{[N]}=a_n^{2^N}$}
	\BlankLine
	$F^{[0]}(x):=F(x)$\;
	\For{$i=1,\ldots , N$}{
		$F^{[i]}(x):=(-1)^n [F^{[i-1]}_e(x)^2 - x\cdot F^{[i-1]}_o(x)^2]$
	}
	\Return{$F^{[N]}(x)$}
	\caption{Graeffe Iteration}
\end{algorithm}

We can now iteratively apply Graeffe iterations in order to square the roots of a polynomial $F(x)$ several times.
In this way, we can now reduce the ``separation factor of the $T_k$-Test'' from polynomial in $n$ (namely, $256n^5$) to a constant value (in our case, this constant will be $\approx 1.41$) when we run $N$, with $N=\Theta(\log \log n)$, Graeffe iterations first, and then apply the $T_k$-test; see Algorithm~\ref{algo:graeffe2}.
From Theorem~\ref{tktest thm} and Theorem~\ref{graeffe thm}, we then obtain the following result:

\begin{lemma}\label{softtest:success}
	Let $\Delta$
	be a disk in the complex plane and $F(x)\in\mathbb{C}[x]$ a polynomial of degree $n$. Let
	\begin{align}
		N:=\lceil\log(1+\log n)\rceil+5
	\end{align}
	and
	\begin{align}
		\rho_1:=\frac{2\sqrt{2}}{3}\approx 0.943\quad\text{and}\quad\rho_2:=\frac{4}{3}
	\end{align}
	Then, we have $\sqrt[2^N]{\frac{1}{16n}}>\rho_1$, and it holds:
	\begin{itemize}
		\item[(a)] If $\Delta$ is $(\rho_1,\rho_2)$-isolating for a set of $k$ roots of $F$, then $T_k^G (\Delta,\frac{3}{2})$ succeeds.
		\item[(b)] If $T_k^G  (\Delta,K)$ succeeds for some $K\ge 1$, then $\Delta$ contains exactly $k$ roots.
	\end{itemize}
\end{lemma}

\begin{proof}
	The lower bound on $\rho(n):=\sqrt[2^N]{\frac{1}{16n}}$ follows by a straight forward computation that shows that $\rho(n)$, considered as a function in $n$, is strictly increasing and that $\rho(2)\approx 0.947> \frac{2\sqrt{2}}{3}\approx 0.943$. Now, let
	$F_{\Delta}^{[N]}$ be the polynomial obtained from $F_{\Delta}$ after performing $N$ recursive Graeffe iterations.
	If $\Delta$ is $(\rho_1,\rho_2)$-isolating for a set of $k$ roots of $F$, then the unit disk $\Delta':=\Delta(0,1)$ is also $(\rho_1,\rho_2)$-isolating for a set of $k$ roots of $F_{\Delta}$, that is, $\Delta'$ contains $k$
	roots of $F_{\Delta}$ and all other roots of $F_{\Delta}$ have absolute value
	larger than $\frac{4}{3}$. Hence, we conclude that $F_{\Delta}^{[N]}$ has $k$ roots of absolute value less than $\rho_1^{2^N}<\frac{1}{16n}$, whereas the remaining roots have absolute value larger than
	$\rho_2^{2^N}\ge 16n^4$. From Corollary~\ref{tktest cor2}, we thus
	conclude that $T_k(\Delta',\frac{3}{2},F_{\Delta}^{[N]})$ succeeds.
	This shows (a).
	Part (b) is an immediate consequence of Theorem~\ref{thm:pellet} and
	the fact that Graeffe iteration does not change the number of roots contained in the unit disk.
\end{proof}

In the special case where $k=0$, the failure of $T_0^G (\Delta)$ already implies that $\frac{4}{3}\cdot\Delta$ contains at least one root.\\

\begin{algorithm}[t]\label{algo:graeffe2}
	\DontPrintSemicolon
	\SetKwData{Left}{left}\SetKwData{This}{this}\SetKwData{Up}{up}
	\SetKwFunction{Union}{Union}\SetKwFunction{FindCompress}{FindCompress}
	\SetKwInOut{Input}{Input}\SetKwInOut{Output}{Output}
	\Input{Polynomial $F(x)$ of degree $n$, disk $\Delta=\Delta(m,r)$, real value $K$ with $1\le K\le \frac{3}{2}$}
	\Output{True or False. If the algorithm returns True, $\Delta$ contains exactly $k$ roots of $F$.}
	\BlankLine
	Call Algorithm~\ref{algo:graeffe} with input
	$F_{\Delta}(x):=F(m+r\cdot x)$
	and $N:=\lceil\log(1 +\log n)\rceil+5$, which returns $F_{\Delta}^{[N]}$\;
	\Return{$T_k(0,1,K,F^{[N]}(x))$}
	\caption{$T_k^G(\Delta,K)$-Test}
\end{algorithm}

The following result is a direct consequence of Theorem~\ref{graeffe thm}. We will later use it in the analysis of our algorithm:

\begin{corollary}
	Let $F_{\Delta}$ and $F^{[N]}_{\Delta}$ be defined as in Algorithm~\ref{algo:graeffe2}. Then, it holds:
	$$\LOG(\|F^{[N]}_{\Delta}(x)\|_\infty,\|F^{[N]}_{\Delta}(x)\|_\infty^{-1})=O(\log n\cdot(n+\LOG(\|F_{\Delta}\|_\infty,\|F_{\Delta}\|_\infty^{-1})).$$
\end{corollary}


\subsection{The \texorpdfstring{$\tilde T_k^G$}{tTkG}-Test: Using Approximate Arithmetic}\label{sec:graeffeapx}

So far, the $T_k$-test is formulated in a way such that, in general, high-precision
arithmetic, or even exact arithmetic, is needed in order to compute its output. Namely, if the two expressions on both sides of (\ref{Tktest}) are actually equal, then exact arithmetic is needed to decide
equality. Notice that, in general, we cannot even handle this case as we have only access to (arbitrary good) approximations of the coefficients of the input polynomial $F$.
But even if the two expression are different but almost equal, then we need to evaluate the
polynomial $F$ and its higher order derivatives with a very high precision in order to decide the
inequality, which induces high computational costs. This is a typical problem that appears in many algorithms, where a sign predicate $\mathcal{P}$ is
used to draw conclusions, which in turn decide a branch of the algorithm. Suppose that, similar as for the $T_k$-test (with $E_{\ell}=\frac{|F^{(k)}(m)|\cdot r^k }{k!} $ and $E_r=\sum_{i\neq k} \frac{|F^{(i)}(m) |\cdot r^i }{i!}$), there exist
two non-negative expressions $E_{\ell}$ and $E_r$ such that $\cal P$ succeeds\footnote{We assume that the predicate $\cal P$ either returns ``True'' or ``False''. We say that $\cal P$ succeeds if it returns True. Otherwise, we say that it fails.} if and only
if $E_{\ell}-E_r$ has a positive sign (or, equivalently, if $E_{\ell}>E_r$). We further denote by
$\mathcal{P}_{\frac{3}{2}}$ the predicate that succeeds if and only if the stronger inequality
$E_{\ell}-\frac{3}{2}\cdot E_r>0$ holds.\footnote{You may replace $\frac{3}{2}$ by an arbitrary real constant $K$ larger than $1$.}
Then, success of $\mathcal{P}_{\frac{3}{2}}$ implies success of $\cal P$; however, a failure of
$\mathcal{P}_{\frac{3}{2}}$ does, in general, not imply that $\mathcal{P}$ fails as well.
As already mentioned above for the special case, where $\mathcal{P}=T_k(m,r,1,F)$, it might be computationally expensive (or even infeasible) to determine the outcome of
$\mathcal{P}$, namely in the case where the two expressions $E_{\ell}$ and $E_r$ are equal or almost equal. In
order to avoid such undesirable situations, we propose to replace the predicate $\cal P$ by a
corresponding so-called \emph{soft-predicate}~\cite{DBLP:conf/cie/YapS013}, which we denote by
$\tilde{\cal P}$. $\tilde{\mathcal{P}}$ does not only return True or False, but may also return a flag called ``Undecided''. If it returns True or False, the result of $\tilde{\mathcal{P}}$ coincides with that of $\mathcal{P}$. However, if $\tilde{\mathcal{P}}$ returns Undecided, we may only conclude that $E_{\ell}$ is a relative $\frac{3}{2}$-approximation of $E_r$ (i.e.,
$\frac{2}{3}\cdot E_{\ell}<E_r<\frac{3}{2}\cdot E_{\ell}$).
We briefly sketch our approach and give details in Algorithm~\ref{algo:softpredicate}: In the first step, we compute approximations $\tilde E_{\ell}$ and $\tilde E_r$ of the values $E_{\ell}$ and $E_r$, respectively. Then, we check
whether we can already compare the exact values $E_{\ell}$ and $E_r$ by just considering their approximations and taking into account the quality of approximation. If this is the
case, we are done as we can already determine the outcome of $\mathcal{P}$. Hence, we define that $\tilde{\mathcal{P}}$ returns True (False) if we can show
that $E_{\ell}>E_r$ ($E_{\ell}<E_r$). Otherwise, we iteratively increase the quality of approximation until we can either
show that $E_{\ell}>E_r$, $E_{\ell}<E_r$, or $\frac{2}{3}\cdot E_{\ell}\le E_r\le\frac{3}{2}\cdot E_{\ell}$. We may consider the latter case as an indicator that comparing $E_{\ell}$ and $E_r$ is difficult, and thus $\tilde{\cal P}$ returns Undecided in this case.

It is easy to see that Algorithm~\ref{algo:softpredicate} terminates if and only if at least one of the two expressions $E_{\ell}$ and $E_r$ is non-zero, hence we make this a requirement.
In the following lemma, we further give a bound on the precision to which the expressions $E_{\ell}$ and $E_r$ have to be approximated in order to guarantee termination of the algorithm.

\begin{algorithm}[t]\label{algo:softpredicate}
	\DontPrintSemicolon
	\SetKwData{Left}{left}\SetKwData{This}{this}\SetKwData{Up}{up}
	\SetKwFunction{Union}{Union}\SetKwFunction{FindCompress}{FindCompress}
	\SetKwInOut{Input}{Input}\SetKwInOut{Output}{Output}
	\Input{A predicate $\mathcal{P}$ defined by non-negative expressions $E_{\ell}$ and $E_r$, with $E_{\ell}\neq 0$ or $E_r\neq 0$; i.e. $\mathcal{P}$ succeeds if and only if $E_{\ell}>E_r$.}
	\Output{True, False, or Undecided. In case of True (False), $\mathcal{P}$ succeeds (fails). In case of Undecided, we have $\frac{2}{3}\cdot E_{\ell}<E_r\le\frac{3}{2}\cdot E_{\ell}$.
	}
	\BlankLine
	$L:=1$\;
	\While{True}{
		Compute $L$-bit approximations $\tilde E_{\ell}$ and $\tilde E_r$ of the expressions $E_{\ell}$ and $E_r$, respectively.\;
		$E_{\ell}^{\pm}:=\max(0,\tilde E_{\ell}\pm 2^{-L})$ and $E_{r}^{\pm}:=\max(0,\tilde E_{r}\pm 2^{-L})$\;
		\If{$E_{\ell}^->E_r^+$}{\Return True}
		\atcp{\normalfont{\textit{It follows that $E_{\ell}>E_r$.}}}
		\If{$E_{\ell}^+<E_r^-$}{\Return False}
		\atcp{\normalfont{\textit{It follows that $E_{\ell}<E_r$.}}}
		\If{$\frac{2}{3}\cdot E_\ell^+\le E_r^-< E_r^+\le\frac{3}{2}\cdot E_{\ell}^-$,}{\Return Undecided}
		\atcp{\normalfont{\textit{It follows that $\frac{2}{3}\cdot E_\ell\le E_r\le\frac{3}{2}\cdot E_{\ell}$.}}}
		$L:=2\cdot L$\;
	}
	\caption{Soft-predicate $\tilde{\mathcal{P}}$}
\end{algorithm}

\begin{lemma}\label{precision:compare}
	Algorithm~\ref{algo:softpredicate} terminates for an $L$ that is upper bounded by $$L_0:=2\cdot(\LOG(\max(E_{\ell},E_r)^{-1})+4).$$
\end{lemma}

\begin{proof}
	Suppose that $L\ge \LOG(\max(E_{\ell},E_r)^{-1})+4$. We further assume that $E_{\ell}=\max(E_{\ell},E_r)$; the case $E_r=\max(E_{\ell},E_r)$ is then treated in analogous manner.
	It follows that $$E_{r}^+\le E_{r} +2^{-L+1}\le E_{\ell} +2^{-L+1}\le \frac{9}{8}\cdot E_{\ell}\le \frac{3}{2}\cdot E_{\ell}-2^{-L+2}\le \frac{3}{2}\cdot E_{\ell}^-.$$
	Hence, if, in addition, $\frac{2}{3}\cdot E_{\ell}^+\le E_r^-$, then the algorithm returns Undecided in Step 10. Otherwise, we have
	$\frac{9}{8}\cdot E_{\ell}\ge E_{\ell}+2^{-L+1}\ge E_{\ell}^+>\frac{3}{2}\cdot E_{r}^-$,
	and thus
	$$E_{\ell}^-\ge E_{\ell}-2^{-L+1}\ge \frac{7}{8}\cdot E_{\ell}\ge \frac{3}{4}\cdot E_{\ell}+2^{-L+1}\ge E_{r}^-+2^{-L+1}\ge E_r^+,$$
	which shows that the algorithm returns True in Step 6.
	Since we double $L$ in each iteration, it follows that the algorithm must terminate for an $L$ with $L< 2\cdot(\LOG(\max(E_{\ell},E_r)^{-1})+4)$.
\end{proof}

Notice that if $\tilde{\cal P}$ returns True, then $\cal P$ also succeeds. This however does not hold in the opposite direction. In addition,
if $\mathcal{P}_{\frac{3}{2}}$ succeeds, then $E_{\ell}>E_r$ and $E_{\ell}$ cannot be a relative $\frac{3}{2}$-approximation of $E_r$, hence $\tilde{\cal P}$ must return True. We conclude that our soft-predicate is somehow located ``in between'' the two predicates $\cal P$ and $\mathcal{P}_{\frac{3}{2}}$.\\

We now return to the special case, where $\mathcal{P}=T_k(m,r,1,F)$, with
$E_{\ell}=\frac{|F^{(k)}(m)|\cdot r^k }{k!} $ and $E_r=\sum_{i\neq k} \frac{|F^{(i)}(m) |\cdot r^i }{i!}$ the two expressions on the left and the right side of (\ref{Tktest}),
respectively. Then, success of $\mathcal{P}$ implies that the disk $\Delta=\Delta(m,r)$
contains exactly $k$ roots of $F$, whereas a failure of $\mathcal{P}$ yields no further
information. Now, let us consider the corresponding soft predicate $\tilde{\mathcal{P}}=\tilde T_k(\Delta,F)$ of $\mathcal{P}=T_k(\Delta,F)$. If $\tilde{\mathcal{P}}$
returns True, then this implies success of $\mathcal{P}$. In addition, notice that
success of $T_k(\Delta,\frac{3}{2},F)$ implies that $\tilde{\mathcal{P}}$ returns True,
and thus we may replace $T_k(\Delta,\frac{3}{2},F)$ by $\tilde T_k(\Delta,F)$ in the
second part of Theorem~\ref{tktest thm}. Similarly, in Lemma~\ref{softtest:success}, we may also replace $T_k^G(\Delta,\frac{3}{2},F)$ by the soft-version $\tilde T_k^G(\Delta,F)$ of $T_k^G(\Delta,F)$. We give more
details for the computation of $\tilde T_k(\Delta,F)$ and $\tilde T_k^G(\Delta,F)$
in Algorithms~\ref{algo:softTk} and~\ref{algo:softgraeffe2}, which are essentially applications of Algorithm~\ref{algo:softpredicate} to the predicates $T_k(\Delta,F)$ and $T_k^G(\Delta,F)$. The lemma below summarizes our results. Based on Lemma~\ref{precision:compare}, we also provide a bound on the precision $L$ for which Algorithm~\ref{algo:softTk} terminates and a bound for the bit complexity of Algorithm~\ref{algo:softTk}. A corresponding bound for the bit complexity of carrying out the $\tilde T_k^G(\Delta,F)$-test for all $k=0,\ldots,n$ is given in Lemma~\ref{tildeTcost}.

\begin{algorithm}[t]\label{algo:softTk}
	\DontPrintSemicolon
	\SetKwData{Left}{left}\SetKwData{This}{this}\SetKwData{Up}{up}
	\SetKwFunction{Union}{Union}\SetKwFunction{FindCompress}{FindCompress}
	\SetKwInOut{Input}{Input}\SetKwInOut{Output}{Output}
	\Input{A polynomial $F(x)$ of degree $n$, a disk $\Delta:=\Delta(m,r)$ in the complex plane, and an integer $k$ with $0\le k\le n$.}
	\Output{True, False, Undecided. If the algorithm returns True, the disk $\Delta(m,r)$ contains exactly $k$ roots of $F$.}
	\BlankLine
	$L:=1$\;
	\While{True}{
		Compute an approximation $\tilde F_{\Delta}(x)=\sum_{i=0}^n \tilde{f}_i x^i$ of the polynomial $F_{\Delta}(x):=\sum_{i=0}^n f_i\cdot x^i:= F(m+r\cdot x)$ such that $\tilde{f}_i\cdot 2^{L+\lceil\log (n+1)\rceil}\in\mathbb{Z}$ and $|f_i-\tilde{f}_i|<2^{-L+\lceil\log (n+1)\rceil}$ for all $i$.

		\atcp{\normalfont{\textit{$(L+\lceil\log (n+1)\rceil)$-bit approximation of $F_{\Delta}$.}}}
		$f_i^-:=\max(0,|\tilde{f}_i|-2^{-L-\lceil\log (n+1)\rceil})$ for $i=0,\ldots,n$.\;
		$f_i^+:=|\tilde{f}_i|+2^{-L-\lceil\log (n+1)\rceil}$ for $i=0,\ldots,n$.\;
		\atcp{\normalfont{\textit{lower and upper bounds for $|f_i|$.}}}
		\If{$f_k^--\sum_{i\neq k} f_i^+>0$}{\Return True}
		\atcp{\normalfont{\textit{It follows that $T_k(\Delta,F)$ succeeds.}}}
		\If{$ \sum_{i\neq k} f_i^- - f_k^+ > 0 $}{\Return False}
		\atcp{\normalfont{\textit{It follows that $T_k(\Delta,F)$ fails.}}}
		\If{$\sum_{i\neq k} f_i^--\frac{2}{3}\cdot f_k^+\ge 0$\text{ and }
			$\frac{3}{2}\cdot f_k^--\sum_{i\neq k} f_i^+\ge 0$}{\Return False}
		$L:=2\cdot L$\;
	}
	\caption{$\tilde T_k(\Delta,F)$-test}
\end{algorithm}

\begin{lemma}
	\label{tildeTcost}
	For a disk $\Delta:=\Delta(m,r)$ in the complex plane and a polynomial $F\in\mathbb{C}[x]$ of degree $n$, the $\tilde T_k(\Delta,F)$-test terminates with an absolute precision $L$ that is upper bounded by
	\begin{align}\label{def:epsilon}
		L(\Delta,F):=L(m,r,F):=2\cdot\left(4+\LOG(\|F_{\Delta}\|_\infty^{-1})\right).
	\end{align}
	If $T_k(\Delta,\frac{3}{2},F)$ succeeds, the $\tilde T_k(\Delta,F)$-test returns True.
	The cost for running the $\tilde{T}_k(\Delta,F)$-test \emph{for all} $k=0,\ldots,n$ is upper bounded by
	\[
		\tilde{O}(n(n\cdot \LOG(m, r)+\tau_F+L(\Delta,F)))
	\]
	bit operations. The algorithm needs an $\tilde{O}(n\cdot \LOG(m, r)+\tau_F+L(\Delta,F))$-bit approximation of $F$.
\end{lemma}

\begin{proof}
	Let $\mathcal{P}:=T_k(\Delta,1,F)$ be the predicate that succeeds if and only if $E_{\ell}>E_r$, with $E_{\ell}:=|f_k|$ and $E_{r}:=\sum_{i\neq k}|f_i|$. Then, $E_{\ell}^{\pm}:=f_k^{\pm}$ and $E_r^{\pm}:=\sum_{i\neq k} f_i^{\pm}$ are lower and upper bounds for $E_{\ell}$ and $E_r$, respectively, such that $|E_{\ell}^{\pm}-E_{\ell}|\le 2^{-L+1}$ and $|E_{r}^{\pm}-E_{r}|\le 2^{-L+1}$. Hence, Lemma~\ref{precision:compare} yields that Algorithm~\ref{algo:softTk} terminates for an $L$ smaller than $2\cdot(4+\LOG(\max(E_{\ell},E_r)^{-1}))\le L(\Delta,F)$.

	We have already argued above that success of the predicate $\mathcal{P}_{\frac{3}{2}}=T_k(\Delta,\frac{3}{2},F)$ implies that $\tilde{P}=\tilde{T}_k(\Delta,F)$ returns True. Hence, it remains to show the claim on the bit complexity for carrying out the $\tilde T_k(\Delta,F)$-test for all $k=0,\ldots,n$.
	For a given $L$, we can compute an $(L+\lceil \log (n+1)\rceil)$-bit approximation $\tilde F_{\Delta}(x)=\sum_{i=0}^n\tilde{f}_ix^i$ of $F_{\Delta}$ with a number of bit operations that is bounded by $\tilde{O}(n(\tau_F+n\LOG(m, r)+L))$; e.g. see the first part of the proof of~\cite[Lemma~17]{Sagraloff2015}. For a fixed $k$, the computation of the signs of the sums in each of the three IF clauses needs $n$ additions
	of dyadic numbers with denominators of bit size $\lceil \log(n+1)\rceil+L$ and with numerators of bit size $O(L+n\LOG(r)+\tau_F)$, hence the cost is bounded by $O(n(\tau_F+n\LOG(r)+L))$ bit operations. Notice that, when passing from an integer $k$ to a $k'\neq k$, the corresponding sums in one IF clause differ only by two terms, that is, $f^{\pm}_k$ and $f^{\pm}_{k'}$. Hence, we can decide all IF clauses \emph{for all $k$} using $O(n)$ additions. Furthermore, we double the precision $L$ in each step, and the algorithm terminates for an $L$ smaller than $L(\Delta,F)$. Hence, $L$ is doubled at most $\log L(\Delta,F)$ many times, and thus the total cost for all $k$ is bounded by $\tilde{O}(n(\tau_F+n\LOG(m, r)+L(\Delta,F)))$ bit operations.
\end{proof}

We now extend the above soft-variant of the $T_k$-test to a corresponding soft-variant of the $T_k^G$-test, which we denote $\tilde T_k^G$; see Algorithm~\ref{algo:softgraeffe2} for details. We further combine $\tilde{T}_k^G$ for all $k=0,\ldots,n$ to obtain $\T(\Delta,F)$ with
\begin{align}\label{def:T*}
	\T(\Delta,F):=\begin{cases}
	k\quad\text{if there exists a }k\text{ such that }\tilde{T}_k^G(\Delta)\text{ succeeds} \\
	-1\text{ otherwise.}
	\end{cases}
\end{align}

Again, for brevity, we often omit $F$ and just write $\T(\Delta)$.
We say that $\T$ succeeds if it returns a non-negative value. Otherwise, it fails.

The following result, which can be considered as the ``soft variant'' of Lemma~\ref{softtest:success}, can then immediately be deduced from Lemma~\ref{softtest:success} and Lemma~\ref{tildeTcost}:

\begin{algorithm}[t]\label{algo:softgraeffe2}
	\DontPrintSemicolon
	\SetKwData{Left}{left}\SetKwData{This}{this}\SetKwData{Up}{up}
	\SetKwFunction{Union}{Union}\SetKwFunction{FindCompress}{FindCompress}
	\SetKwInOut{Input}{Input}\SetKwInOut{Output}{Output}
	\Input{Polynomial $F(x)\in\mathbb{C}[x]$ of degree $n$, a disk $\Delta:=\Delta(m,r)$ in the complex space.}
	\Output{True, False, or Undecided. If the algorithm returns True, $\Delta$ contains exactly $k$ roots of $F$.}
	\BlankLine
	Let $F^{[N]}_{\Delta}(x)$ be the $N$-th Graeffe iterate of $F_{\Delta}(x):=F(m+r\cdot x)$, where $N:=\lceil\log(1 +\log n)\rceil+5$\;
	Output $\tilde T_k(0,1,F^{[N]}_{\Delta})$.
	\caption{$\tilde T_k^G(\Delta,F)$-Test}
\end{algorithm}

\begin{lemma}[Soft-version of Lemma~\ref{softtest:success}]\label{softtest:success2}
	Let $\Delta:=\Delta(m,r)$ be a disk in the complex plane, $F(x)\in\mathbb{C}[x]$ be a polynomial of degree $n$, and let $\rho_1=\frac{2\sqrt{2}}{3}$ and $\rho_2=\frac{4}{3}$. Then, it holds:
	\begin{enumerate}[(a)]
		\item\label{softtest:firstpart} If $\Delta$ is $(\rho_1,\rho_2)$-isolating for a set of $k$ roots of $F$, then $\T(\Delta)$ returns $k$.
		\item\label{softtest:secondpart} If $\T(\Delta)$ returns a $k\ge 0$, then $\Delta$ contains exactly $k$ roots.
	\end{enumerate}
\end{lemma}

For the complexity analysis of our root isolation algorithm (see Section~\ref{sec:algorithm}), we provide a bound on the total cost for running the $\T$-test.

\begin{lemma}
	\label{allkcost}
	The total cost for carrying out the $\T(\Delta)$ is bounded by
	\begin{align*}
		  & \tilde{O}(n(\tau_F+n\LOG(m,r)+L(\Delta,F)))
		=\tilde{O}(n(\tau_F+n\LOG(m,r)+\LOG((\max_{z\in \Delta}|F(z)|)^{-1})))
	\end{align*}
	bit operations. For this, we need an $L$-bit approximation of $F$ with
	\[
		L=\tilde{O}(\tau_F+n\LOG(m,r)+L(\Delta,F))=\tilde{O}(\tau_F+n\LOG(m,r)+\LOG((\max_{z\in \Delta}|F(z)|)^{-1})).
	\]
\end{lemma}

\begin{proof}
	According to Lemma~\ref{tildeTcost}, the computation of $\tilde{T}_k(0,1,F_{\Delta}^{[N]})$ needs an $L$-bit approximation $\tilde{F}_{\Delta}^{[N]}$ of $F_{\Delta}^{[N]}$ , with $L$ bounded by
	\begin{align}\label{boundforL}
		\tilde{O}(n+\tau_{F_\Delta^{[N]}}+L(0,1,F_{\Delta}^{[N]}))=\tilde{O}(n+\LOG(\|F_{\Delta}^{[N]}\|_\infty,\|F_{\Delta}^{[N]}\|_\infty^{-1})).
	\end{align}
	Given such an approximation $\tilde{F}_{\Delta}^{[N]}$, the cost for running the test for all $k=0,\ldots,n$ is then bounded by $\tilde{O}(n(n+\tau_{F_\Delta^{[N]}}+L))$ bit operations. In each of the $N=O(\log \log n )$ Graeffe iterations, the size of $\LOG(\|F_{\Delta}^{[i]}\|_\infty,\|F_{\Delta}^{[i]}\|_\infty^{-1})$ increases by at most a factor of two plus an additive term $4n$; see Theorem~\ref{graeffe thm}. Hence, we must have
	\begin{align*}
		\LOG(\|F_{\Delta}^{[i]}\|_\infty,\|F_{\Delta}^{[i]}\|_\infty^{-1})
		  & = O(\log n\cdot \LOG(\|F_{\Delta}\|_\infty,\|F_{\Delta}\|_\infty^{-1})+n\log n) \\
		  & =\tilde{O}(n\LOG(m,r)+\tau_F+L(\Delta,F))
	\end{align*}
	for all $i=0,\ldots,N$.
	We conclude that the above bound (\ref{boundforL}) for $L$ can be replaced by $\tilde{O}(\tau_F+n\LOG(m,r)+L(\Delta,F))$.

	It remains to bound the cost for computing an approximation $\tilde{F}_{\Delta}^{[N]}$ of $F_{\Delta}^{[N]}$ with $\|F_{\Delta}^{[N]}-\tilde{F}_{\Delta}^{N}\|_{\infty}<2^{-L}$.
	Suppose that, for a given $\rho\in\mathbb{N}$ we have computed an approximation $\tilde{F}_{\Delta}$ of $F_{\Delta}$, with $\|F_{\Delta}-\tilde{F}_{\Delta}\|_{\infty}<2^{-\rho}$. According to~\cite[Theorem~8.4]{Schoenhage82} (see also \cite[Theorem~14]{DBLP:journals/corr/abs-1304.8069} and \cite[Lemma~17]{Sagraloff2015}), this can be achieved using a number of bit operations bounded by $\tilde{O}(n(n\LOG(m,r)+\tau_F+\rho))$.
	In each Graeffe iteration, an approximation $\tilde{F}_{\Delta}^{[i]}$ of $F_{\Delta}^{[i]}$ is split into two polynomials $\tilde{F}_{\Delta,o}^{[i]}$ and $\tilde{F}_{\Delta,e}^{[i]}$ with coefficients of comparable bit size (and half the degree), and an approximation $\tilde{F}_{\Delta}^{[i+1]}$ of $F_{\Delta}^{[i]}$ is then computed as the difference of $\tilde{F}_{\Delta,e}^{[i]}$ and $x\cdot \tilde{F}_{\Delta,o}^{[i]}$.
	If all computations are carried out with fixed point arithmetic and an absolute precision of $\rho$ bits after the binary point, then the precision loss in the $i$-th step, with $i=0,\ldots,N$, is bounded by $O(\log n+\log\|F_{\Delta}^{[i]}\|_\infty)=O(2^i(\log n+\log\|F_{\Delta}\|_\infty))=O(\log n(\log n+\log\|F_{\Delta}\|_\infty))$ bits after the binary point. The cost for the two multiplications and the addition is bounded by $\tilde{O}(n(\rho+\log\|F_{\Delta}^{[i]}\|_\infty))$.
	Since there are only $N=O(\log\log n)$ many iterations, we conclude that it suffices to start with an approximation $\tilde{F}_{\Delta}$ of $F_{\Delta}$, with $\|F_{\Delta}-\tilde{F}_{\Delta}\|_{\infty}<2^{-\rho}$ and $\rho=\tilde{O}(n\LOG (m,r)+\tau_F+L(\Delta,F))$.
	The total cost for all Graeffe iterations is then bounded by $\tilde{O}(n\rho)$ bit operations, hence the claim follows together with the fact that $\max_{z\in\Delta}|F(z)|\le (n+1)\|F_{\Delta}\|_{\infty}$ as shown in~\eqref{formula:FinftyDelta} in the proof of Theorem~\ref{graeffe thm}.
\end{proof}


\section{\texorpdfstring{$\CC$\textsc{Isolate}}{CIsolate}: An Algorithm for Root Isolation}\label{sec:algorithm}

We can now formulate our algorithm, which we denote by $\mathbb{C}$\textsc{Isolate}, to isolate all complex roots of a polynomial $F(x)$ that are contained in some given square\footnote{As already mentioned in Section~\ref{sec:definitions}, we only consider closed, axis-aligned squares $B\subset\mathbb{C}$. Hence, these properties are not further mentioned throughout the following considerations.} $\mathcal{B}\subset \mathbb{C}$.
If the enlarged square $2\mathcal{B}$ contains only simple roots of $F$, then our algorithm returns
isolating disks for all roots that are contained in $\mathcal{B}$. However, it might also
return isolating disks for some of the roots that are not contained in $\mathcal{B}$ but in
the complement $2\mathcal{B}\setminus \mathcal{B}$. In particular, in the important special case,
where $F$ is square-free and where we start with a square $\mathcal{B}$ that is known to contain
all complex roots of $F$, our algorithm isolates all complex roots of $F$.
Before we give details, we need some further definitions, which we provide in Section~\ref{sec:defcomponent}. In Section~\ref{subsec:algorithm}, we first give an overview of our algorithm before we provide details and the proof for termination and correctness.


\subsection{Connected Components}\label{sec:defcomponent}

Given a set $S=\{B_1,\ldots,B_m\}$
of squares $B_1,\ldots,B_m\subset \mathbb{C}$, we say that two squares $B,B'\in S$ are connected in $S$ ($B\sim_{S} B'$ for short) if there
exist squares $B_{i_1},\ldots,B_{i_{s'}}\in S$ with $B_{i_1}=B$, $B_{i_{s'}}=B'$, and $B_{i_j}\cap B_{i_{j+1}}\neq \emptyset$ for all $j=1,\ldots,s'-1$.
This yields a decomposition of $S$ into equivalence classes $C_1,\ldots,C_k\subset S$ that correspond to maximal connected and disjoint components $\bar{C}_{\ell}=\bigcup_{i:B_i\in C_{\ell}} B_i$, with $\ell=1,\ldots,k$. Notice that $C_{\ell}$ is defined as the set of squares $B_i$ that belong to the same equivalence class, whereas $\bar{C}_{\ell}$ denotes the closed region in $\mathbb{C}$ that consists of all points that are contained in a square $B_i\in C_{\ell}$. However, for simplicity, we abuse notation and simply use $C$ to denote the set of squares $B$ contained in a component $C$ as well as to denote the set of points contained in the closed region $\bar{C}$.
Now, let $C=\{B_1,\ldots,B_s\}$ be a connected component consisting of equally sized squares $B_i$ of width $w$, then we define (see also Figure~\ref{fig:yyy}):
\begin{itemize}
	\item $B_C$ is the axis-aligned closed square in $\mathbb{C}$ of minimal width such that $C\subset B_C$ and
	      \[
	      	\min_{z\in B_C}\Re(z)=\min_{z\in C} \Re(z)
	      	\text{ and }
	      	\max_{z\in B_C}\Im(z)=\max_{z\in C} \Im(z),
	      \]
	      where $\Re(z)$ denotes the real part and $\Im(z)$ the imaginary part of an arbitrary complex value $z$.
	      We further denote $m_C$ the center of $B_C$, and $\Delta_C:=\Delta(m_C,\frac{3}{4} w(B_C))$ a disk containing $B_C$, and thus also $C$. 

	      We further define the \emph{diameter $w(C)$ of the component $C$} to be the width of $B_C$, i.e. $w(C):=w(B_C)$, and $r(C):=\frac{w(C)}{2}$ to be the \emph{radius of~$C$}. 
	\item $C^+:=\bigcup_{i:B_i\in C}2B_i$ is defined as the union of the enlarged squares $2 B_i$. Notice that $C^+$ is the $\frac{w}{2}$-neighborhood of $C$ (w.r.t. max-norm).
\end{itemize}


\subsection{The Algorithm}\label{subsec:algorithm}

We start with an informal description of our algorithm $\mathbb{C}\textsc{Isolate}$, where we focus on the main ideas explaining the ratio behind our choices. For the sake of
comprehensibility, we slightly simplified some steps at the cost of complete formal correctness, hence,
the considerations below should be taken with a grain of salt. A precise definition of
the algorithm including all details is given in Algorithm~\ref{cisolate} and the
subroutines \textsc{NewtonTest} (Algorithm~\ref{newtontest}) and \textsc{Bisection}
(Algorithm~\ref{bisection}).

\begin{algorithm}[t!]
	\DontPrintSemicolon
	\SetKwInOut{Input}{Input}\SetKwInOut{Output}{Output}
	\Input{A polynomial $F(x)\in\mathbb{C}[x]$ as in (\ref{def:polyF}) and a square $\mathcal{B}\subset \mathbb{C}$ of width $w_0:=w(\mathcal{B})=2^{\ell_0}$, with $\ell_0\in\mathbb{Z}$; $F$ has only simple roots in $2\mathcal{B}$.}
	\Output{A list $\mathcal{O}$ of disjoint disks $\Delta_1,\ldots,\Delta_s\subset \mathbb{C}$ such that, for each $i=1,\ldots,s$, the disk $\Delta_i$ as well as the enlarged disk $2\Delta_i$ is isolating for a root of $F$ that is contained in $2\mathcal{B}$. In addition, for each root $z\in\mathcal{B}$, there exists a disk $\Delta_i\in\mathcal{O}$ that isolates $z$.}
	\BlankLine

	$\mathcal{O} = \{\}$  \atcp{\normalfont{\textit{list of isolating disks}}}
	$\mathcal{C} = \{(\mathcal{B},4)\}$  \atcp{\normalfont \textit{list of pairs $(C,N_C)$, with $C$ a connected com-}}
	\atcp{\normalfont \textit{ponent consisting of $s_C$ equally sized squares,}}
	\atcp{\normalfont \textit{each of width $2^{\ell_C}$, where $\ell_C\in\mathbb{Z}_{\le \ell_0}$. $N_C$ is an}}
	\atcp{\normalfont \textit{integer with $N_C=2^{2^{n_C}}$ and $n_C\in\mathbb{N}_{\ge 1}$.}}
	\medskip
	\nonl\texttt{//} \textit{* Preprocessing *}\texttt{//}\;
	\Repeat{$\bigcup_{C:(C,N_C)\in\mathcal{C}}C \neq \mathcal{B}$}{
		Let $(C,N_C)$ be the unique pair in $\mathcal{C}$\;
		\atcp{\normalfont \textit{If $\bigcup_{C:(C,N_C)\in\mathcal{C}}C = \mathcal{B}$, then there exists a}}
		\atcp{\normalfont \textit{unique component $C$ with $(C,N_C)\in\mathcal{C}$.}}\nonl
		\nonl\texttt{//} \textit{* linear step *}\texttt{//}\;
		$\{C'_1,\ldots, C'_\ell\} \as \textsc{Bisection}(C)$ and $\mathcal{C} = \{(C'_1,4),\ldots, (C'_\ell,4)\}$
	}
	\medskip
	\nonl\texttt{//} \textit{* Main Loop *}\texttt{//}\;
	\While{$\mathcal{C}$ is non-empty}{
		Remove a pair $(C,N_C)$ from $\mathcal{C}$.\;
		\If{$4\Delta_C \cap C' = \emptyset$ for each $(C',N_{C'})\in\mathcal{C}$ with $C'\neq C$ {\bf and} there exists a $k_C\in\{1,\ldots,n\}$ such that $k_C=\T(2\Delta_C)=\T(4\Delta_C)$\medskip \label{testbeforenewton}}{
			\atcp{\normalfont \textit{If the second condition holds, $k_C$ equals the}}\nonl
			\atcp{\normalfont \textit{number of roots contained in $2\Delta_C$ and $4\Delta_C$.}}\nonl
			\If{$k_C=1$}{
				Add the disk $2\Delta_C$ to $\mathcal{O}$, \textbf{continue}\label{addisolating}\;
			}
			\If{$k_C>1$}{ \label{ReachesNewtonTest}
				Let $x_C\in\mathcal{B}\setminus C$ be an arbitrary point with distance $2^{\ell_C-1}$ from~$C$ and distance $2^{\ell_C-1}$ or more from the boundary of~$\mathcal{B}$.\label{ChoosePoint}\;
				\atcp{\normalfont \textit{Existence of such a point follows from the}}
				\atcp{\normalfont \textit{proof of Theorem~\ref{thm:termination}. It holds that $F(x_C)\neq 0$.}}
				\If{\textsc{NewtonTest}$(C, N_C, k_C, x_C)=(\textsc{Success},C')$}{
					\nonl\texttt{//} \textit{* quadratic step *}\texttt{//}\;
					Add $(C', N_C^2)$ to $\calC$, \textbf{continue} \label{newtonadd}\;
				}
			}
		}

		\nonl\texttt{//} \textit{* linear step *}\texttt{//}\;
		$\{C'_1,\ldots, C'_{\ell}\} \as$ \textsc{Bisection}$(C)$.\;
		Add $(C'_1, \max(4, \sqrt{N_C})), \ldots , (C'_\ell, \max(4, \sqrt{N_C}))$ to $\calC$.\label{bisectionadd} \;
	}
	\Return $\mathcal{O}$.
	\;
	\caption{$\mathbb{C}\textsc{Isolate}$}
	\label{cisolate}
\end{algorithm}

From a high-level perspective, our algorithm follows the classical subdivision approach of Weyl~\cite{Weyl}.
That is, starting from the input square $\mathcal{B}$, we recursively subdivide
$\mathcal{B}$ into smaller squares, and we remove squares for which we can show that they do not contain a root of $F$. Eventually, the algorithm returns regions that are isolating for a root of $F$. In order to discard a square $B$, with $B\subset\mathcal{B}$, we call the $\T(\Delta_B,F)$-test\footnote{In fact, it suffices to just call the $\tilde{T}_0^G(\Delta_B,F)$-test. However, since $\tilde{T}_0^G$ and $\T$ have comparable complexity, we just stick to $\T$ to simplify the presentation.}, with $\Delta_B$ the disk containing $B$.
The remaining squares are then clustered into maximal connected components. We further check whether a
component $C$ is well separated from all other components, that is, we test whether the distance from $C$ to all other components is considerably larger than its diameter. If this is the case,
we use the $\T$-test in order to determine the ``multiplicity''
$k_C$ of the component $C$, that is, the number of roots contained in the enclosing
disk $\Delta_C$; see Line~\ref{testbeforenewton} of Algorithm~\ref{cisolate} and Figure~\ref{fig:yyy} for details. If $k_C=1$, we may return an isolating disk for the corresponding unique root. Otherwise, there is a cluster consisting of two or more roots, which still have to be separated from each other.
A straight-forward approach to separate these roots from each other is to recursively subdivide each square into four equally sized squares and
to remove squares until, eventually, each of the remaining components contains exactly one
root that is well separated from all other roots; see also Algorithm~\ref{bisection} (\textsc{Bisection}) and Figure~\ref{fig:zzz}. However, this approach itself yields only
linear convergence to the roots, and, as a consequence, there might exist (e.g. for Mignotte polynomials) long sequences $C_1,\ldots,C_s$ of
interlaced connected components with
invariant multiplicity $k$, that is $C_1\supset C_2\supset\cdots\supset C_s$ and $k=k_{C_1}=\cdots =k_{C_s}>1$. The main idea to
traverse such sequences more efficiently is to consider a cluster of $k$ roots as a
single root of multiplicity $k$ and to use Newton iteration (for multiple roots) to compute a better approximation of this root.
For this, we use an adaptive trial and error approach similar to the quadratic interval
refinement (QIR) method, first introduced by Abbott~\cite{abbott-quadratic}; see Algorithm~\ref{newtontest} (\textsc{NewtonTest}) and Figure~\ref{fig:xxx}.
In its original form, QIR has been combined with the secant method to efficiently refine an interval
that is already known to be isolating for a real root of a real polynomial. Recent work~\cite{Sagraloff:2012:NMD:2442829.2442872} considers a modified approach of the QIR method that uses Newton iteration (for multiple roots) and Descartes' Rule of Signs. It has been refined and integrated in almost optimal methods~\cite{Sagraloff:2014:NAC:2608628.2608632,Sagraloff2015} for isolating and approximating the real roots of a real (sparse) polynomial, where it constitutes the crucial ingredient for quadratic convergence. In this paper, we further extend the QIR approach for approximating complex roots of a polynomial.
\begin{algorithm}[t!]
	\DontPrintSemicolon
	\SetKwInOut{Input}{Input}\SetKwInOut{Output}{Output}
	\Input{A tuple $(C,N_C,k_C,x_C)$: $C=\{B_1,\ldots,B_{s_C}\}$ is a connected component consisting of equally sized and aligned squares $B_i$ contained in $\mathcal{B}$ and of size $2^{\ell_C}$, $N_C$ is an integer of the form $2^{2^{n_C}}$ with $n_C\in\mathbb{N}_{\ge 1}$, $k_C$ is the number of roots in $4\Delta_C$, and $x_C$ is a point with $F(x_C)\neq 0$.}
	\Output{Either \textsc{Failure} or $(\textsc{Success}, C')$, where $C'\subset C$ is a connected component that contains all roots contained in $C$. $C'$ consists of at most $4$ equally sized and aligned squares, each of width $\frac{2^{\ell_C-1}}{N_C}$.}
	\BlankLine
	\If{Algorithm~\ref{algo:softpredicate} does not return False for the input $E_{\ell}:=4 r(C) |F'(x_C)|$ and $E_r:=|F(x_C)|$}{
		\nonl\vspace{-4mm}\hspace{40mm}\texttt{// }\textit{This implies $|F(x_C)|<6 r(C) |F'(x_C)|$.}\label{softcallinnewton}\;
		\For{$L=1,2,4,\ldots$ \label{resultless}}{
			Compute $L$-bit approximations of $F(x_C)$ and $F'(x_C)$ and derive an $(6-\ell_C+\log N_C)$-bit approximation
			$\wt{x}'_C$ of the Newton iterate
			\begin{align}
				\label{def:newton}
				x_C'\as x_C-k_C\cdot \frac{F(x_C)}{F'(x_C)} \quad\text{ such that }\;|\tilde{x}_C'-x'_C|<\frac{1}{64}\cdot\frac{2^{\ell_C}}{N_C}.\;
			\end{align}
			\atcp{\normalfont \textit{For more details, consider the similar}}
			\atcp{\normalfont \textit{computation in}~\cite[Step~2 of~\textsc{NewtonTest}]{Sagraloff2015}.}
		}
		Let $\Delta' \as \Delta(\wt{x}'_C, \frac{1}{8}\cdot\frac{2^{\ell_C}}{N_C})$.	\label{deltaprime}\;
		\If{$\Delta'\cap C=\emptyset$}{\Return  \textsc{Failure}}\;
		\If{$\T(\Delta')=k_C$ holds
			\atcp{\normalfont \textit{Then, $\Delta'$ contains all roots contained in $2\Delta_C$.}}
			}{
			Decompose each square $B_i$ into $4N_C^2$ many equally sized sub-squares $B_{i,j}$\;
			\Return $(\textsc{Success},C')$, with $C'$ the unique connected component consisting of all squares $B_{i,j}$ of width $\frac{2^{\ell_C-1}}{N_C}$ that intersect $\Delta'$.\label{newtonaddboxes}
		}
	}
	\Return \textsc{Failure}\;
	\caption{$\textsc{NewtonTest}$}
	\label{newtontest}
\end{algorithm}

The main crux of the \textsc{NewtonTest} (and the QIR method in general) is that we never have to check in advance whether Newton
iteration actually yields an improved approximation of the cluster of roots. Instead,
correctness is verified
independently
using the $\T$-test. In order to achieve quadratic convergence in the
presence of a well isolated root cluster, we assign, in each iteration, an integer $N_C$ to each component $C$. The reader may think of $N_C$ as the actual speed of
convergence to the cluster of roots contained in $C$. Then, in case of success of the
\textsc{NewtonTest}, the component $C$ is replaced by a component $C'\subset C$ of diameter
$w(C')\approx w(C)\cdot N_C^{-1}$. In this case, we ``square the speed' of convergence", that is, we set $N_{C'}:=N_C^2$. If the \textsc{NewtonTest} fails, we fall back to bisection and decrease the
speed of convergence, that is, we set $N_{C'}:=\sqrt{N_C}$ for all components $C'$ into which the component~$C$ is split. Our analysis shows that the \textsc{NewtonTest} is the crucial ingredient for quadratic convergence. More precisely, we prove that, in the worst-case, the number $s$ of components in each sequence $C_1,\ldots,C_s$ as above becomes logarithmic in the length of such a sequence if only bisection would be used; see Lemma~\ref{pathlength}.

We now turn to the proof of termination and correctness of the algorithm. In addition, we derive further properties, which will turn out to be useful in the analysis.

\begin{algorithm}[t!]
	\DontPrintSemicolon
	\SetKwInOut{Input}{Input}\SetKwInOut{Output}{Output}
	\Input{A connected component $C=\{B_1,\ldots,B_{s_C}\}$ consisting of aligned squares $B_i$, each of width $w(B_i)=2^{\ell_C}$.}
	\Output{A list of components $C_j'\subset C$, each consisting of aligned and equally sized squares of width $2^{\ell_C-1}$. The union of all $C_j'$ contains all roots of $F$ that are contained in $C$.}
	\BlankLine
	$C'\as \emptyset$\;
	\For{each $B_i\in C$}{
		Remove $B_i$ from $C$ and subdivide $B_i$ into four equally sized sub-squares $B_{i,j}$, with $j=1,\ldots,4$, and add these to $C'$.\;
	}
	\For{each $B\in C'$}{
		\If{$\T(\Delta_B)=0$
			\atcp{\normalfont \textit{This implies that $B$ contains no root.}}
			}{
			Remove $B$ from $C'$.\label{DiscardCBisect}
		}
	}
	Compute maximal connected components $C_1', \ldots C'_\ell$ from the squares in $C'$.\label{grouping}\;
	\Return $C_1', \ldots C'_\ell$\;
	\caption{$\textsc{Bisection}$}
	\label{bisection}
\end{algorithm}

\begin{theorem}\label{thm:termination}
	The algorithm $\mathbb{C}\textsc{Isolate}$ terminates and returns a correct result. In addition, \emph{at any stage of the algorithm}, it holds that:
	\begin{itemize}
		\item[(a)] For any $(C,N_C)\in\mathcal{C}$, the connected component $C$ consists of disjoint, aligned, and equally-sized squares $B_{1},\ldots,B_{s_C}$, each of width $2^{\ell_C}$ with some $\ell_C\in\mathbb{Z}$.
		\item[(b)] For any two distinct pairs $(C_1,N_{C_1})\in \mathcal{C}$ and $(C_2,N_{C_2})\in \mathcal{C}$, the distance between $C_1$ and $C_2$ is at least $\max(2^{\ell_{C_1}},2^{\ell_{C_2}})$. In particular, the enlarged regions $C_1^+$ and $C_2^+$ are disjoint.

		\item[(c)] The union of all connected components $C$ covers all roots of $F$ contained in $\mathcal{B}$. In mathematical terms,
		      \[
		      	F(z)\neq 0\text{ for all }z\in\mathcal{B}\setminus\bigcup_{C:(C,N_C)\in\mathcal{C}}C.
		      \]

		\item[(d)] For each square $B$ produced by the algorithm that is not equal to the initial square $\mathcal{B}$, the enlarged square $2B$ contains at least one root of $F$.
		\item[(e)] Each component $C$ considered by the algorithm consists of $s_C\le 9\cdot |\MMM(C^+)|$ squares. The total number of squares in all components $C$ is at most $9$-times the number of roots contained in $2\mathcal{B}$, that is,\footnote{We will later prove that even the total number of squares produced by the algorithm in \emph{all iterations} is near-linear in the number $\MMM(2\mathcal{B})$ of roots contained in $2\mathcal{B}$.}
		      $$\sum_{C:\exists (C,N_C)\in\mathcal{C}} s_C\le 9 \cdot |\MMM(2\mathcal{B})|.$$

		\item[(f)] Let $(C, N_C)$ be a pair produced by the algorithm. The sequence of ancestors of a component $C$ produced by the algorithm is recursively defined as follows. It consists of the component $C'$ from which $C$ resulted followed by the ancestors of $C'$. We denote with $\ancstar(C)$, the first ancestor of $C$ for which the Newton Test was successful, if such exists, otherwise $\ancstar(C)=\BBB$. Then $w(C)\le\frac{2w(\ancstar(C))}{ \sqrt{N_C}}\le\frac{2w(\BBB)}{\sqrt{N_C}}$. Moreover,
		\item[(g)] $\frac{\sigma_F(2\mathcal{B})^2}{2^{17}\cdot n^2\cdot w(\mathcal{B})}\le 2^{\ell_C}\le w(\mathcal{B})$
		      and $4\le N_C\le \left(\frac{2^{9}\cdot w(\mathcal{B})}{\sigma_F(2\mathcal{B})}\right)^2$,
		      where $\sigma_F(2\mathcal{B}):=\min_{i:z_i\in 2\mathcal{B}}\sigma_F(z_i)$ is the separation of $F$ restricted to $2\mathcal{B}$.
	\end{itemize}
\end{theorem}

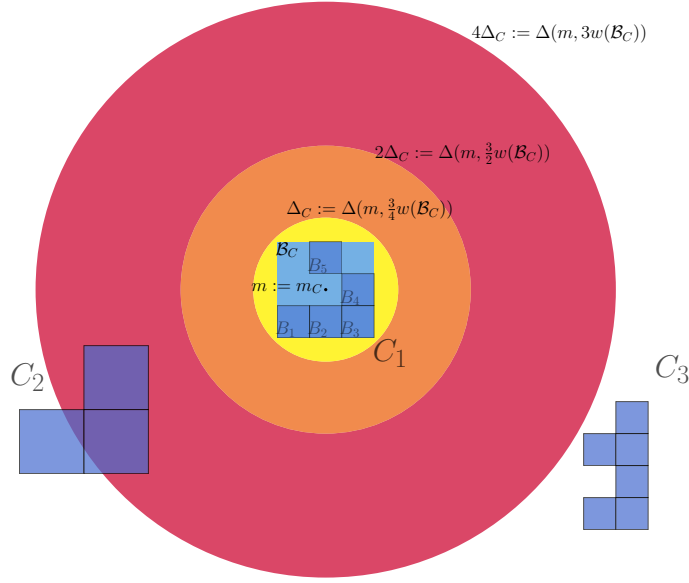
\begin{figure}[t!]
	\begin{center}
		\resizebox{0.75\textwidth}{!}{

\begin{tikzpicture}
\tikzstyle{every node}=[font=\Large,scale=4.5]
\begin{scope}
  \fill[alizarin,opacity=0.8] (6,6) circle (36) node[black,opacity=1] (n1) at (35,38) {$4\Delta_C:=\Delta(m,3w(\BBB_C))$} node[black,scale=2,opacity=1] (n2) at (30,15) {};
  \fill[white] (6,6) circle (18);
  \fill[orange,opacity=0.7] (6,6) circle (18) node[black,opacity=1] at (23,23) {$2\Delta_C:=\Delta(m,\frac{3}{2}w(\BBB_C))$};
  \fill[white] (6,6) circle (9);
  \fill[yellow,opacity=0.8] (6,6) circle (9) node[black,opacity=1] at (11.5,16) {$\Delta_C:=\Delta(m,\frac{3}{4}w(\BBB_C))$};

  \fill[color=mybrightblue] (0,0) -- (12,0) -- (12,12) -- (0,12) -- (0,0) node[black] at (1.2,11) {$\BBB_C$};

  \draw[fill=hanblue,opacity=0.5] (0,0) -- (4,0) -- (4,4) -- (0,4) -- (0,0) node at (1,1) {$B_1$};
  \draw[fill=hanblue,opacity=0.5] (4,0) -- (8,0) -- (8,4) -- (4,4) -- (4,0) node at (5,1) {$B_2$};
  \draw[fill=hanblue,opacity=0.5] (8,0) -- (12,0) -- (12,4) -- (8,4) -- (8,0) node at (9,1) {$B_3$};
  \draw[fill=hanblue,opacity=0.5] (8,4) -- (12,4) -- (12,8) -- (8,8) -- (8,4)node at (9,5) {$B_4$};
  \draw[fill=hanblue,opacity=0.5] (4,8) -- (8,8) -- (8,12) -- (4,12) -- (4,8)node at (5,9) {$B_5$};

\draw[opacity=0.7] node at (14,-2) {\Huge $C_1$};

  \fill (6,6) circle (2mm) node at (1.2,6.2) {$m:=m_C$};

	\begin{scope}[shift={(-32,-17)}]
	\draw[fill=hanblue,opacity=0.7] (0,0) -- (8,0) -- (8,8) -- (0,8) -- (0,0) node at (5,1) {};
  \draw[fill=hanblue,opacity=0.7] (8,0) -- (16,0) -- (16,8) -- (8,8) -- (8,0) node at (9,1) {};
  \draw[fill=hanblue,opacity=0.7] (8,8) -- (16,8) -- (16,16) -- (8,16) -- (8,8) node at (1,12) {\Huge $C_2$};
\end{scope}

\begin{scope}[shift={(34,-24)}]
	\draw[fill=hanblue,opacity=0.7] (4,0) -- (8,0) -- (8,4) -- (4,4) -- (4,0) node at (5,1) {};
  \draw[fill=hanblue,opacity=0.7] (8,0) -- (12,0) -- (12,4) -- (8,4) -- (8,0) node at (9,1) {};
  \draw[fill=hanblue,opacity=0.7] (8,4) -- (12,4) -- (12,8) -- (8,8) -- (8,4)node at (9,5) {};
\end{scope}

\begin{scope}[shift={(34,-16)}]
  \draw[fill=hanblue,opacity=0.7] (4,0) -- (8,0) -- (8,4) -- (4,4) -- (4,0) node at (5,1) {};
  \draw[fill=hanblue,opacity=0.7] (8,0) -- (12,0) -- (12,4) -- (8,4) -- (8,0) node at (9,1) {};
  \draw[fill=hanblue,opacity=0.7] (8,4) -- (12,4) -- (12,8) -- (8,8) -- (8,4)node at (15,12) {\Huge $C_3$};
\end{scope}

\end{scope}

\end{tikzpicture}
		}
	\end{center}
	\caption{
		A component $C_1:=C$ consisting of 5 squares $B_1,\ldots B_5$, the enclosing square $\BBB_C$ with center $m:=m_C$ and the disks $\Delta_C$, $2\Delta_C$ and $4\Delta_C$. The disk $4\Delta_C$ intersects the component $C_2$ but does not intersect the component $C_3$.
	}
	\label{fig:yyy}
\end{figure}

\begin{proof}
	Part (a) follows almost immediately via induction. Namely, a component $C$ consisting of squares of size $2^{\ell_C}$ is either replaced by a single connected component consisting of (at most $4$) squares of width $2^{\ell_C-1}/N_C$ in line~\ref{newtonadd} after \textsc{NewtonTest} was called, or it is replaced by a set of connected components $C'\subset C$, each consisting of squares of size $2^{\ell_C-1}$ in line~\ref{bisectionadd} after \textsc{Bisection} was called.

	For (b), we can also use induction on the number of iterations. Suppose first that a component $C$ is obtained from processing a component $D$ in line~\ref{bisectionadd}. If $C$ is the only connected component obtained from $D$, then, by the induction hypotheses, it follows that the distance to all other components $C'$, with $C'\cap D=\emptyset$, is at least $\max(2^{\ell_D},2^{\ell_{C'}})\ge \max(2^{\ell_C},2^{\ell_{C'}})$. If $D$ splits into several components $C_1,\ldots,C_s$, with $s>1$, their distance to any component $C'$, with $C'\cap D=\emptyset$, is at least $\max(2^{\ell_D},2^{\ell_{C'}})\ge \max(2^{\ell_{C_i}},2^{\ell_{C'}})$ for all $i$. In addition, the pairwise distance of two disjoint components $C_i$ and $C_j$ is at least $2^{\ell_C-1}=2^{\ell_{C_i}}$ for all $i$. Finally, suppose that, in line~\ref{newtonadd}, we replace a component $D$ by a single component $C$. In this case, $C\subset D$ and $C$ consists of squares of width $2^{\ell-1}/N_C$. Hence, the distance from $C$ to any other component $C'$ is also lower bounded by $\max(2^{\ell_C},2^{\ell_{C'}})$.

	\begin{figure}[t!]
		\begin{center}
			\resizebox{0.6\textwidth}{!}{

\begin{tikzpicture}[x=1cm,y=1cm, font=\boldmath]
\tikzstyle{every node}=[font=\Huge,scale=.5]
\begin{scope}


  \fill[color=mybrighterblue,opacity=0.8] (4,0) -- (12,0) -- (12,8) -- (4,8) -- (4,0) node[black] at (5,7) {$\BBB_C$};

  \draw[color=darkgray,fill=hanblue, opacity=0.7] (4,0) -- (8,0) -- (8,4) -- (4,4) -- (4,0) node at (5,1) {\textcolor{black}{$B_1$}};
  \draw[color=darkgray,fill=hanblue, opacity=0.7] (8,0) -- (12,0) -- (12,4) -- (8,4) -- (8,0) node at (9,1) {\textcolor{black}{$B_2$}};
  \draw[color=darkgray,fill=hanblue, opacity=0.7] (8,4) -- (12,4) -- (12,8) -- (8,8) -- (8,4)node at (9,7) {\textcolor{black}{$B_3$}};

  \begin{scope}[shift={(2,3.5)}]
  \draw[color=darkgreen,fill=grannysmithapple] (8,0) -- (8.5,0) -- (8.5,0.5) -- (8,0.5) -- (8,0) node at (9,1) {\textcolor{black}{$C'$}};
  \end{scope}

  \begin{scope}[shift={(1.5,4)}]
  \draw[color=darkgreen,fill=grannysmithapple] (8,0) -- (8.5,0) -- (8.5,0.5) -- (8,0.5) -- (8,0) node at (8.5,1) {};
  \end{scope}
  \begin{scope}[shift={(2,4)}]
  \draw[color=darkgreen,fill=grannysmithapple] (8,0) -- (8.5,0) -- (8.5,0.5) -- (8,0.5) -- (8,0) node at (8.5,1) {};
  \end{scope}

  \fill (8,4) circle (0.5mm) node[above left] {$m$};
  \fill[alizarin] (10.1,4.1) circle (1/8) node at (9.5,3.7) {$\Delta'$};

\end{scope}

\end{tikzpicture}
			}
		\end{center}
		\caption{
			The \textsc{NewtonTest}: If $\T(\Delta')=k_C$, with $\Delta':=\Delta(\tilde{x}_C',\frac{2^{\ell_{C}-3}}{N_C})$, then $\Delta'$ contains exactly $k_C$ roots of $F$. Since $\T(2\Delta_C)=k_C$ and $C^+\subset 2\Delta_C$, it follows that $\Delta'$ contains all roots contained in $C^+$.
			The sub-squares $B_{i,j}$ of width $2^{\ell_C-1}/N_C$ that intersect $\Delta'$ yield a connected component $C'$ of width at most $2^{\ell_C}/N_C\le w(C)/N_C$. In addition, all roots that are contained in $C$ are also contained in $C'$. Further notice that if $\tilde{x}_C'$ is contained in $C$, then $\Delta'$ intersects at most four squares $B_{i,j}$. Otherwise, it intersects at most three squares. In each case, the squares are connected with each other, and the corresponding connected component $C'$ has width at most $2^{\ell_C}/N_C\le w(C)/N_C$.}
		\label{fig:xxx}
	\end{figure}
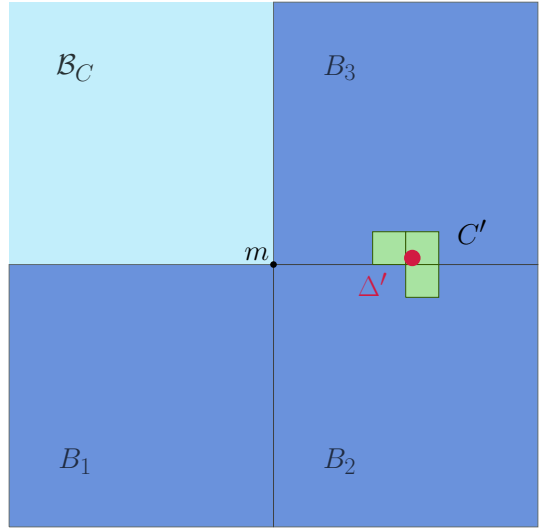

	For (c), notice that in line~\ref{DiscardCBisect} of \textsc{Bisection}, we discard a square $B$ only if the $\T(\Delta_B)=0$. Hence, in this case, $B$ contains no root of $F$. It remains to show that each root of $F$ contained in $C$ is also contained in $C'$, where $C'\subset C$ is a connected component as produced in line~\ref{newtonadd} after \textsc{NewtonTest} was called.
	If $\T(\Delta')=k_C$, then $\Delta'$ contains $k_C$ roots; see Lemma~\ref{softtest:success2}. Hence, since $\Delta'$ is contained in $2\Delta_C$, and since $2\Delta_C$ also contains $k_C$ roots (as $\T(2\Delta_C)=k_C$ holds), it follows that $\Delta'$ contains all roots that are contained in $C$.
	The disk $\Delta'$ intersects no other component $C'\neq C$ as the distance from $C$ to $C'$ is larger than $2^{\ell_C}$, and thus, by induction, we conclude that $(\Delta'\cap\mathcal{B})\setminus C$ contains no root of $F$. This shows that $C'$ already contains all roots contained in $C$.

	We can now prove part (d) and part (e).
	Any square $B\neq\mathcal{B}$ that is considered by the algorithm either results from the \textsc{Bisection} or from the \textsc{NewtonTest} routine. If a square $B$ results from the \textsc{Bisection} routine, then the disk $\Delta_B=\Delta(m_B,w(B))$ contains at least one root of $F$, and thus also $2B$ contains at least one root.  If a square $B$ results from the \textsc{NewtonTest} routine, then $2B$ even contains two roots or more.
	Namely, in this case, $\T(\Delta')=k_C$ holds for the disk $\Delta'=\Delta(m',r')$, with $r'=\frac{1}{4}w(B)$ and $k_C>1$, and thus $\Delta'$ contains $k_C$ roots. Since $2B$ contains the latter disk, $2B$ must contain at least $k_C$ roots. This shows (d). From (d), we immediately conclude that, for each component $C\neq\mathcal{B}$ produced by the algorithm, the enlarged component $C^+$ contains at least one root of $F$.
	In addition, since $C^+$ is contained in $2\mathcal{B}$, each of these roots must be contained in $2\mathcal{B}$. The first part in (e) now follows from the fact that, for a fixed root of $F$, there can be at most $9$ different squares $B$ of the same size such that $2B$ contains this root. From part (b), it follows that, for any two distinct components $C_1$ and $C_2$, the enlarged components $C_1^+$ and $C_2^+$ do not intersect, and thus the total number of squares in all components is upper bounded by $9\cdot |\MMM(2\mathcal{B})|$, which proves the second part in (e).

	For (f), we may assume that $N_C>4$ as, otherwise, the inequality becomes trivial.
	Denote with $C_1,\ldots,C_{s}$ the sequence of ancestors of $C$, with $C_1:=\ancstar(C)$, $C_s=C$, and $C_i\supset C_{i+1}$. By definition of $\ancstar(C)$, it holds that $w(C_2)\le \tfrac{w(C_1)}{N_{C_1}} =\tfrac{w(C_1)}{\sqrt{N_{C_2}}}$, since the step from $C_1$ to $C_2$ is a quadratic step. It follows that
	\[
		w(C)
		= w(C_s)
		\le w(C_{s-1})
		\le w(C_2)
		\le \tfrac{w(C_1)}{N_{C_1}}
		= \tfrac{w(C_1)}{\sqrt{N_{C_2}}}
		\le \tfrac{w(C_1)}{\sqrt{N_{C_s}}}
		= \tfrac{w(\ancstar(C))}{\sqrt{N_{C}}},
	\]
	since $N_{C_i}= \sqrt{N_{C_{i-1}}}$ for $i=2,\ldots, s$.

	\begin{figure}[t!]
		\centering{
			\resizebox{0.6\textwidth}{!}{

\begin{tikzpicture}[x=1cm,y=1cm]
\tikzstyle{every node}=[font=\Huge,scale=1.1]
\begin{scope}
  \fill[darkyellow,opacity=0.2] (6,6) circle (9) node[black,opacity=1] at (12,15) {$\Delta_C:=\Delta(m,\frac{3}{4}w(\BBB_C))$};

  \fill[color=brightgray] (0,0) -- (12,0) -- (12,12) -- (0,12) -- (0,0) node[black] at (1,11) {$\BBB_C$};

  \draw[color=gray,fill=hanblue] (0,0) -- (4,0) -- (4,4) -- (0,4) -- (0,0) node at (1,1) {};
  \draw[color=gray,fill=hanblue] (4,0) -- (8,0) -- (8,4) -- (4,4) -- (4,0) node at (5,1) {};
  \draw[color=gray,fill=hanblue] (8,0) -- (12,0) -- (12,4) -- (8,4) -- (8,0) node at (9,1) {};
  \draw[color=gray,fill=hanblue] (8,4) -- (12,4) -- (12,8) -- (8,8) -- (8,4)node at (9,7) {};
  \draw[color=gray,fill=hanblue] (4,8) -- (8,8) -- (8,12) -- (4,12) -- (4,8)node at (5,9) {};


  \draw[color=darkgray,fill=grannysmithapple] (0,0) -- (2,0) -- (2,2) -- (0,2) -- (0,0)   node at (9,1) {};
  \begin{scope}[shift={(2,0)}]
      \draw[color=darkgray] (0,0) -- (2,0) -- (2,2) -- (0,2) -- (0,0)   node at (9,1) {};
  \end{scope}
  \begin{scope}[shift={(2,2)}]
      \draw[color=darkgray] (0,0) -- (2,0) -- (2,2) -- (0,2) -- (0,0)   node at (9,1) {};
  \end{scope}
  \begin{scope}[shift={(0,2)}]
      \draw[color=darkgray] (0,0) -- (2,0) -- (2,2) -- (0,2) -- (0,0)   node at (9,1) {};
  \end{scope}

  \begin{scope}[shift={(4,0)}]
      \draw[color=darkgray] (0,0) -- (2,0) -- (2,2) -- (0,2) -- (0,0)   node at (9,1) {};
      \begin{scope}[shift={(2,0)}]
          \draw[color=darkgray,fill=grannysmithapple] (0,0) -- (2,0) -- (2,2) -- (0,2) -- (0,0)       node at (9,1) {};
      \end{scope}
      \begin{scope}[shift={(2,2)}]
          \draw[color=darkgray,fill=grannysmithapple] (0,0) -- (2,0) -- (2,2) -- (0,2) -- (0,0)       node at (9,1) {};
      \end{scope}
      \begin{scope}[shift={(0,2)}]
          \draw[color=darkgray,fill=grannysmithapple] (0,0) -- (2,0) -- (2,2) -- (0,2) -- (0,0)       node at (9,1) {};
      \end{scope}
  \end{scope}

  \begin{scope}[shift={(8,0)}]
      \draw[color=darkgray] (0,0) -- (2,0) -- (2,2) -- (0,2) -- (0,0)   node at (9,1) {};
      \begin{scope}[shift={(2,0)}]
          \draw[color=darkgray,fill=grannysmithapple] (0,0) -- (2,0) -- (2,2) -- (0,2) -- (0,0)       node at (9,1) {};
      \end{scope}
      \begin{scope}[shift={(2,2)}]
          \draw[color=darkgray,fill=grannysmithapple] (0,0) -- (2,0) -- (2,2) -- (0,2) -- (0,0)       node at (9,1) {};
      \end{scope}
      \begin{scope}[shift={(0,2)}]
          \draw[color=darkgray,fill=grannysmithapple] (0,0) -- (2,0) -- (2,2) -- (0,2) -- (0,0)       node at (9,1) {};
      \end{scope}
  \end{scope}

  \begin{scope}[shift={(8,4)}]
      \draw[color=darkgray,fill=grannysmithapple] (0,0) -- (2,0) -- (2,2) -- (0,2) -- (0,0)   node at (9,1) {};
      \begin{scope}[shift={(2,0)}]
          \draw[color=darkgray] (0,0) -- (2,0) -- (2,2) -- (0,2) -- (0,0)       node at (9,1) {};
      \end{scope}
      \begin{scope}[shift={(2,2)}]
          \draw[color=darkgray] (0,0) -- (2,0) -- (2,2) -- (0,2) -- (0,0)       node at (9,1) {};
      \end{scope}
      \begin{scope}[shift={(0,2)}]
          \draw[color=darkgray] (0,0) -- (2,0) -- (2,2) -- (0,2) -- (0,0)       node at (9,1) {};
      \end{scope}
  \end{scope}

  \begin{scope}[shift={(4,8)}]
      \draw[color=darkgray,fill=grannysmithapple] (0,0) -- (2,0) -- (2,2) -- (0,2) -- (0,0)   node at (9,1) {};
      \begin{scope}[shift={(2,0)}]
          \draw[color=darkgray] (0,0) -- (2,0) -- (2,2) -- (0,2) -- (0,0)       node at (9,1) {};
      \end{scope}
      \begin{scope}[shift={(2,2)}]
          \draw[color=darkgray,fill=grannysmithapple] (0,0) -- (2,0) -- (2,2) -- (0,2) -- (0,0)       node at (9,1) {};
      \end{scope}
      \begin{scope}[shift={(0,2)}]
          \draw[color=darkgray,fill=grannysmithapple] (0,0) -- (2,0) -- (2,2) -- (0,2) -- (0,0)       node at (9,1) {};
      \end{scope}
  \end{scope}
\end{scope}

\end{tikzpicture}
			}
		}
		\caption{
			The \textsc{Bisection} routine: The green (brighter) sub-squares are all squares $B$ for which $\T(\Delta_B)\neq 0$. They are grouped together into three maximal connected components, which contain all roots contained in $C$. All other sub-squares are discarded.
		}
		\label{fig:zzz}
	\end{figure}
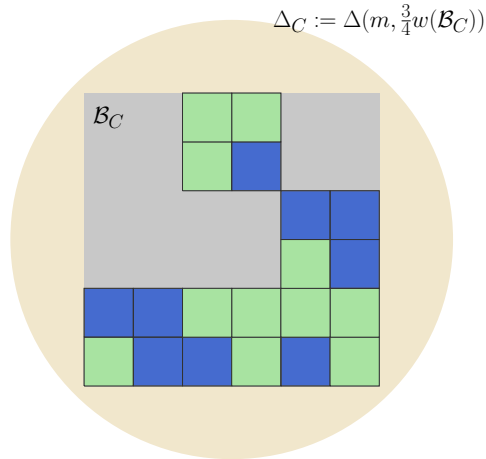

	We can now show that the algorithm terminates; the inequalities in (g) will then follow from the proof of termination: Suppose that the algorithm produces a sequence
	$C_1,C_2,\ldots,C_s$ of connected components, with $s\ge\log n+6$ and $C_1\supset C_2\supset\cdots\supset C_s$. If, for at least one index $i\in\{1,\ldots,s-1\}$, $C_{i+1}$ is obtained from $C_{i}$ via a quadratic step, then $w(C_{i+1})\le w(C_i)/N_{C_{i}}\le w(C_i)/4$. Hence, in this case, we also have $w(C_s)\le \frac{w(C_1)}{4}$.
	Now, suppose that each $C_{i+1}$ is obtained from $C_i$ via a linear step, then each square in $C_i$
	has size $2^{\ell_{C_1}-i+1}$, and thus $w(C_s)\le 9n\cdot 2^{\ell_{C_1}-s+1}\le \frac{w(C_1)}{2}$. This shows that, after
	at most $\log n+6$ iterations, the width of each connected component is halved. Hence, in order to prove
	termination of the algorithm, it suffices to prove that each component $C$ of small enough width is terminal,
	that is $C$ is replaced by an isolating disk in line~\ref{addisolating} or discarded in \textsc{NewtonTest} or \textsc{Bisection}. The following argument shows that each component $C$ of width smaller
	than $w:=\frac{1}{32}\cdot\sigma_F(2\mathcal{B})$ that is not discarded is replaced by an isolating disk. We have already shown that $C^+$
	must contain a root $\xi$ of $F$, and thus we have $|m_C-\xi|<2w(C)<\sigma_F(\xi)/16$ and $r_C<\sigma_F(\xi)/16$. We conclude that the disks $2\Delta_C$ and $\Delta(m_C,8r_C)$ are both
	isolating for $\xi$. Then, Lemma~\ref{softtest:success2} guarantees that $\T(2\Delta_C)=1$ and $\T(4\Delta_C)=1$ hold. Hence, if $4\Delta_C$ intersects no other component $C'\neq C$, then the
	algorithm replaces $C$ by the isolating disk $2\Delta_C$ in line~\ref{addisolating}, because the if-clause in line~\ref{testbeforenewton} succeeds with $k_C=1$. It remains to show that the latter
	assumption is always fulfilled. Namely, suppose that $4\Delta_C$ intersects a component $C'\neq C$, and let
	$B$ and $B'$ be arbitrary squares contained in $C$ and $C'$, respectively. Then, the enlarged squares $2B$ and $2B'$ contain roots $\xi$ and $\xi'$, respectively, and $\xi$ and $\xi'$ must be distinct as $C^+$ and $(C')^+$ are disjoint. Hence, the distance between $B$ and $B'$, and thus also the distance $\delta$ between $C$ and $C'$, must be larger than
	$\sigma_F(2\mathcal{B})-2^{\ell_C}-2^{\ell_{C'}}=32w-2^{\ell_C}-2^{\ell_{C'}}\ge 31 w-2^{\ell_{C'}}$. Hence, if $2^{\ell_{C'}}\le 25w$, then $4\Delta_C\subset\Delta(m_C,6w)$ does not intersect $C'$.
	Vice versa, if $2^{\ell_{C'}}>25w$, then the distance between $C$ and $C'$ is at least $\max(2^{\ell_C},2^{\ell_{C'}})> 25w$, and thus $4\Delta_C$ does not intersect $C'$ as well. Notice that (g) now follows almost directly from the above considerations. Indeed, let $C\neq \mathcal{B}$ be an arbitrary component $C$ and let $D$ be any component that contains $C$.
	Since $D$ is not terminal, we conclude that $w(D)\ge w$, and thus $N_D\le \frac{4w(\mathcal{B})}{w}$ according to (f). Since $N_C$ is smaller than or equal to the square of the maximum of all values $N_D$, the second inequality in (g) follows. The first inequality follows from the fact that $2^{\ell_{C}}\ge \min_{D:C\subset D}\frac{2^{\ell_D-1}}{N_D}\ge \frac{w}{9n\cdot\max_{D:C\subset D} N_D}$.

	For correctness, we remark that each disk $D$ returned by the algorithm is actually isolating for a root of $F$
	contained in $2\mathcal{B}$ and that $2D$ also isolates this root. Namely, for each component $C$ produced by
	the algorithm, the enlarged component $C^+$ contains at least one root. Now, if the if-clause in line~\ref{testbeforenewton} succeeds on $C$ with $k_C=1$, it holds that $\T(2\Delta_C)=1$, and thus the disk $2\Delta_C$ contains exactly one root $\xi$. Hence, since $\Delta_C$ contains $C^+$, this root must be
	contained in $C^+$. In addition, if also $\T(4\Delta_C)=1$ holds, then the disk $4\Delta_C$ isolates $\xi$ as well.
	Finally, it remains to show that the algorithm returns an isolating disk for each root $\xi$ that is contained
	in $\mathcal{B}$. From (a) and (c), we conclude that there is a unique maximal sequence $\mathcal{S}=C_1,C_2,\ldots,C_s$ of connected
	components, with $C_1\supset C_2\supset\cdots\supset C_s$, such that each $C_i$ contains $\xi$.
	Now, when processing $C_s$, $C_s$ cannot be replaced by other connected components $C'\subset C_s$ as one of these components would contain $\xi$, and this would contradict the assumption that the sequence $\mathcal{S}$ is maximal. Since $C_s$ contains $\xi$, it cannot be discarded in \textsc{Bisection} or \textsc{NewtonTest}, hence $C_s$ is replaced by an isolating disk for $\xi$ in line~\ref{addisolating}.
\end{proof}

\noindent\emph{Remarks.} We remark that our requirement on the input polynomial $F$ to have only simple roots in
$2\mathcal{B}$ is only needed for the termination of the algorithm. Running the algorithm on an arbitrary polynomial (possibly having multiple roots) yields isolating disks for the simple roots as well as arbitrarily small connected components converging against the multiple roots of $F$ in $\mathcal{B}$. Namely, if $\mathcal{B}$ is not discarded in the first iteration, then the
enlargement $C^+$ of each component $C$ contains at least one root. Since $C$ consists of at most $9n$ squares, each of size $2^{\ell_C}$,
it holds that each point in $C$ approximates a root of $F$ to an error of less than $n\cdot 2^{\ell_C+4}$. In addition, the union of all components covers all roots contained in
$\mathcal{B}$, and thus our algorithm yields $L$-bit approximations of all roots in $\mathcal{B}$ if we iterate until $\ell_C\le-4-\log n-L$ for all components $C$. In the special situation, where
we run the algorithm on an input square that is known to contain all roots and if, in addition, the number $k$ of distinct
roots of $F$ is given as input, our algorithm can be used to return isolating regions for all roots. Namely, in this situation, we may
proceed until the total number of connected components $C$ equals $k$. Then, each of the enlarged components $C^+$
isolates a root of $F$. The latter problem is of special interest in the context of computing a cylindrical algebraic decomposition, where we have to isolate the roots of a not necessarily square-free polynomial with algebraic coefficients. In this case, it might be easier to first compute $k$ via a symbolic pre-computation and to consider sufficiently good approximations of the initial polynomial instead of computing approximations of the square-free part of $F$. A corresponding approach based on approximate polynomial factorization has been presented in~\cite{MSW-rootfinding2013}, and we refer the reader to this work for more details and for a motivation of the problem.


\section{Complexity Analysis}\label{sec:complexity}

We split the analysis of our algorithm into two parts. In the first part, we focus on the
number of iterations that are needed to isolate the roots of $F(x)$ that are contained in a given square $\mathcal{B}$. We will see that this number is near-linear\footnote{More precisely, it is linear in $|\MMM(2\mathcal{B})|$ up to a factor that is polynomially bounded in $\log n$, $\log\LOG(w(\mathcal{B}))$, and $\log\LOG(\sigma_F(2\mathcal{B})^{-1})$.
	If $2\mathcal{B}$ contains no root, then there is only one iteration.} in $|\MMM(2\mathcal{B})|$, the number of roots contained in the enlarged square $2\mathcal{B}$. We further remark that, for any fixed non-negative constant $\epsilon$, the total number of iterations is near-linear in $|\MMM((1+\epsilon)\cdot\mathcal{B})|$; however, for the sake of simplifying analysis, we only provide details for the special case $\epsilon=1$.
Hence, we conclude that our algorithms performs near-optimal with respect to the number of subdivision steps if the input square $\mathcal{B}$ has the property that each root contained in $(1+\epsilon)\cdot \mathcal{B}$ is also contained in $\mathcal{B}$; in particular, this is trivially fulfilled if $\mathcal{B}$ is chosen large enough to contain all roots of $F$.

In the second part of our analysis, we give bounds on the number of bit operations that are needed to process a component $C$. This eventually yields a bound on the overall
bit complexity that is stated in terms of the degree of $F$, the absolute values and the separations of the roots in $\MMM(2\mathcal{B})$, and the absolute value of the derivative $F'$ at these roots. For the special case, where our algorithm is used to isolate all roots of a polynomial of degree $n$ with integer coefficients of bit size less than $\tau$, the bound on the bit complexity simplifies to $\tilde{O}(n^3+n^2\tau)$.

\subsection{Size of the Subdivision Tree}\label{subsec:treesize}

We consider the subdivision tree $\mathcal{T}_{\mathcal{B}}$, or simply $\mathcal{T}$,
induced by our algorithm, where $\mathcal{B}$ is the initial square/component. More specifically, the nodes of the (undirected) graph $\mathcal{T}$ are the pairs
$(C,N_C)\in \mathcal{C}$ produced by the algorithm, and two nodes $(C,N_C)$ and $(C',N_{C'})$
are connected via an edge if and only if $C\subset C'$ (or $C'\subset C$) and there exists no other component $C''$ with $C\subset C''\subset C'$ ($C'\subset C''\subset C$). In the first
case, we say that $(C,N_C)$ is a child of $(C',N_{C'})$, whereas, in the second case, $(C,N_C)$ is a parent of
$(C',N_{C'})$. For brevity, we usually omit the integer $N_C$, and just refer to $C$ as the nodes of $\mathcal{T}$. Notice that, according to Theorem~\ref{thm:termination}, the so obtained graph is indeed a tree rooted at $\mathcal{B}$. A node $C$ is called \emph{terminal} if and only if it has no children. We further use the following definition to refer to some special nodes:

\begin{definition}\label{def:splitting}
	A node $(C,N_C)\in\mathcal{T}$ is called \emph{special}, if one of the following conditions is fulfilled:
	\begin{itemize}
		\item The node $(C,N_C)$ is terminal.
		\item The node $(C,N_C)$ is the root of $\mathcal{T}$, that is, $(C,N_C)=(\mathcal{B},4)$.
		\item The node $(C,N_C)$ is the last node for which \textsc{Bisection} is called in the preprocessing phase of the algorithm. We call this node the \emph{base} of $\mathcal{T}$. Notice that the first part of the tree consists of a unique path connecting the root and the base of the tree.
		\item For each child $D$ of $C$, it holds that $\MMM(D^+)\neq \MMM(C^+)$.
	\end{itemize}
\end{definition}

Roughly speaking, except for the root and the base of $\mathcal{T}$, special nodes either isolate a root of $F$ or they are split into two or more disjoint clusters each containing roots of $F$. More precisely, from Theorem~\ref{thm:termination}, we conclude that, for any two distinct nodes $C,D\in\mathcal{T}$, the enlarged regions $\MMM(C^+)$ and $\MMM(D^+)$ are
either disjoint or one of the nodes is an ancestor of the other one. In the latter case, we have $C^+\subset D^+$ or $D^+\subset C^+$. Since, for any two
children $D_1$ and $D_2$ of a node $C$, the enlarged regions $D_1^+$ and $D_2^+$ are
disjoint, we have $\sum_{i=1}^{k} \MMM(D_i^+)\le \MMM(C^+)$, where $D_1$ to
$D_k$ are the children of $C$. Hence, since each $D_i^+$ contains at least one root, the fourth
condition in Definition~\ref{def:splitting} is violated if and only if $C$ has exactly one
child $D$ and $\MMM(C^+)=\MMM(D^+)$. The number of special nodes is at most $2\cdot(1+|\MMM(2\mathcal{B})|)$ as there is one root and one base, at most $|\MMM(2\mathcal{B})|$ terminal nodes $C$ with $C\neq\mathcal{B}$, and each occurrence of a special node, which fulfills the fourth condition, yields a reduction of the non-negative number $\sum_{C} (|\MMM(C^+)|-1)$ by at least one.
The subdivision tree $\mathcal{T}$ now decomposes into special nodes and sequences of non-special nodes $C_1,\ldots,C_s$, with $C_1\supset C_2\supset\cdots\supset C_s$, that connect two consecutive special nodes. The remainder of this section is dedicated to the proof
that the length $s$ of such a sequence is bounded by some value $s_{\max}$ of size
\begin{align}\label{def:smax}
	s_{\max} & =O\left(\log n + \log\LOG(w(\mathcal{B}) + \log\LOG(\sigma_F(2\mathcal{B})^{-1}) \right)             \\ \nonumber
	         & =O\left(\log \left(n\cdot \LOG(w(\mathcal{B}))\cdot \LOG(\sigma_F(2\mathcal{B})^{-1})\right)\right).
\end{align}

For the proof, we need the following lemma, which provides sufficient conditions for the success of the \textsc{NewtonTest}.

\begin{lemma}[Success of \textsc{NewtonTest}]
	\label{newtonsucceeds}
	Let $C=\{B_1,\ldots,B_{s_C}\}$ be a non-terminal component with $\mathcal{B}\setminus C\neq \emptyset$, let $\mathcal{B}_C$ be the corresponding enclosing square of width $w(C)$ and center $m=m_C$, and let $\Delta:=\Delta_C=\Delta(m,r)$, with $r:=\frac{3}{4}w(C)$, be the corresponding enclosing disk. Let $z_1,\ldots,z_{k}$ be the roots contained in the enlarged component $C^+$, and suppose that all these roots are contained in a disk $\Delta'':=\Delta(m'',r'')$ of radius $r''=2^{-20-\log n} \frac{r}{N_C}$. In addition, assume that the disk $\Delta(m,2^{2\log n +20} N_C r)$ contains none of the roots $z_{k+1},\ldots,z_{n}$. Then, the algorithm $\mathbb{C}\textsc{Isolate}$ performs a quadratic step, that is, $C$ is replaced by a single component $C'$ of width $w(C')\le \frac{w(C)}{N_C}$.
\end{lemma}
\begin{proof}
	We first argue by contradiction that $4\Delta$ does not intersect any other component $C'$, which implies that the first condition of the \textbf{and} in the if-clause in line~\ref{testbeforenewton} is fulfilled. If $4\Delta$ intersects $C'$, then the
	distance between $C$ and $C'$ is at most $8r$, and thus
	$2^{\ell_{C'}}<8r$ as the distance between $C$ and $C'$ is at least $\max(2^{\ell_C},2^{\ell_{C'}})$. Hence, we conclude that the disk $\Delta(m,64r)$ completely contains $2B'$ for some square $B'$ of $C'$.
	Since $2B'$ also contains at least one root and since each such root must be distinct from any of the roots $z_1,\ldots,z_k$, we get a contradiction.

	According to our assumptions, each of the two disks $\Delta$ and $8\Delta$ contains the roots $z_1,\ldots,z_k$ but no other root of $F$.
	Hence, according to Lemma~\ref{softtest:success2}, $\T(2\Delta)=\T(4\Delta)=k$ holds. Since we assumed $C$ to be non-terminal, we must have $k\ge 2$, and thus the algorithm reaches line~\ref{ReachesNewtonTest} and the \textsc{NewtonTest} is called.
	We assumed that $C$ does not entirely cover the initial square $\mathcal{B}$, hence, in a previous iteration, we must have discarded a square of width $2^{\ell_C}$ or more whose boundary shares at least one point with the boundary of $C$.
	Hence, we can choose a point in such a square as the point $x_C\in\mathcal{B}\setminus C$ in the \textsc{NewtonTest} such that the distance from $x_C$ to $C$ is equal to $2^{\ell_C-1}$ and such that the distance from $x_C$ to the boundary of $\mathcal{B}$ is at least $2^{\ell_C-1}$. Notice that also the distance from $x_C$ to any other component $C'$ is at least $2^{\ell_C-1}$, and thus the distance from $x_C$ to any root of $F$ is at least $2^{\ell_C-1}$, which is larger than or equal to $\frac{r}{27n}$ as $C$ consists of at most $9n$ squares. From our assumptions, we thus conclude that
	\begin{align*}
		  &            & |x_C - m''| \le 4r \quad\text{and}\quad
		|x_C - z_i| \ge \frac{r}{27n} \quad\text{ for } i\le k,\\
		  & \text{and} &                                                                  \\
		  &            & |x_C - z_i| \ge 2^{20} n^2 N_C\cdot r - 4r>2^{19} n^2 N_C\cdot r
		\quad\text{ for }i>k.
	\end{align*}
	Using the fact that $\frac{F'(x)}{F(x)}=\sum_{i=1}^n\frac{1}{x-z_i}$ for any $x$ with $F(x)\neq 0$, we can bound the distance from the Newton iterate $x_C'$ as defined in (\ref{def:newton}) to the ``center'' $m''$ of the cluster of roots:
	\begin{align*}
		  & \left| \frac{1}{k} \frac{(x_C-m'')F'(x_C)}{F(x_C)}
		-1 \right|
		=
		\left|
		\frac{1}{k} \sum_{i=1}^{k} \frac{x_C-m''}{x_C-z_i}
		+
		\frac{1}{k} \sum_{i>k} \frac{x_C-m''}{x_C-z_i}
		-1 \right|\\
		  & =
		\frac{1}{k}
		\left|
		\sum_{i=1}^{k} \frac{z_i-m''}{x_C-z_i}
		+
		\sum_{i>k} \frac{x_C-m''}{x_C-z_i}
		\right|
		\le
		\frac{1}{k}
		\sum_{i=1}^{k} \frac{|z_i-m''|}{|x_C-z_i|}
		+
		\sum_{i>k} \frac{|x_C-m''|}{|x_C-z_i|}\\
		  & \le \frac{r''}{r/(27n)} + \frac{n-k}{k}\frac{4r}{n^2 2^{19}N_Cr}
		<
		\frac{27nr}{rn^2 2^{20}N_C}
		+
		\frac{4nr}{r n^2 2^{20} N_C}\le
		\frac{1}{2^{14} n N_C}.
	\end{align*}
	Hence, there is an $\eps\in\mathbb{C}$, with $|\eps|<\frac{1}{2^{14}nN_C}$, such that $\frac{1}{k} \frac{(x_C-m'')F'(x_C)}{F(x_C)}=1+\eps$. This implies that $\frac{|F'(x_C)|}{|F(x_C)|}\ge \frac{1}{|x_C-m''|}\ge \frac{1}{4r}$, and thus the \textsc{NewtonTest} must reach line~\ref{resultless} as Algorithm~\ref{algo:softpredicate} must return True or Undecided. With $x_C'=x_C-k\cdot\frac{F(x_C)}{F'(x_C)}$, it further follows that
	\begin{align*}
		|m'' - x_C'|
		  & =
		|m''-x_C|
		\cdot
		\left|
		1 -
		\frac{1}
		{ \frac{1}{k} \frac{(x_C-m'')F'(x_C)}{F(x_C)} }
		\right|
		=
		|m''-x_C|
		\cdot
		\left|
		1 -
		\frac{1}
		{ 1 + \eps }
		\right|\\
		  & =
		\left|
		\frac{\eps(m''-x_C)}{1+\eps}
		\right|
		\le
		\frac{4r}{2^{13}nN_C}
		\le
		\frac{r}{2^{11}nN_C}< \frac{2^{\ell_C}}{128N_C}.
	\end{align*}
	We can therefore bound
	\begin{align*}
		|\tilde x_C' - m''|
		  & \le
		|\tilde x_C' - x_C'|
		+
		|x_C' - m''|
		\le
		\frac{2^{l_C}}{64N_C}
		+
		|x_C' - m''|\\
		  & \le
		\frac{2^{l_C}}{64N_C}
		+\frac{2^{l_C}}{128N_C}
		<
		\frac{2^{l_C}}{32N_C}.
	\end{align*}

	Since the distance from $m''$ to any of the roots $z_1,\ldots,z_k$ is also smaller than $r''<\frac{2^{l_C}}{32N_C}$, we conclude that the
	disk $\Delta(\tilde{x}_C',\frac{2^{l_C}}{16N_C})$ contains all roots $z_1,\ldots,z_k$.
	Hence, we conclude that $\Delta':=\Delta(\tilde{x}_C',\frac{2^{l_C}}{8N_C})$ is $(\frac{1}{2},\frac{4}{3})$-isolating for the roots $z_1,\ldots,z_k$, and thus  $\T(\Delta')=k$ must hold according to Lemma~\ref{softtest:success2}. This shows that we reach line~\ref{newtonaddboxes} and that the \textsc{NewtonTest} returns \textsc{Success}.
\end{proof}

In essence, the above lemma states that, in case of a well separated cluster of roots
contained in some component $C$, our
algorithm performs a quadratic step. That is, it replaces the component $C$ by a component $C'$ of width
$w(C')\le \frac{2^{\ell_C}}{N_C}\le\frac{w(C)}{N_C}$, which contains all roots that are contained in $C$.
Now, suppose that there exists a sequence $C_1,\ldots,C_s$ of non-special nodes, with
$C_1\supset\cdots\supset C_s$, such that $C_s$ has much smaller width than $C_1$. Then, $C_1$ contains a cluster of nearby roots but no other root of $F$. We will see
that, from a considerably small (i.e.,
comparable to the bound in (\ref{def:smax})) index on, this cluster is also well separated from the
remaining roots (with respect to the size of $C_i$) such that the requirements in the above lemma are fulfilled. As a
consequence, only a small number of steps from $C_i$ to $C_{i+1}$ are linear, which in turn implies
that the whole sequence has small length. For the proof, we need to consider a sequence $(s_i)_i=(x_i,n_i)_i$, which we define in a rather abstract way. The rationale behind our choice for $s_i$ is that, for all except a small number of indices and a suitable choice for $s_i$, the sequence $(s_i)_i$ behaves similarly to the sequence $(2^{\ell_{C_i}},\log\log N_{C_i})_i$. We remark that $(s_i)_i$ has already been introduced in~\cite{Sagraloff2015}, where it serves as a crucial ingredient for the analysis of the real root isolation method \textsc{ANewDSC}.

\begin{lemma}[\cite{Sagraloff2015}, Lemma~25]
	\label{sequence}
	Let $w$, $w'\in\mathbb{R}^+$ be two positive reals with $w>w'$, and let $m\in\mathbb{N}_{\ge 1}$ be a positive integer. We recursively define the sequence $(s_i)_{i\in\mathbb{N}_{\ge 1}}:=((x_i,n_i))_{i\in\mathbb{N}_{\ge 1}}$ as follows: Let $s_1=(x_{1},n_1):=(w,m)$, and
	\[
		s_{i+1}=\left(x_{i+1},n_{i+1}\right):=\begin{cases}
		\left(\epsilon_{i}\cdot x_{i},n_{i}+1\right)\text{ with an } \epsilon_{i}\in [0,\frac{1}{N_{i}}],   & \text{if }\frac{x_{i}}{N_{i}}\ge w'\\
		\left(\delta_{i}\cdot x_{i},\max(1,n_{i}-1)\right)\text{ with a } \delta_{i}\in [0,\frac{1}{2}],  & \text{if }\frac{x_{i}}{N_{i}}<w',
		\end{cases}
	\]
	where $N_i:=2^{2^{n_{i}}}$ and $i\ge 1$. Then, the smallest index $i_0$ with $x_{i_0}\le w'$ is bounded by $8(n_1+\log\log \max(4,\frac{w}{w'}))$.
\end{lemma}

We are now ready to prove the claimed bound on the maximal length of a sequence of non-special nodes:

\begin{lemma}
	\label{pathlength}
	Let $\P=(C_1,N_1),\ldots, (C_s,N_s)$, with $C_1\supset\cdots\supset C_s$, be a sequence of consecutive non-special nodes. Then, we have $s\le s_{\max}$ with an $s_{\max}$ of size
	\begin{align*}
		s_{\max} & =O\left(\log n+\log\LOG(w(\mathcal{B}))+\log\LOG(\sigma_F(B^+)^{-1}) \right)                         \\
		         & =O\left(\log \left(n\cdot \LOG(w(\mathcal{B}))\cdot \LOG(\sigma_F(2\mathcal{B})^{-1})\right)\right).
	\end{align*}
\end{lemma}

\begin{proof}
	We distinguish two cases. We first consider the special case, where $\P$ is an arbitrary sub-sequence of the unique initial sequence from the child of the root of the tree to the parent of the base of the tree; if there exists no non-special root in between the root and the base of the tree, there is nothing to prove. Due to Theorem~\ref{thm:termination}, part (e), $C_s$ consists of at most $9\cdot |\MMM(2\BBB)|$ squares. It follows that $2^{2s}\le 9n$ as $C_i$ consists of at least $2^{2i}$ squares. This yields $s=O(\log n)$.

	We now come to the case, where we can assume that each $C_i$ is a successor of the base of the tree. In particular, we have $\mathcal{B}\setminus C_i\neq \emptyset$.  W.l.o.g., we may further assume that $z_1,\ldots,z_k$ are the roots contained in the enlarged component $C_1^+$. Since all $C_i$ are assumed to be non-special, each $C_i^+$ contains $z_1$ to $z_k$ but no other root of $F$. Let $w_i:=w(C_i)$ be the width of the component $C_i$, $r_i:=\frac{3}{4}\cdot w_i$ be the radius of the enclosing disk $\Delta_i:=\Delta_{C_i}$, and $2^{\ell_i}:=2^{\ell_{C_i}}$ be the width of each of the squares into which $C_i$ decomposes. Notice that, for each index $i$, the enlarged component
	$C_i^+$ is contained in the disk $2\Delta_i$ of radius $\frac{3}{2}\cdot w_i$, and thus the disk $2\Delta_s$ of radius
	$\frac{3}{2}\cdot w_s$ contains the roots $z_1$ to $z_k$.
	We now split the sequence $\mathcal{P}$ into three (possibly empty) subsequences
	$\P_1=(C_1,N_1),\ldots, (C_{i_1},N_{i_1})$,
	$\P_2=(C_{i_1+1},N_{i_1+1}),\ldots, (C_{i_2},N_{i_2})$, and
	$\P_3=(C_{i_2+1},N_{i_2+1}),\ldots, (C_{s},N_{s})$, where $i_1$ and $i_2$ are defined as follows:
	\begin{itemize}
		\item $i_1$ is the first index with $2^{\ell_{1}}>2^{3\log n+32}\cdot N_{i_1}\cdot 2^{\ell_{i_1}}$. If there exists no such index, we set $i_1:=s$. Further notice that, for any index $i$ larger than $i_1$, we also have $2^{\ell_{1}}>2^{3\log n+32}\cdot N_{i}\cdot 2^{\ell_{i}}$, which follows from induction and the fact that $2^{\ell_i}$ and $N_i$ are replaced by $\frac{2^{\ell_i}}{2 N_i}$ and $N_i^2$ in a quadratic step.
		\item $i_2$ is the first index larger than or equal to $i_1$ such that the step from $i_2$ to $i_2+1$ is quadratic and $2^{\ell_{s}}\cdot 2^{3\log n+32}\cdot N_{i_2}\ge 2^{\ell_{i_2}}$. If there exists no such index $i_2$, we set $i_2:=s$.
	\end{itemize}
	From the definition of $i_2$, it is easy to see that $\P_3$ has length bounded by
	$O(\log n)$. Namely, if $i_2=s$, there is nothing to prove, hence we may assume that the step from
	$i_2$ to $i_2+1$ is quadratic and $2^{\ell_s}\ge 2^{-3\log n-32}\cdot N_{i_2}^{-1}\cdot 2^{\ell_{i_2}}=2^{-3\log n-31}\cdot 2^{\ell_{i_2+1}}$. Hence, we conclude that $s-(i_2+1)\le 3\log n+31$ as $\ell_i$ is reduced by at least~$1$ in each step.

	Let us now consider an arbitrary index $i$ from the sequence $\P_2$. The distance from an
	arbitrary point in $C_i^+$ to the boundary of $C_1^+$ is at least $2^{\ell_1-1}\ge 2^{3\log n+31}\cdot N_{i}\cdot 2^{\ell_i}>2^{2\log n+20}\cdot N_i\cdot r_i$, where the latter inequality
	follows from $r_i=\frac{3}{4}w_i\le \frac{3}{4}\cdot 9n\cdot 2^{\ell_i}$. Since $C_1^+$
	contains only the roots $z_1,\ldots,z_k$, this implies that the distance from an arbitrary
	point in $C_i^+$ to an arbitrary root $z_{k+1},\ldots,z_n$ is larger than $2^{2\log n+20}\cdot N_i\cdot r_i$. Hence, the second requirement from Lemma~\ref{newtonsucceeds} is fulfilled
	for each component $C_i$ with $i\ge i_1$. Now, suppose that $2^{\ell_{s}}\cdot 2^{3\log n+32}\cdot N_{i}< 2^{\ell_{i}}$, then the roots $z_1$ to $z_k$ are contained in a disk of radius $\frac{3}{2}\cdot 9n\cdot 2^{\ell_s}<2^{-20-\log n}\cdot N_i^{-1}\cdot r_i$, and thus also the
	first requirement from Lemma~\ref{newtonsucceeds} is fulfilled. Hence, from the definition of $i_2$, we conclude that
	the algorithm performs a quadratic step if and only if $2^{\ell_{s}}\cdot 2^{3\log n+32}\cdot N_{i}< 2^{\ell_{i}}$. We now define the sequence $s_i:=(2^{\ell_i},\log\log N_i)$, where $i$
	runs from $i_1$ to the first index, denoted $i_1'$, for which
	$2^{\ell_{i_1'}}<2^{\ell_{s}}\cdot 2^{-3\log n-32}$. Then, according to Lemma~\ref{sequence},
	it holds that $i_1'-i_1\le 8(m+\log\log\max(4,\frac{w}{w'}))$, with $w:=2^{\ell_{i_1}}$,
	$m:=\log\log N_{i_1}$, and $w':=2^{\ell_{s}}\cdot 2^{3\log n+32}$. Theorem~\ref{thm:termination} (g) yields that $m=O(\log\LOG(w(\mathcal{B}))+\log\LOG(\sigma_F(2\mathcal{B})^{-1})$. Hence, since
	$i_2-i_1'\le 3\log n+32$, we conclude that $i_2-i_1\le O(\log n+\log\LOG(w(\mathcal{B}))+\log\LOG (\sigma_F(2\mathcal{B})^{-1}))$.

	It remains to show that the latter bound also applies to $i_1$. From the upper bound on the numbers $N_i$, it follows the existence of an $m_{\max}$ of size $O(\log n+\log\LOG(w(\mathcal{B}))+\log\LOG(\sigma_F(2\mathcal{B})^{-1}))$ such that each sequence of consecutive quadratic steps has length less than $m_{\max}$, and such that after $m_{\max}$ consecutive linear steps, the number $N_i$ drops to $4$.
	Since the number $\ell_i$ decreases by at least $1$ in each step, there exists an index $i'$ of size $O(\log n)$ such that $2^{\ell_{i'}}\cdot 2^{3\log n+34}<2^{\ell_1}$.
	Now, if the sequence $C_{i'},C_{i'+1},\ldots$ starts with $m_{\max}$ or more consecutive
	linear steps, we must have $N_{i'+m_{\max}}=4$, and thus $2^{\ell_{i'+m_{\max}}}\cdot 2^{3\log n+32}N_{i'+m_{\max}}<2^{\ell_1}$. Hence, we conclude that $i_1\le i'+m_{\max}$ in this case. Otherwise, there must exist an index $i''$, with $i'\le i''<i'+m_{\max}$, such that the step from $i''$ to $i''+1$ is quadratic, whereas the step from $i''+2$ is linear. Then, it holds that
	\[
		N_{i'' +2 }
		=
		\sqrt{N_{i''+1}}
		=
		N_{i''}
		\; \text{ and } \;
		2^{\ell_{i''+2}}
		\le
		2^{\ell_{i''+1}}
		=
		\frac{2^{\ell_{i''}}}{2N_{i''}}
		<
		2^{-3\log n-32}
		\frac{2^{\ell_1}}{N_{i''}},
	\]
	which implies that $i_1\le i''+2\le i'+m_{\max}+1=O(\log n+\log\LOG(w(\mathcal{B}))+\log\LOG(\sigma_F(2\mathcal{B})^{-1}))$. Hence, the claimed bound on $i_1$ follows.
\end{proof}
We can now state the first main result of this section, which immediately follows from the above bound on $s_{\max}$ and the fact that there exists at most $2\cdot(|\MMM(2\mathcal{B})|+1)$ special nodes:

\begin{theorem}\label{thm:treesize}
	The subdivision tree $\mathcal{T}$ induced by $\CC$\textsc{Isolate} has size
	\begin{align}\label{equ:treesize}
		|\mathcal{T}| & \le 2\cdot(|\MMM(2\mathcal{B})|+1)\cdot s_{\max}                                                                                \\ \nonumber
		              & = O\left(|\MMM(2\mathcal{B})|\cdot \log \left(n\cdot \LOG(w(\mathcal{B}))\cdot \LOG(\sigma_F(2\mathcal{B})^{-1})\right)\right).
	\end{align}
	If $\mathcal{B}$ contains all complex roots of $F$, and if $\LOG (w(\mathcal{B}))=O(\Gamma_F+\log n)$,\footnote{Notice that we can compute such a square $\mathcal{B}$ with $\tilde{O}(n^2\Gamma_F)$ bit operations; see Section~\ref{sec:definitions}.} then the above bound writes as
	\begin{align}\label{equ:treesizeglobal}
		O\left(n\cdot \log \left(n\cdot \Gamma_F\cdot \LOG(\sigma_F^{-1})\right)\right).
	\end{align}
\end{theorem}

We can also give simpler bounds for the special case, where our input polynomial has integer coefficients. Suppose that $f(x)\in\mathbb{Z}[x]$ has integer coefficients of bit size less than $\tau$. We first divide $f$ by its leading coefficient $\operatorname{lcf}(f)$ to obtain the polynomial $F:=f/\operatorname{lcf}(f)$, which meets our requirement from (\ref{def:polyF}) on the leading coefficient. Then, we have $\Gamma_F=O(\tau)$ and $\sigma_F=2^{-O(n(\log n+\tau))}$; e.g.~see~\cite{yap-fundamental} for a proof of the latter bound. Hence, we obtain the following result:

\begin{corollary}\label{cor:treesize}
	Let $f$ be a polynomial of degree $n$ with integer coefficients of bit size less than $\tau$, let $F:=f/\operatorname{lcf}(f)$, and let $\mathcal{B}$ be a square of width $2^{O(\Gamma_F+\log n)}$. Then, the algorithm $\mathbb{C}\textsc{Isolate}$ (with input $F$ and $\mathcal{B}$) uses
	\[
		O(|\MMM(2\mathcal{B})|\cdot \log(n\tau))=O(n\cdot\log(n\tau))
	\]
	iterations to isolate all roots of $F$ that are contained in $\mathcal{B}$.
\end{corollary}

The above results show that our algorithm performs near-optimal with respect to
the number of components that are produced by the algorithm. In addition, since each component
consists of at most $9n$ squares, we immediately obtain an upper bound for the
total number of squares produced by the algorithm that exceeds the bound from
(\ref{equ:treesizeglobal}) by a factor of $n$. Indeed, we will see that the
actual number of squares is considerably smaller, that is, of size
$O(|\MMM(2\mathcal{B})|\cdot s_{\max}\cdot\log n)$, which exceeds the bound in
(\ref{equ:treesizeglobal}) only by a factor $\log n$. For the proof, we
consider two mappings $\phi$ and $\psi$, where $\phi$ maps a component
$C=\{B_1,\ldots,B_{s_C}\}$ to a root $z_i\in C^+$, and $\psi$ maps a square $B_j$
to a root $z_i\in 4B_j\cap C^+$. The claimed bound for the total number of squares then follows from the fact that we can define $\phi$ and $\psi$ in a way such that the pre-image of an arbitrary root
$z_i\in2\mathcal{B}$ (under each of the two mappings) has size $O(s_{\max}\cdot\log n)$. The rest of this section is dedicated to the definitions of $\phi$ and $\psi$ and the proof of the latter claim. In what follows, we may assume that $2\mathcal{B}$ contains at least one root as, otherwise, all four sub-squares of $\mathcal{B}$ are already discarded in the first iteration of the preprocessing phase.

\begin{definition}
	For a root $\xi\in2\mathcal{B}$, we define the \emph{canonical path} $\P_{\xi}$ of $\xi$ as the unique path in the subdivision tree $\mathcal{T}_{\mathcal{B}}$ that consists of all nodes $C$ with $\xi\in C^+$.
\end{definition}

Notice that the canonical path is well-defined as, for any two nodes $C_1$ and $C_2$, either $C_1^+$ and $C_2^+$ are disjoint or one of the two components contains the other one. We can now define the maps $\phi$ and $\psi$:

\begin{definition}[Maps $\phi,\psi$]\label{def:maps}
	Let $C=\{B_1,\ldots,B_{s_C}\}$ be a node in the subdivision tree $\mathcal{T}_{\mathcal{B}}$, and let $B:=B_j$ be an arbitrary square in $C$. Then, we define maps $\phi$ and $\psi$ as follows:
	\begin{itemize}
		\item[($\phi$)] Starting at $C$, we descend in the subdivision tree as follows: If the current node $D$ is a non-terminal special node, we go to the child $E$ that minimizes $|\MMM(E^+)|$. If $D$ is terminal, we stop. If $D$ is non-special, then there is a unique child of $D$ to proceed with. Proceeding this way, the number $|\MMM(D^+)|$ is at least halved in each non-terminal special node $D$, except for the base node. Hence, since any sequence of consecutive non-special nodes has length at most $s_{\max}$, it follows that after at most $s_{\max}\cdot(\log\lceil(|\MMM(C^+)|)\rceil+1)\le s_{\max}\cdot(\log n+2)$ many steps we reach a terminal node $F$. We define $\phi(C)$ to be an arbitrary root contained in $\MMM(F^+)$.
		\item[($\psi$)] According to part (d) of Theorem~\ref{thm:termination}, the enlarged square $2B$ contains at least one root $\xi$. Now, consider the unique maximal subpath $P'_{\xi}=C_1,C_2,\ldots,C_s$ of the canonical path $P_{\xi}$ that starts at $C_1:=C$. If $s\le\lceil\log (18n)\rceil$, we define $\psi(B):=\xi$.
		      Otherwise, consider the component $C':=C_{\lceil\log(18n)\rceil}$ and define $\psi(B):=\phi(C')$.
	\end{itemize}
\end{definition}
It is clear from the above definition that $\phi(C)$ is contained in $C^+$ as each root contained in the enlarged component $F^+$ corresponding to the terminal node $F$ is also contained in $C^+$. It remains to show that $\psi(B)\in 4B\cap C^+$. If the length of the sub-path $P_{\xi}'$ is $\lceil\log(18n)\rceil$ or less, then $\psi(B)=\xi\in 2B$, hence, there is nothing to prove. Otherwise, the squares in $C'$ have width less than $\frac{w(B)}{18n}$. Since $C'$ can contain at most $9n$ squares, we conclude that $w(C')< \frac{w(B)}{2}$, and since $\xi$ is contained in $B^+$ as well as in $(C')^+$, we conclude that $(C')^+\subset 4B$, and thus $\psi(B)=\phi(C')\in 4B\cap (C')^+\subset 4B\cap C^+$.

Now, consider the canonical path $P_{\xi}=C_1,\ldots,C_s$, with $C_1:=\mathcal{B}$, of an arbitrary root $\xi\in2\mathcal{B}$.
Then, a component $C$ can only map to $\xi$ via $\phi$ if $C=C_i$ for some $i$
with $s-i\le s_{\max}\cdot (\log |\MMM(2\mathcal{B})|+1)$. Hence, the pre-
image of $\xi$ has size $O(s_{\max}\cdot\LOG |\MMM(2\mathcal{B})|)$.
For the map $\psi$, notice that a square $B$ can only map to $\xi$ if $B$ is
contained in a component $C=C_i$ for some $i$ with $s-i=s_{\max}\cdot(\log |\MMM(2\mathcal{B})|+1)+\lceil\log(18n)\rceil$. Since, for each component $C_{i}$, there exist
at most a constant number of squares $B'\in C_i$ with $\xi\in 4B'$, we
conclude that the pre-image of $\xi$ under $\psi$ is also of size $O(s_{\max}\cdot\LOG \MMM(2\mathcal{B}))$. Hence, the total number of squares
produced by our algorithm is bounded by $O(|\MMM(2\mathcal{B})|\cdot s_{\max}\cdot \LOG |\MMM(2\mathcal{B})|)$. We summarize:

\begin{theorem}
	\label{mapping}
	Let $\xi\in 2\mathcal{B}$ be a root of $F$ contained in the enlarged square $2\mathcal{B}$. Then, with mappings $\phi$ and $\psi$ as defined in Definition~\ref{def:maps}, the pre-image of $\xi$ under each of the two mappings has size $O(s_{\max}\cdot\LOG |\MMM(2\mathcal{B})|)$. The total number of squares produced by the algorithm $\mathbb{C}\textsc{Isolate}$ is bounded by
	\[
		O(s_{\max}\cdot |\MMM(2\mathcal{B})|\cdot \LOG |\MMM(2\mathcal{B})|)=\tilde{O}\left(n\cdot\log\left(\LOG(w(\mathcal{B}))\cdot\LOG(\sigma_F(2\mathcal{B})^{-1})\right)\right)
	\]
\end{theorem}

We can also state a corresponding result for polynomials with integer coefficients:

\begin{corollary}
	Let $f\in\mathbb{Z}[x]$ and $F:=f/\operatorname{lcf}(f)$ be polynomials and $\mathcal{B}$ be a square as in Corollary~\ref{cor:treesize}. Then, for isolating all roots of $f$ contained in $\mathcal{B}$, the algorithm $\mathbb{C}\textsc{Isolate}$ (with input $F$ and $\BBB$) produces a number of squares bounded by
	\[
		O(|\MMM(2\mathcal{B})|\cdot \log(n\tau)\cdot \log n)=O(n\log^2(n\tau)).
	\]
\end{corollary}

\subsection{Bit Complexity}\label{subsec:bitcomplexity}

\providecommand{\boxfun}{\eta}
\providecommand{\compfun}{\chi}
\providecommand{\Tcent}{\mathcal{T}_{\operatorname{wcent}}}

For our analysis, we need to introduce the notion of a square or component being \emph{weakly centered} and \emph{centered}:
\begin{definition}
	We say that a square $B$ of width $w:=w(B)$ is \emph{weakly centered}, if $\min\{|z|:z\in B\} \le 4\cdot w$. Similarly, we say that a square is \emph{centered} if $\min\{|z|:z\in B\} \le w/4$. In addition, we define a (weakly) centered component to be a component that contains a (weakly) centered square.
\end{definition}
Notice that the child of a component that is \emph{not} weakly centered can never become weakly centered. Hence, it follows that the set of weakly centered components forms a subtree $\Tcent$ of the subdivision tree $\TTT$ that is either empty or contains the root component $\BBB$; see Figure~\ref{figure:centralpath} for an illustration.

Moreover, let $C$ and $C'$ be siblings in $\Tcent$
and let $w$ and $w'$ be the sizes of the boxes in the components $C$ and $C'$, respectively. We already argued that the distance between $C$ and $C'$ is at least $\max\{w,w'\}$. W.l.o.g.~let $\min\{|z|:z\in C\}\le \min\{|z|:z\in C'\}$, then the distance of $C'$ to the origin is at least $w'/2$, and thus $C'$ is not centered. We further conclude that a descendant of depth 3 of $C'$ is not weakly centered because the width of the boxes in the component is at least halved in each step. It follows that each path in $\Tcent$ consisting of only weakly centered components has length at most 3.
From this observation, we conclude that the subtree $\Tcent$ has a very special structure. Namely, it consists of only one (possibly empty) \emph{central path}
$P=C_1,\ldots, C_\ell$ of all centered components, to which some trees of depth at most 4 of weakly centered components are attached, see Figure~\ref{figure:centralpath}.
Since there can only be a constant number of disjoint not weakly centered squares of the same size, it further holds that the degree of each node is bounded by a constant. Hence, each of the attached trees has constant size, and each node in the tree contains at most a constant number of weakly centered squares.

Notice that not weakly centered components $C\in \mathcal{T}\setminus \Tcent$ have the crucial property that any two points in $C^+$ have absolute values of comparable size; see Lemma~\ref{lem:non-centered uniform}. This is not true in general for (weakly) centered components as the size of $C$ might be very large whereas the distance from $C^+$ to the origin is small.

\begin{lemma}\label{lem:non-centered uniform}
	If $B$ is a not weakly centered square, it holds that
	\begin{align*}
		\max_{z\in 4B}\LOG(z)\le 5+\min_{z\in 4B}\LOG z.
	\end{align*}
	Moreover, if $C$ is a not weakly centered component, then it holds that
	\begin{align*}
		\max_{z\in C^+}\LOG(z)\le \log(64n)+\min_{z\in C^+}\LOG z.
	\end{align*}
\end{lemma}
\begin{proof}
	We first prove the claim for a not weakly centered square $B$. Let $\underline z:=\argmin\{|z|:z\in 4B\}$ and $\overline z:=\argmax\{|z|:z\in 4B\}$. By definition the distance of $B$ to the origin is at least $4w$, where $w$ denotes the size of $B$. Moreover, with $x:=\argmin\{|z|:z\in B\}$, we get $
	|\underline z| \ge |x| - |x-\underline z| \ge 4w - \sqrt{12.5}w\ge w/4.
	$ Thus
	\begin{align*}
		|\overline z|-|\underline z|
		\le |\overline z-\underline z|
		\le 4\sqrt{2} w
		\le 16\sqrt{2}\cdot |\underline z|,
	\end{align*}
	and the statement follows.

	It remains to prove the claim for a not weakly centered component. Let $\underline z:=\argmin\{|z|:z\in C^+\}$, $\overline z:=\argmax\{|z|:z\in C^+\}$, and let $B\subset C$ be a square in $C$ such that $\underline z\in 4B$. Then, as above, it follows that $\underline z\ge w/4$ and since $C$ contains at most $9n$ squares it holds that
	\begin{align*}
		|\overline z|-|\underline z|
		\le |\overline z-\underline z|
		\le \sqrt{2} \cdot (9n w+w)
		\le \sqrt{2}\cdot 10nw
		\le 57n |\underline z|,
	\end{align*}
	and thus $\max_{z\in C^+}\LOG(z)\le \log(64n)+\min_{z\in C^+}\LOG z$.
\end{proof}

\begin{figure}[t]
	\begin{center}
		\resizebox{0.75\textwidth}{!}{

\tikzset{
    every node/.style={minimum size = 0.35cm},
    treenode/.style = {align=center, inner sep=0pt, text centered},
    black_node/.style = {treenode, circle, black, draw=black, fill=black},
    red_node/.style = {treenode, diamond, alizarin, draw=alizarin, fill=alizarin, minimum size = 0.9cm},
    red_node_example/.style = {treenode, diamond, white, draw=alizarin, fill=alizarin, minimum size = 0.9cm},
    red_split_node/.style = {treenode, rectangle, white, draw=alizarin, fill=alizarin, minimum size = 1.2cm},
    blue_node/.style ={treenode, circle, hanblue, draw=hanblue, fill=hanblue, minimum size = 0.7cm},
    blue_node_example/.style ={treenode, circle, white, draw=hanblue, fill=hanblue, minimum size = 0.7cm},
    emph/.style={edge from parent/.style={alizarin,line width=3.0pt,draw}},
    norm/.style={edge from parent/.style={gray,line width=1.5pt,draw}}
}

\begin{tikzpicture}[-,>=stealth',level/.style={sibling distance = 10cm/#1,
  level distance = 1.3cm}, font=\boldmath] 
\node [red_split_node] {\huge $C_{i_1}$}
    child[emph]{ node [red_split_node] {\huge $C_{i_2}$} 
        child[norm]{ node [blue_node] {10} 
            child[norm]{ node [blue_node] {5} 
                child[norm]{ node [black_node] {1}}
            } 
            child[norm]{ node [blue_node] {}}
        }
        child[emph]{ node [red_node] {20}
                child[emph]{ node [red_node] {18}
                    child[norm]{ node [black_node] {12}}
                    child[emph]{ node [red_split_node] {\huge $C_{i_3}$}
                        child[emph]{ node [red_node_example] {\Large $C'$}
                            child[norm]{ node [blue_node] {19}
                                child[norm]{ node [blue_node_example] {\Large $C$}}
                            }
                            child[emph]{ node [red_node] {28}
                                child[emph]{ node [red_split_node] {\huge $C_{i_4}$}
                                    child[norm]{ node [blue_node] {19}}
                                    child[emph]{ node [red_node] {19}
                                        child[norm]{ node [blue_node] {19}
                                            child[norm]{ node [blue_node] {19}}}
                                    }
                                }
                            }
                        }
                        child[norm]{ node [blue_node] {10}}
                        child[norm]{ node [blue_node] {10}
                            child[norm]{ node [blue_node] {59}}
                            child[norm]{ node [blue_node] {59}
                                child[norm]{ node [blue_node] {19}}
                                child[norm]{ node [blue_node] {}
                                    child[norm]{ node [blue_node] {19}}
                                        child[norm]{ node [black_node] {19}}
                                        child[norm]{ node [black_node] {19}
                                            child[norm]{ node [black_node] {19}}
                                            child[norm]{ node [black_node] {19}}
                                        }
                                    }
                            }
                        }
                    }
            }
                child[norm]{ node [blue_node] {}}
        }
    }
    child[norm]{ node [black_node] {47}
        child[norm]{ node [black_node] {38} 
            child[norm]{ node [black_node] {39}}
        }
        child[norm]{ node [black_node] {51}
            child[norm]{ node [black_node] {49}
                child[norm]{ node [black_node] {}
                    child[norm]{ node [black_node] {36}}
                    child[norm]{ node [black_node] {36}}
                }
            }
            child[norm]{ node [black_node] {}
                child[norm]{ node [black_node] {36}}
            }
        }
    }
; 
\end{tikzpicture}
		}
	\end{center}
	\caption{
		A possible subdivision tree of $\mathbb{C}\textsc{Isolate}$ on an arbitrary input box. Red nodes (squares and diamonds) correspond to centered components, blue nodes (large circles) to weakly centered components and black nodes (small circles) to not weakly centered components. The subtree $\Tcent$ consists of all red and blue nodes. Note that $\Tcent$ consists of one path of centered components (marked by the red edges) and trees of depth at most 4 attached to it (consisting of blue nodes).
		The rectangular red nodes correspond to those components $C_{i_1},\ldots, C_{i_s}$ on the central path that are split nodes according to Definition~\ref{def:splitting}, that is, components for which $\MMM(C_{i_j}^+)\supsetneq \MMM(C_{i_j+1}^+)$. The difference of these two sets contains the roots that we map the (weakly) centered components to.
		For instance, the centered box $C'$ is mapped to a root $\hat \phi(C)$ that is contained in $\MMM(C_{i_3+1}^+)\setminus \MMM(C_{i_3}^+)$. All weakly centered boxes that have $C'$ as their first centered predecessor, as for example $C$, are mapped to the same root $\phi(C)$.
	}
	\label{figure:centralpath}
\end{figure}
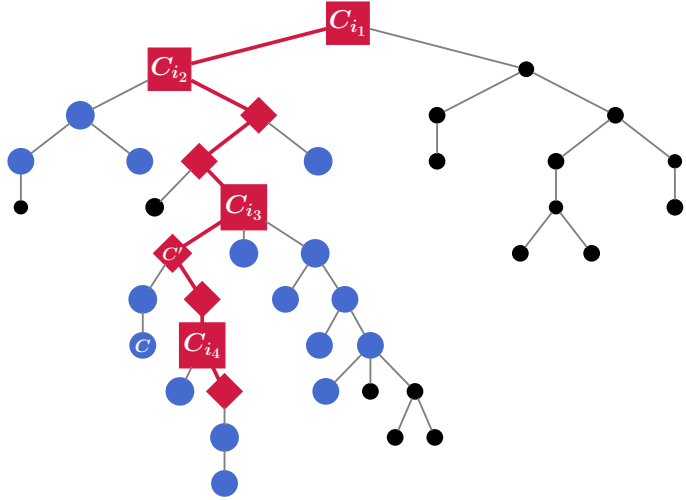

In the previous section, we introduced mappings $\phi$ and $\psi$ that map components $C$ and squares $B$ to roots contained in $C^+$ and $4B\cap C^+$, respectively, such that the preimage (under each of the two mappings) of each root has size at most $O(s_{\max}\cdot\LOG |\MMM(2\BBB)|)=O(s_{\max}\cdot\log n)$, with $$s_{\max}=O(\log(n\LOG(w(\BBB))\LOG(\sigma_F(2\BBB)^{-1})))$$ as defined in Lemma~\ref{pathlength}.
The crucial idea in our analysis is to bound the cost for processing a certain component $C$ (square $B$) in terms of values that depend only on the root $\phi(C)$ ($\psi(B)$), such as its absolute value, its separation, or the absolute value of the derivative $F'$ at the root; see Lemma~\ref{lem:costTk}. Following this approach, each root in $2\BBB$ is ``charged'' only a small (i.e.~logarithmic in the ``common'' parameters) number of times, and thus  we can profit from amortization when summing the cost over all components (squares).
For each not weakly centered component $C$, the width of $C$ and the absolute value of any point in $C$ is upper bounded by $64n\cdot |\phi(C)|$, which allows us to bound each occurring term $\LOG w(C)$ by $O(\log n+\LOG |\phi(C)|)$. However, for weakly centered components (squares), this does not hold in general, and thus some extra treatment is required.
For this, we will introduce slightly modified mappings $\hat{\phi}$ and $\hat{\psi}$ that coincide with $\phi$ and $\psi$ on all not weakly centered components and squares, respectively, but map a weakly centered component (square) to a root of absolute value that is comparable to the size of the component (square). In the next step, we will show that this can be done in a way such that the pre-image of each root is still of logarithmic size. We give details:

Let $i_1,\ldots,i_s$, with $1\le i_1<i_2 < \ldots < i_s\le \ell-1$, be the indices of components in the central path $P=C_1,\ldots, C_\ell$ such that $\MMM(C_{i_j}^+)\supsetneq \MMM(C_{i_j+1}^+)$ for $j=1,\ldots,s$. That is, $C_{i_j}$ are the weakly centered components that are also special according to Definition~\ref{def:splitting}.
In addition, we say that $z_0:=\max\{|x|:x\in 2\BBB\}$ is the \emph{pseudo-root of $2\BBB$}, and define $\MMM^+(2\BBB):=\MMM(2\BBB)\cup\{z_0\}$ as the set consisting of all (pseudo-) roots in $2\BBB$.

\begin{definition}[Maps $\hat\phi$ and $\hat\psi$]\label{def:canonicalmap}
	Let $C$ be a component and $B\subset C$ be a square contained in $C$.
	\begin{enumerate}
		\item If the component $C$ is not weakly centered, we define $\hat{\phi}(C):=\phi(C)$.
		\item If the component $C$ is weakly centered, let $C_i\in P$ be the first centered predecessor of $C$ in $\Tcent$. If there exists no such $C_i$ or if $i\in[1,i_1]$, we define $\hat{\phi}(C)=z_0$.
		      For $i\in(i_2,\ell]$, let $j$ be maximal with $i_j< i$. Then, there exists a root $\xi\in \MMM(C^+_{i_j})\setminus \MMM(C^+_{i_j+1})$. We define $\hat{\phi}(C):=\xi$.
		\item If $B$ is weakly centered, we define $\hat{\psi}(B):=\hat{\phi}(C)$. Otherwise, we define $\hat\psi(B):=\psi(B)$.
	\end{enumerate}
\end{definition}

We derive the first crucial property of the mappings $\hat\phi$ and $\hat\psi$:

\begin{lemma}
	It holds that
	\begin{align}\label{property:maps}
		\max_{z\in C^+}\LOG z\le \log(64n)+\LOG\hat{\phi}(C) \;\text{ and }\;
		\max_{z\in 4B}\LOG z\le \log(64n)+\LOG\hat{\psi}(B)
	\end{align}
	for all
	components $C$ (squares $B$).
\end{lemma}
\begin{proof}
	For a not weakly centered component $C$ (square $B\subset C'$), this follows directly from Lemma~\ref{lem:non-centered uniform} and the fact that $\hat\phi(C)=\phi(C)$ ($\hat\psi(B)=\psi(B)$) and that $\phi(C)\in C^+$ ($\psi(B)\in 4B\cap C^+$).

	For a centered component $C$, we either have $\hat{\phi}(C)=z_0$ or $\hat\phi(C)\notin C^+$ due to the definition of $\hat\phi$. The first case is trivial, hence, we may assume that $\hat\phi(C)\notin C^+$. Since $C$ is centered, it contains a centered square $B$, and thus the distance of $B$ to the origin is at most $w/4$. It follows that $C^+$ contains the disk of radius $w/4$ around the origin, hence $|\hat\phi(C)|\ge w/4$. Since the distance between any two points in $C^+$ is upper bounded by $10\sqrt{2}nw$, we conclude that $|z|\le w/4+10\sqrt{2}nw\le (1+40\sqrt{2})\cdot |\hat\phi(C)|\le 64n\cdot |\hat\phi(C)|$ for all $z\in C^+$. The same argument further shows that $|z|\le 64n\cdot |\hat\psi(B)|$ for all centered squares and all $z\in 4B$.

	It remains to show the claim for a weakly centered component $C$ (square $B$) that is not centered. In this case, we either have $\hat{\psi}(C)=z_0$ or $\hat{\psi}(C)=\hat{\phi}(C')$, where $C'$ is a centered component on the central path that contains $C$; see Definition~\ref{def:canonicalmap}. In the first case, there is nothing to prove. In the second case, we have have already shown that Inequality (\ref{property:maps}) holds for $C'$, hence, it must hold for $C$ as well. The same argument also applies to squares $B$ in $C$ that are weakly centered but not centered.
\end{proof}

Notice that, for a centered component $C$ (square $B$), the image of the corresponding mapping $\hat{\phi}$ ($\hat{\psi}$) may no longer be contained in the enlarged component $C^+$ (enlarged square $4B$), as it is the case for the mappings $\phi$ and $\psi$.
However, it still holds that the preimage of its (pseudo-) root under each of the two mappings $\hat\phi$ and $\hat\psi$ is of small size:

\begin{lemma}\label{lem:canonicalhit}
	Let $\xi\in\MMM^+(2\BBB)$. Then, the preimage of $\xi$ under $\hat{\phi}$ and $\hat{\psi}$ has size at most $O(s_{\max}\cdot\LOG |\MMM(2\BBB)|)=O(s_{\max}\cdot\log n)$.
\end{lemma}

\begin{proof}
	Since $\hat{\phi}$ coincides with $\phi$ on all components $C\notin \Tcent$, it suffices to show the claim for the restriction $\hat{\phi}|_{\Tcent}$ of $\hat\phi$ to the components $C\in\Tcent$ that are weakly centered.
	Let $\xi=\hat{\psi}(C)$, with $C\in\Tcent$, be an arbitrary root contained in the image of $\hat{\phi}|_{\Tcent}$, and let $C_i$ be the first predecessor of $C$ that is central and located on the central path.
	Then, there exists a $j\in\{1,\ldots,\ell-1\}$ with $i_j<i\le i_{j+1}$, and we have $\xi\in\mathcal{Z}(C_{i_j}^+)\setminus \mathcal{Z}(C_{i_{j}+1}^+)$. In addition, $C$ is connected with $C_i$ via a path of constant length as the distance from $C$ to the central path on $\Tcent$ is bounded by a constant and there cannot be more than $3$ consecutive components that are weakly centered but not centered.
	Since there exists at most $s_{\max}$ components on the central path between $C_{i_j}$ and $C_{i_{j+1}}$, it follows that the number of components $C\in\Tcent$ that are mapped to $\xi$ is bounded by $O(s_{\max})$. The same argument applies to the special case, where $C$ is mapped to the pseudo-root $z_0$. Also, from the same argument and the definition of $\hat\psi$, it further follows that there can be at most $O(s_{\max})$ many weakly centered squares that are mapped to the same (pseudo-) root, since each component contains at most constantly many weakly centered squares.
\end{proof}

We can now start with the bit complexity analysis of \cisolate. Let $C=\{B_1,\ldots,B_{s_{C}}\}$ be any component produced by the algorithm. When processing $C$, our algorithm calls the $\T$-test in up to three steps. More specifically, in line~\ref{testbeforenewton} of \cisolate\ the $\T(2\Delta_C)$ and the $\T(4\Delta_C)$-test are called.
In the \textsc{NewtonTest}, the $\T(\Delta')$-test is called, with $\Delta'$ as defined in line~\ref{deltaprime} in \textsc{NewtonTest}. Finally, in \textsc{Bisection}, the
$\T(\Delta_{B'})$-test is called for each of the $4$ sub-squares $B'$ into which each square $B_i$ of $C$ is decomposed. Our goal is to provide bounds for the cost of each of these calls.
For this, we mainly use Lemma~\ref{allkcost}, which provides a bound on the cost for calling $\T(\Delta)$ that depends on the degree of $F$, the value $\tau_F$, the size of the radius and the center of $\Delta$, and the maximal absolute value that $F(x)$ takes on the disk $\Delta$. Under the assumption that $\Delta$ has non-empty intersection with $C$, we may reformulate the latter value in terms of parameters (such as the absolute value, the separation, etc.) that depend on an arbitrary root contained in $C^+$.

\begin{lemma}\label{Fmbighelp}
	Let $C$ be a component,
	and let $\Delta:=\Delta(m,r)$ be a disk that has non-empty intersection with $C$.
	If $\MMM(C^+)\ge 1$, then it holds that $\sigma_F(z_i)<n\cdot 2^{\ell_C+6}$ and
	\[
		\max_{z\in\Delta}|F(z)|>2^{-16n}\cdot\sigma_F(z_i)\cdot|F'(z_i)|\cdot (|\MMM(C^+)|\cdot\max\nolimits_1\lambda)^{-2n},
	\]
	where $z_i$ is an arbitrary root of $F$ contained in $C^+$ and $\lambda:=\frac{2^{\ell_C}}{r}$ the ratio of the size of a square in $C$ and the radius of $\Delta$.
	If $\MMM(C^+)=0$, then $C$ is the component consisting of the single input square $\BBB$, and it holds that $\max_{z\in\Delta}|F(z)|>(2w(\BBB))^{-n-2}$.
\end{lemma}

\begin{proof}
	Notice that $|\MMM(C^+)|=0$ is only possible if $C=\mathcal{B}$ as $C^+$ contains at least one root for each component $C$ that is not equal to the single input square $\BBB$. Hence, each point $z\in\Delta\cap C$ has distance at least $w(B)/2$ to each of the roots of $F$, and thus $|F(z)|\ge (2w(\BBB))^{-n-2}$.

	In what follows, we now assume that $C^+$ contains at least one root $\xi$. Let us first bound the separation of $\xi$:
	If there exists another root $\xi'\in C^+$, then we must have $\sigma_F(\xi)\le |\xi-\xi'|<(9\cdot|\MMM(C^+)|+1)\cdot \frac{3}{2}\cdot 2^{\ell_C}<n\cdot 2^{\ell_C+5}$, where we used that each component $C$ consists of at most $9\cdot|\MMM(C^+)|$ squares, each of size $2^{\ell_C}$.
	Now, let $|\MMM(C^+)|=1$, and let $C'$ be the direct ancestor of $C$. When processing $C'$, the \textsc{NewtonTest} failed as its success would imply that $k:=|\MMM((C')^+)|=|\MMM(C^+)|\ge 2$. In addition, since $C'$ is non-terminal, the disk $8\Delta_{C'}$ must contain at least two roots as otherwise $\mathbf{T_*}(2\Delta_{C'})$ as well as $\mathbf{T_*}(4\Delta_{C'})$ would return $1$. Hence, we have $\sigma_F(\xi)<n\cdot 2^{\ell_{C'}+5}=n\cdot 2^{\ell_C+6}$. We conclude that, in any case, any root $\xi\in C^+$ has separation $\sigma_F(\xi)<n\cdot 2^{\ell_C+6}$.

	In the next step, we show that there exists an $m'\in\Delta\cap C^+$ whose distance to $C$ is at most $2^{\ell_C-2}$ and whose distance to any root of $F$ is at least $\frac{\min(2^{\ell_C-2},r)}{2\sqrt{n}}$. Namely, due to our assumption, there exists a point $p\in\Delta\cap C$. Then, the two disks $\Delta(p,2^{\ell_C-2})$ and $\Delta$ share an area of size larger than $\min(2^{\ell_C-2},r)^2$, and thus there must exist an $m'\in\Delta(p,2^{\ell_C-2})\cap\Delta\subset C^+$ whose distance to any root of $F$ is lower bounded by $\left(\frac{\min(2^{\ell_C-2},r)^2}{\pi\cdot n}\right)^{1/2}$.

	Now, let $z_i\in C^+$ be an arbitrary but fixed root of $F$. If $z_j$ is a root not contained in $C^+$, then $|m'-z_j|>2^{\ell_C-2}$, and
	\begin{align*}
		\frac{|z_i-z_j|}{|m'-z_j|}
		  & \le
		\frac{|z_i-m'|+|m'-z_j|}{|m'-z_j|}
		\le
		1+\frac{|z_i-m'|}{|m'-z_j|}<1+\frac{9\cdot n\cdot 2^{\ell_C+1}}{2^{\ell_{C}-2}}\\
		  & = 1+72\cdot |\MMM(C^+)|<2^7\cdot n.
	\end{align*}

	If $z_j$ is a root in $C^+$, then
	$$\frac{|z_i-z_j|}{|m'-z_j|}\le\frac{9 n 2^{\ell_C+1}}{\delta}= \frac{18\cdot n^{3/2}\cdot 2^{\ell_C+1}}{\min(2^{\ell_C-2},r)}\le 2^8\cdot n^{3/2}\cdot \max\nolimits_1\lambda.$$
	Hence, we get
	\begin{align*}
		|F(m')|
		  & =
		|F_n|\cdot\prod_{j=1}^{n}|m'-z_j|
		=
		|F'(z_i)|\cdot|m'-z_i|\cdot\prod_{j\neq i}\frac{|m'-z_j|}{|z_j-z_i|}\\
		  & \ge
		|F'(z_i)|\cdot|m'-z_i|\cdot (2^8 n^{3/2}\max\nolimits_1\lambda)^{-|\MMM(C^+)|}\cdot (2^7 n)^{-(n-\MMM(C^+))}\\
		  & >
		|F'(z_i)|\cdot \frac{\min(2^{\ell_C-2},r)}{2\sqrt{n}}\cdot (2^8n^{3/2}\max\nolimits_1\lambda)^{-n}\\
		  & >
		|F'(z_i)|\cdot\frac{\sigma_F(z_i)}{2^9n^{3/2}\max\nolimits_1\lambda}\cdot (2^8n^{3/2}\max\nolimits_1\lambda)^{-n}\\
		  & >
		2^{-16n}\cdot\sigma_F(z_i)\cdot|F'(z_i)|\cdot (n\cdot\max\nolimits_1\lambda)^{-2n},
	\end{align*}
	where, in the second to last inequality, we used that $\sigma_F(z_i)<2^{\ell_{C}+6}\cdot n$.
\end{proof}

We can now bound the cost for processing a component $C$:
\begin{lemma}\label{lem:costTk}
	When processing a component $C=\{B_1,\ldots,B_{s_C}\}$ with $|\MMM(C^+)|\ge 1$, the cost for all steps outside the \textsc{NewtonTest} are bounded by
	\begin{align}\nonumber
		  & \tilde{O}(n\cdot(n\LOG\hat{\phi}(C)+\LOG(\sigma_F(\phi(C))^{-1})+\LOG(F'(\phi(C))^{-1})))\text{ }+                         \\
		  & \tilde{O}(n\cdot(\sum_{i=1}^{s_C} n\LOG\hat{\psi}(B_{i})+\tau_F+\LOG\sigma_F(\psi(B_{i}))^{-1}+\LOG F'(\psi(B_{i}))^{-1}))
		\label{cost:component:nonewton}
	\end{align}
	bit operations. The cost for the \textsc{NewtonTest} is bounded by
	\begin{align}\label{cost:component:newton}
		\tilde{O}(n( & n\LOG\hat{\phi}(C)+\LOG \sigma_F(\phi(C))^{-1}+\LOG F'(\phi(C))^{-1}+|\MMM(C^+)|\cdot\log N_C))
	\end{align}
	bit operations.
	If $C^+$ contains no root, then $C$ is the component consisting of the single input square $\BBB$, and the total cost for processing $C$ is bounded by $\tilde{O}(n (\tau_F+n\LOG(w(\BBB),w(\BBB)^{-1})))$ bit operations.
\end{lemma}
\begin{proof}
	We start with the special case, where $C^+$ contains no root. Notice that this is only possible if $C=\mathcal{B}$ and $2\mathcal{B}$ contains no root. Hence, in this case, the algorithm performs four $\T(\Delta)$-tests in the preprocessing phase and then discards $\mathcal{B}$. Due to Lemma~\ref{allkcost} and Lemma~\ref{Fmbighelp}, the cost for each of these tests is bounded by $\tilde O(n(\tau_F + n\LOG(w(\BBB),w(\BBB)^{-1}))$ bit operations.
	Hence, in what follows, we may assume that $C^+$ contains at least one root.
	We first estimate the cost for calling $\T$ on a disk $\Delta=\Delta(m,r)$, where $\Delta=2\Delta_C$, $\Delta=4\Delta_C$, or $\Delta=\Delta_{B'}$, where $B'$ is one of the four sub-squares into which a square $B_i$ is decomposed. If $\Delta=2\Delta_C$ or $\Delta=4\Delta_C$, then $\LOG(m,r)=O(\log n+\LOG \hat{\phi}(C))$, and thus the cost for the corresponding tests is bounded by
	\[
		\tilde{O}(n\cdot (n\LOG\hat{\phi}(C)+\tau_F+\LOG \sigma_F(\phi(C))^{-1}+\LOG F'(\phi(C))^{-1}))
	\]
	bit operations, where we again use Lemma~\ref{allkcost} and Lemma~\ref{Fmbighelp}. For $\Delta=\Delta_{B'}$, we have $\LOG(m,r)=O(\log n+\LOG \hat{\psi}(C))$, and thus Lemma~\ref{allkcost} and Lemma~\ref{Fmbighelp} yields the bound
	$$\tilde{O}(n\cdot (n\LOG\hat{\psi}(B_{i})+\tau_F+\LOG\sigma_F(\psi(B_{i}))^{-1}+\LOG F'(\psi(B_{i}))^{-1}))$$
	for processing each of the four sub-squares $B'\subset B_i$ into which $B_i$ is decomposed. Hence, the bound in (\ref{cost:component:nonewton}) follows.
	We now consider the \textsc{NewtonTest}: In line~\ref{resultless} of the \textsc{NewtonTest}, we have to compute an approximation $\tilde{x}_C'$ of the Newton iterate $x_C'$ such that $|\tilde{x}_C'-x'_C|<\frac{1}{64}\cdot \frac{2^{\ell_C}}{N_C}$. For this, we choose a point $x_C\in\mathcal{B}\backslash C$ in line~\ref{ChoosePoint} of \cisolate\, whose distance to $C$ is $2^{\ell_C-1}$ and whose distance to the boundary of $\mathcal{B}$ is at least $2^{\ell_C-1}$. Since the union of all components covers all roots of $F$ that are contained in $\mathcal{B}$, and since the distance from $C$ to any other component is at least $2^{\ell_C}$, it follows that the distance from $x_C$ to any root of $F$ is larger than $2^{\ell_C-1}$.
	With $\Delta:=\Delta(x_C,2^{\ell_C-3})$, it thus follows that $|F(x_C)|\ge 2^{-n}\cdot\max_{z\in\Delta}|F(z)|$, and using
	Lemma~\ref{Fmbighelp}, we conclude that
	\begin{align}\label{boundF(xC)}
		\LOG F(x_C)^{-1}
		  & =O(n\log n+\LOG \sigma_F(\phi(C))^{-1}+\LOG F'(\phi(C))^{-1}).
	\end{align}
	It follows that the cost for calling Algorithm~\ref{algo:softpredicate} in Line~\ref{softcallinnewton} of the \textsc{NewtonTest} is bounded by
	\begin{align*}
		  & \tilde{O}(n\cdot (\LOG F(x_C)^{-1}+\tau_F+n\LOG x_C)) \\ &\hspace{2cm}=
		\tilde{O}(n\cdot (\tau_F+\LOG \sigma_F(\phi(C))^{-1}+\LOG F'(\phi(C))^{-1}+n\LOG \hat{\phi}(C)))
	\end{align*}
	bit operations,
	Namely, Algorithm~\ref{algo:softpredicate} succeeds with an absolute precision bounded by $O(\LOG |F(x_C)|^{-1})$ and, within Algorithm~\ref{algo:softpredicate}, we need to approximately evaluate $F$ and $F'$ at the point $x_C$ to such a precision; see also~\cite[Lemma~3]{DBLP:journals/corr/abs-1104-1362} for the cost of evaluating a polynomial of degree $n$ to a certain precision.
	If we pass line~\ref{softcallinnewton}, then we must have $|F'(x_C)|>\frac{|F(x_C)|}{6r(C)}$. Hence, in the for-loop of the \textsc{NewtonTest}, we succeed for an $L$ of size $O(\LOG F(x_C)^{-1}+n\LOG w(C)-\ell_C +\log N_C)$, and thus the cost for computing $\tilde{x}_C'$ is bounded by
	\begin{align*}
		\tilde{O}(n\cdot (\LOG \sigma_F(\phi(C))^{-1}+\LOG F'(\phi(C))^{-1}+\LOG \hat{\phi}(C)+\log N_C)),
	\end{align*}
	where we again use (\ref{boundF(xC)}) and the fact that $\sigma_F(\phi(C))<n\cdot 2^{\ell_C+6}$ and $\LOG w(C)=O(\log n+\LOG \hat{\phi}(C))$.
	It remains to bound the cost for calling $\T$ on $\Delta':=\Delta(\tilde{x}_C',\frac{1}{8}\cdot\frac{2^{\ell_C}}{N_C})$ in the \textsc{NewtonTest}. Again, we use Lemma~\ref{allkcost} and Lemma~\ref{Fmbighelp} to derive an upper bound of size
	\[
		O(n\log n+\LOG \sigma_F(\phi(C))^{-1}+\LOG F'(\phi(C))^{-1}+|\MMM(C^+)|\cdot \log N_C)
	\]
	for $\LOG (\max_{z\in\Delta'}|F(z)|)^{-1}$, and thus a bit complexity bound of size
	\[
		\tilde{O}(n\cdot (\LOG\hat{\phi}(C)+\LOG \sigma_F(\phi(C))^{-1}+\LOG F'(\phi(C))^{-1}+|\MMM(C^+)|\cdot \log N_C))
	\]
	for $\T(\Delta')$. This proves correctness of the bound in (\ref{cost:component:newton}).
	We finally remark that the cost for all other (mainly combinatorial) steps are negligible. Namely, the bit size $b_C$ of a square in a component $C$ is bounded by $O(\LOG(w(C),w(C)^{-1})+\log n)=O(\log n+\LOG\sigma_F(\phi(C))^{-1}+\LOG\hat{\phi(C)})$. Hence, each combinatorial step outside the \textsc{NewtonTest}, such as grouping together squares into maximal connected components in Line~\ref{grouping} of \textsc{Bisection}, needs
	$\tilde{O}(n)$ arithmetic operations with a precision $O(b_C)$, and thus a number of bit operations bounded by (\ref{cost:component:nonewton}).

	In the \textsc{NewtonTest}, we need to determine the squares $B_{i,j}$ of size $2^{\ell_C-1}/N_C$ that intersect the disk $\Delta'$. This step requires only a constant number of additions and multiplication, each carried out with a precision bounded by $O(\LOG(w(C),w(C)^{-1})+\log n+\log N_C)$. Hence, the cost for these steps is bounded by (\ref{cost:component:newton}).\qedhere
\end{proof}

When summing up the bound in (\ref{cost:component:nonewton}) over all components $C$ produced by the algorithm, we obtain the bit complexity bound
\begin{align}\label{totalbound1}
	\begin{split}
	  & \tilde{O}\Big(n\cdot (n\cdot(|\MMM(2\BBB)| +\LOG\Mea_{F}(2\BBB)+\LOG(w(\BBB),w(\BBB)^{-1}))+\tau_F\cdot |\MMM(2\BBB)| \\
	  & \hspace{1cm} +\sum_{z_i\in \MMM(2\BBB)}(\LOG\sigma_F(z_i)^{-1}+\LOG F'(z_i)^{-1}))\Big),
	\end{split}
\end{align}
for all steps outside the \textsc{NewtonTest}. Here, we exploit the fact that the preimage of each (pseudo-) root in $\MMM^+(2\BBB)$ under each of the mappings $\phi,\hat{\phi},\psi,\hat{\psi}$ has size $O(s_{\max}\cdot \log n)$, and that $|z_0|>w(\BBB)/2$ for the pseudo-root $z_0\in\MMM^+(2\BBB)$. If we now sum up the bound (\ref{cost:component:newton}) for the cost of the \textsc{NewtonTest} over all components, we obtain a comparable complexity bound;
however, with an additional term $n\cdot \sum_C |\MMM(C^+)|\cdot\log N_C$, where the sum is only taken over the components for which the \textsc{NewtonTest} is called. The following considerations show that the latter sum is also dominated by the bound in (\ref{totalbound1}).

\begin{lemma}\label{bound:sumZNC}
	Let $\mathcal{T}_{\textsc{New}}\subset \mathcal{T}$ be the set of all components in the subdivision tree $\mathcal{T}$ for which the \textsc{NewtonTest} is called. Then, $\sum_{C\in\mathcal{T}_{\textsc{New}}}|\MMM(C^+)|\cdot\log N_C$ is bounded by
	\[
		\tilde{O}(n\cdot(\LOG w(\BBB)+\LOG\Mea_{F}(2\BBB)+|\MMM(2\BBB)|)+\tau_F\cdot |\MMM(2\BBB)|+\sum_{z_i\in\MMM(2\BBB)}\LOG F'(z_i)^{-1}).
	\]
\end{lemma}

\begin{proof}
	We define $\mathcal{T}_{\textsc{New}}^{=4}:=\{C\in\mathcal{T}_\textsc{New}:N_C=4\}$ and $\mathcal{T}_{\textsc{New}}^{>4}:=\{C\in\mathcal{T}_\textsc{New}:N_C>4\}$. Then, we have
	\[
		\sum_{C\in \mathcal{T}_{\textsc{New}}^{=4}}|\MMM(C^+)|\cdot\log N_C
		\le \sum_{C\in\mathcal{T}_{\textsc{New}}^{=4}}2n
		\le \sum_{C\in\mathcal{T}}2n
		\le 2n\cdot |\MMM(2\BBB)|\cdot s_{\max},
	\]
	hence it remains to consider only the components $C\in\mathcal{T}_{\textsc{New}}^{>4}$. For such a component, let $\anc^*(C)\in\mathcal{T}$ be the last ancestor for which the \textsc{NewtonTest} succeeded.
	According to Theorem~\ref{thm:termination} (f), we have $N_C\le 4\cdot(w(\anc^*(C))/w(C))^2$, and thus $|\MMM(C^+)|\cdot\log N_C\le 2\cdot|\MMM(C^+)|\cdot (1+\log w(\anc^*(C))-\log w(C))$. In order to bound the later expression in terms of values that depend on the roots of $F$, let $z_i\in C^+$ be an arbitrary root in $C^+$. Then, assuming $w(C)\le 1$, we obtain
	\begin{align*}
		|F'(z_i)|
		  & =|a_n|\cdot\prod_{j\neq i: z_j\in \MMM(2\Delta_C)}|z_i-z_j|\prod_{j:z_j\notin \MMM(2\Delta_C)}|z_i-z_j|   \\
		  & \le |a_n|\cdot (2 w(C))^{|\MMM(2\Delta_C)|-1}\cdot\frac{\Mea_{F(z_i-x)}}{|a_n|}                           \\
		  & \le 2^{2n+\tau_F} w(C)^{|\MMM(2\Delta_C)|-1}\cdot \MAX(z_i)^n                                           ,
	\end{align*}
	where the last inequality follows from Landau's inequality~\cite[Theorem 6.31]{vzGG03}, that is, $\Mea_p\le \|p\|_2$ for any complex polynomial $p\in\CC[x]$, and $\|F(z_i-x)\|_2 \le \|F(z_i-x)\|_1 \le  2^{\tau_F}\MAX(z_i)^n2^{n+1}$. For a more detailed derivation of the latter see~\cite[Lemma 22]{Sagraloff2015}. Furthermore, using that $\MMM(2\Delta_C)$ contains at least two roots (as the \textsc{NewtonTest} is called) and that
	$C^+\subset 2\Delta_C$, yields
	\begin{align*}
		|F'(z_i)|
		\le 2^{2n+\tau_F}\MAX(z_i)^n w(C)^{\frac{|\MMM(2\Delta_C)|}{2}}
		\le 2^{2n+\tau_F} \MAX(z_i)^n w(C)^{\frac{|\MMM(C^+)|}{2}}.
	\end{align*}
	With $z_i:=\phi(C)$ and $\LOG\phi(C)\le \log(64n)+\LOG\hat{\phi}(C)$, we conclude that
	\begin{align}\label{bound:ZNC}
		-|\MMM(C^+)|\cdot\log w(C)\le 6n\log(64n)+2\tau_F+2n\LOG\hat{\phi}(C)+\LOG F'(\phi(C))^{-1}.
	\end{align}
	Notice that the latter inequality is trivially also fulfilled for a component $C$ with $w(C)<1$.
	In addition, it holds that
	\begin{align}\label{bound:wanc}
		|\MMM(C^+)|\cdot \log w(\anc^*(C))\le|\MMM(C^+)|\cdot (\LOG \hat{\phi}(\anc^*(C))+\log(128n)).
	\end{align}
	The sum of the term at the right side of (\ref{bound:ZNC}) over all $C$ can be bounded by
	\begin{align*}
		  & \tilde{O}(s_{\max}\cdot (|\MMM(2\BBB)|\cdot\tau_F+\sum_{z_i\in\MMM^*(2\BBB)}n\LOG(z_i)+\sum_{z_i\in\MMM(2\BBB)}\LOG F'(z_i)^{-1}))       \\
		  & \quad= \tilde{O}(s_{\max}\cdot (|\MMM(2\BBB)|\cdot\tau_F+n\LOG w(\BBB)+n\LOG\Mea_{F}(2\BBB)+\sum_{z_i\in\MMM(2\BBB)}\LOG F'(z_i)^{-1})),
	\end{align*}
	as the preimage of each (pseudo-) root $z_i\in 2\BBB$ (under $\phi$ and $\hat{\phi}$) has size at most $s_{\max}\log n$.
	We may further omit the factor $s_{\max}$ in the above bound as $\log\LOG\sigma_F(z_i)^{-1}=O(\log(\tau_F+n\LOG(z_i)+\LOG F'(z_i)^{-1}))$ for an arbitrary root $z_i$ of $F$, and thus
	\begin{align*}
		  & s_{\max}=O(\log n+\log\LOG w(\BBB)+\log\LOG\sigma_{F}(2\BBB)^{-1})                                          \\
		  & =O(\log n+\log\LOG w(\BBB)+\log\tau_F+\log\LOG\Mea_F(2\BBB)+\log\sum_{z_i\in\MMM(2\BBB)}\LOG F'(z_i)^{-1}).
	\end{align*}

	It remains to bound the sum of the term $|\MMM(C^+)|\cdot\LOG\hat{\phi}(\anc^*(C))$ on the right side of (\ref{bound:wanc}) over all $C\in \mathcal{T}_{\textsc{New}}^{>4}$. Let $\mathcal{T}_{\anc^*}:=\{C^*\in\mathcal{T}:\exists C\in\mathcal{T}\text{ with }\anc^*(C)=C^*\}$.
	Then, for a fixed $C^*\in\mathcal{T}_{\anc^*}$, it holds that any $C\in\mathcal{T}_{\textsc{New}}^{>4}$ with $\anc^*(C)=C^*$ is
	connected with $C^*$ in $\mathcal{T}$ via a path of length at most $\hat{s}:=\log\LOG w(\BBB)+\log\LOG \sigma_F(2\BBB)^{-1}$. Namely, we have already shown that
	$\log\log N_C\le \hat{s}$ for all $C\in\mathcal{T}$, and thus $N_{C}>4$ implies that the path connecting $C$ and $C^*$ must have length at most $\hat{s}$.
	Hence, it follows that
	\begin{align*}
		  & \sum_{C\in\mathcal{T}_{\textsc{New}}^{>4}:\anc^*(C)=C^*} |\MMM((C)^+)|\cdot\log w(\anc^*(C))                    \\
		  & =\sum_{C\in\mathcal{T}_{\textsc{New}}^{>4}:\anc^*(C)=C^*} |\MMM((C)^+)|\cdot\log w(C^*)                         \\
		  & \le (\LOG \hat{\phi}(C^*)+\log(128n))\cdot \sum_{C\in\mathcal{T}_{\textsc{New}}^{>4}:\anc^*(C)=C^*} |\MMM(C^+)| \\
		  & \le \hat{s}\cdot |\MMM((C^*)^+)|\cdot (\LOG \hat{\phi}(C^*)+\log(128n))                                         \\
		  & \le \hat{s}\cdot |\MMM(2\BBB)|\cdot (\LOG \hat{\phi}(C^*)+\log(128n)),
	\end{align*}
	where we use the fact that, for any two components $C_1,C_2$ in the above sum, either $C_1^+\cap C_2^+=\emptyset$, $C_1^+\subset C_2^+$, or  $C_2^+\subset C_1^+$. We conclude that
	\begin{align*}
		  & \sum_{C\in\mathcal{T}_{\textsc{New}}^{>4}} |\MMM(C^+)|\cdot\log w(\anc^*(C))                                                             \\
		  & \hspace{0.5cm}	\le	 \hat{s}\cdot |\MMM(2\BBB)|\cdot\sum_{C^*\in \mathcal{T}_{\anc^*}}(\LOG \hat{\phi}(C^*)+\log(128n))                   \\
		  & \hspace{0.5cm}	 =O(s_{\max}\cdot\hat{s}\cdot n\log n\cdot\sum_{z_i\in\MMM^+(2\BBB)}(\LOG(z_i)+\log(128n))                                \\
		  & \hspace{0.5cm}	 =\tilde{O}(s_{\max}\cdot \hat{s}\cdot n\cdot(\LOG w(\BBB)+\LOG\Mea_{F}(2\BBB)+|\MMM(2\BBB)|))                            \\
		  & \hspace{0.5cm} =\tilde{O}(n\cdot(\LOG w(\BBB)+\LOG\Mea_{F}(2\BBB)+|\MMM(2\BBB)|+\log\sum_{z_i\in\MMM(2\BBB)}\LOG F'(z_i)^{-1})).\qedhere
	\end{align*}
\end{proof}

Let us summarize our results:

\begin{theorem}\label{thm:main1}
	Let $F$ be a polynomial as defined in (\ref{def:polyF}) and let $\mathcal{B}\subset\mathbb{C}$ be an arbitrary axis-aligned square. Then, the algorithm {\cisolate} with input $\BBB$ uses
	\begin{align}\label{bound:main2}
		  & \tilde{O}(n\cdot (n\cdot(\MMM(2\BBB)+\LOG\Mea_{F}(2\BBB)+\LOG(w(\BBB),w(\BBB)^{-1}))+\tau_F\cdot |\MMM(2\BBB)| \\ \nonumber
		  & \quad+\sum_{z_i\in \MMM(2\BBB)}(\LOG\sigma_F(z_i)^{-1}+\LOG F'(z_i)^{-1}))),
	\end{align}
	bit operations. As input, the algorithm requires an $L$-bit approximation of $F$ with
	\begin{align}\label{bound:precision}
		  & \tilde{O}(n\cdot(\MMM(2\BBB)+\LOG\Mea_{F}(2\BBB)+\LOG(w(\BBB),w(\BBB)^{-1}))+\tau_F\cdot |\MMM(2\BBB)| \\ \nonumber
		  & \quad+\sum_{z_i\in \MMM(2\BBB)}(\LOG\sigma_F(z_i)^{-1}+\LOG F'(z_i)^{-1})).
	\end{align}
\end{theorem}
\begin{proof}
	The bound (\ref{bound:main2}) on the bit complexity follows immediately from Lemma~\ref{lem:costTk}, Lemma~\ref{bound:sumZNC}, and the remark following Lemma~\ref{lem:costTk}.
	The bound (\ref{bound:precision}) on the precision demand follows directly from our considerations in the proof of Lemma~\ref{lem:costTk} and Lemma~\ref{allkcost}.
\end{proof}

Notice that the above complexity bounds are directly related to the size of the input box $\BBB$ as well as to parameters that only depend on the roots located in $2\BBB$. This makes our complexity bound adaptive in a very strong sense. In contrast, one might also be interested in (probably simpler) bounds when using our algorithm to isolate all complex roots of $F$? In this case, we may first compute an input box
$\BBB$ of width $w(\BBB)=2^{\Gamma+2}$ that is centered at the origin. Here,
$\Gamma\in\mathbb{N}_{\ge 1}$ is an integer bound for $\Gamma_F$ with $\Gamma=\Gamma_F+O(\log n)$. We have already argued in Section~\ref{sec:definitions} that such a bound $\Gamma$ can be computed using $\tilde{O}(n^2\Gamma_F)$ bit operations. Such a box $\BBB$ contains all complex roots of $F$, and thus running {\cisolate} with input $\BBB$ yields corresponding isolating disks. Hence, we obtain the following result:
\begin{corollary}
	Let $F$ be a square-free polynomial as in (\ref{def:polyF}). Then, for isolating all complex roots of $F$, {\cisolate} needs
	\begin{align}\label{bound:main3}
		  & =\tilde{O}(n\cdot(n^2+n \LOG \Mea_F+\sum_{i=1}^n \LOG F'(z_i)^{-1})) \\ \label{bound:main4}
		  & = \tilde{O}(n\cdot (n^2 + n \LOG\Mea_F +\LOG \Disc_F^{-1})).
	\end{align}
	bit operations, where $\Disc_F := |A_n|^{2n - 2} \prod_{1 \le i < j \le n} (z_j - z_i)^2$ is the \emph{discriminant of $F$}. As input, the algorithm requires an $L$-bit approximation of $F$ with
	\begin{align}\label{bound:precision2}
		\tilde{O}(n^2 +n \LOG \Mea_F + \LOG \Disc_F^{-1})
	\end{align}
\end{corollary}

\begin{proof}
	The above bounds follow directly from the bounds in (\ref{bound:main2}) and (\ref{bound:precision}), and the fact that
	$n\LOG(w(\BBB),w(\BBB)^{-1})+n\tau_F=O(n^2+n\LOG\Mea_F)$, $$\sum_{i=1}^n \LOG \sigma_F(z_i)^{-1}=O(n^2+n\LOG(\Mea_F)+\sum_{i=1}^n\LOG F'(z_i)^{-1}),$$ and $$\sum_{i=1}^{n}\LOG F'(z_i)^{-1}=O(n^2+n\LOG \Mea_F+\LOG \Disc_F^{-1});$$ e.g., see~\cite[Section 2.5]{Sagraloff2015} and the proof of~\cite[Theorem~31]{Sagraloff2015} for proofs of the latter bounds.
\end{proof}

Again, we provide simpler bounds for the special case, where the input polynomial $f$ has integer coefficients:

\begin{corollary}\label{thm:main2}
	Let $f\in\mathbb{Z}[x]$ be a square-free integer polynomial of degree $n$ with integer coefficients of bit size less than $\tau$, let $F:=f/\operatorname{lcf}(f)$, and let $\mathcal{B}\subset\mathbb{C}$ be an axis-aligned square with $2^{-O(\tau)}\le w(\BBB)\le 2^{O(\tau)}$. Then, $\mathbb{C}$\textsc{Isolate} with input $\BBB$ needs $\tilde{O}(n^3+n^2\tau)$ bit operations. The same bound also applies when using {\cisolate} to compute isolating disks for all roots of $f$.
\end{corollary}

\begin{proof}
	The claimed bound follows immediately from (\ref{bound:main4}) and the fact that $\Disc_F=\operatorname{lcf}(f)^{2n-2}\cdot \Disc_f\ge 2^{-(2n-2)\tau}$. For the second claim, we may simply run {\cisolate} on a box $\BBB$ of width $2^{\tau+2}$ centered at the origin. According to Cauchy's root bound, $\BBB$ contains all roots of $f$, and thus {\cisolate} yields corresponding isolating disks.
\end{proof}

\section{Conclusion}\label{sec:conclusion}

We proposed a simple and efficient subdivision algorithm to isolate the complex roots of a polynomial with arbitrary complex coefficients. Our algorithm achieves complexity bounds that are comparable to the best known bounds
for this problem, which
are achieved by methods based on fast polynomial factorization~\cite{Pan:survey,MSW-rootfinding2013,Pan:alg}. Compared to these methods,
our algorithm is quite simple and uses only fast algorithms for polynomial multiplication and Taylor shift computation but no other, more involved, asymptotically fast subroutines. Hence, also by providing a self-contained presentation and
pseudo-code for all subroutines of the algorithm, we hope that there will soon be implementations of our method. So far, we have not discussed a series of questions concerning an efficient implementation, including heuristics and filtering techniques to speed up the computations in practice. In particular, we remark that Graeffe iteration is not well suited for implementations that are restricted to single- or double precision arithmetic. This is due to the fact that, after only a few iterations, it produces very large intermediate values with exponents in the corresponding floating point representations that are outside of the allowed range of the IEEE 754 specifications. However, when using multi-precision arithmetic, which would be a natural choice when implementing our solver, this issue is no longer relevant. In fact, preliminary tests on a few examples show that, when using multi-precision arithmetic, the relative error in the $T^G_k$-test stays quite stable even after a number of Graeffe iterations.\footnote{Personal communication with Alexander Kobel from Max-Planck-Institute for Informatics in Saarbr\"ucken.} Hence, we are confident that a careful implementation of the $\mathbf{T_*}$ will turn out to be efficient.

Another possible direction of future research is to
extend our current Newton-bisection technique and
complexity analysis to the analytic roots algorithm
in~\cite{DBLP:conf/cie/YapS013}.
See~\cite{Strzebonski2012282} for an alternative approach for the
computation of the real roots of analytic functions obtained by composing polynomials and the functions $\log$, $\operatorname{exp}$, and $\operatorname{arctan}$.

At the end of Section~\ref{subsec:algorithm}, we sketched how to use our algorithm to isolate the roots of a not necessarily square-free polynomial for which the number of distinct complex roots is given as additional input. Also, we may use our algorithm to further refine the isolating disks for the roots of a polynomial in order to compute $L$-bit approximations of all roots. There exist dedicated methods~\cite{DBLP:journals/corr/abs-1104-1362,MSW-rootfinding2013,Pan:alg,pantsi:ISSAC13,DBLP:conf/snc/PanT14a} for refining intervals or disks, that are already known to be isolating for the roots of a polynomial. For large $L$, that is if $L$ dominates other parameters, their bit complexity is $\tilde{O}(nL)$. In comparison, using $\mathbb{C}\textsc{Isolate}$ for the refinement directly, its bit complexity would be of size $\tilde{O}(n^2L)$. We suspect that this bound can be further improved by using a proper modification of the $T_*$-test, which only needs to evaluate $F$ and its first derivative, and approximate multipoint evaluation.
We have not analyzed these extensions, however, we are confident that our approach yields similar bit complexity bounds as provided in~\cite{MSW-rootfinding2013} for the modified variant of Pan's method. This will be subject of future work.

\bibliography{bib}

\begin{thebibliography}{10}

\bibitem{abbott-quadratic}
J.~Abbott.
\newblock {Quadratic Interval Refinement for Real Roots}.
\newblock {\em Communications in Computer Algebra}, 28:3--12, 2014.
\newblock Poster presented at the International Symposium on Symbolic and
  Algebraic Computation (ISSAC), 2006.

\bibitem{akritas-strzebonski:comparison:05}
A.~G. Akritas and A.~Strzebo{\'n}ski.
\newblock A comparative study of two real root isolation methods.
\newblock {\em Nonlinear Analysis:Modelling and Control}, 10(4):297--304, 2005.

\bibitem{Becker12}
R.~Becker.
\newblock {The Bolzano Method to isolate the real roots of a bitstream
  polynomial}.
\newblock Bachelor thesis, Saarland University, Saarbr{\"u}cken, Germany, 2012.

\bibitem{Graeffe49}
G.~Best.
\newblock {Notes on the Graeffe method of root squaring}.
\newblock {\em American Mathematical Monthly}, 56(2):91--94, 1949.

\bibitem{Bini-Fiorentino}
D.~Bini and G.~Fiorentino.
\newblock {Design, Analysis, and Implementation of a Multiprecision Polynomial
  Rootfinder}.
\newblock {\em Numerical Algorithms}, 23:127--173, 2000.

\bibitem{DBLP:journals/jcam/BiniR14}
D.~A. Bini and L.~Robol.
\newblock Solving secular and polynomial equations: A multiprecision algorithm.
\newblock {\em Journal of Computational and Applied Mathematics}, 272:276 --
  292, 2014.

\bibitem{DBLP:journals/jsc/BurrK12}
M.~Burr and F.~Krahmer.
\newblock Sqfreeeval: An (almost) optimal real-root isolation algorithm.
\newblock {\em Journal of Symbolic Computation}, 47(2):153--166, 2012.

\bibitem{Collins}
G.~E. Collins.
\newblock Continued fraction real root isolation using the {H}ong bound.
\newblock {\em Journal of Symbolic Computation}, 2014.
\newblock in press.

\bibitem{Collins-Akritas}
G.~E. Collins and A.~G. Akritas.
\newblock Polynomial real root isolation using {D}escartes' rule of signs.
\newblock In {\em Symposium on symbolic and algebraic computation (SYMSAC)},
  pages 272--275, 1976.

\bibitem{Collins:1992:EAI:143242.143308}
G.~E. Collins and W.~Krandick.
\newblock An efficient algorithm for infallible polynomial complex root
  isolation.
\newblock In {\em Papers from the International Symposium on Symbolic and
  Algebraic Computation}, ISSAC '92, pages 189--194, New York, NY, USA, 1992.
  ACM.

\bibitem{davenport:85}
J.~H. Davenport.
\newblock Computer algebra for cylindrical algebraic decomposition.
\newblock {T}ech.\ {R}ep., The Royal Inst. of Technology, Dept. of Numerical
  Analysis and Computing Science, S-100 44, Stockholm, Sweden, 1985.
\newblock Reprinted as Tech.~Report 88-10 , School of Mathematical Sci., U. of
  Bath, Claverton Down, Bath BA2 7AY, England. URL
  http://www.bath.ac.uk/\textasciitilde masjhd/TRITA.pdf.

\bibitem{du-sharma-yap:sturm:07}
Z.~Du, V.~Sharma, and C.~K. Yap.
\newblock {\em Amortized Bound for Root Isolation via Sturm Sequences}, pages
  113--129.
\newblock Birkh{\"a}user Basel, Basel, 2007.

\bibitem{eigenwillig-phd}
A.~Eigenwillig.
\newblock {\em Real Root Isolation for Exact and Approximate Polynomials Using
  {D}escartes' Rule of Signs}.
\newblock PhD thesis, Saarland University, 2008.

\bibitem{Eigenwillig2005}
A.~Eigenwillig, L.~Kettner, W.~Krandick, K.~Mehlhorn, S.~Schmitt, and
  N.~Wolpert.
\newblock {\em A Descartes Algorithm for Polynomials with Bit-Stream
  Coefficients}, pages 138--149.
\newblock Springer Berlin Heidelberg, Berlin, Heidelberg, 2005.

\bibitem{Pan:survey}
I.~Z. Emiris, V.~Y. Pan, and E.~P. Tsigaridas.
\newblock Algebraic algorithms.
\newblock In {\em Computing Handbook, Third Edition: Computer Science and
  Software Engineering}, pages 10: 1--30. 2014.

\bibitem{DBLP:journals/jsc/Fortune02}
S.~Fortune.
\newblock An iterated eigenvalue algorithm for approximating roots of
  univariate polynomials.
\newblock {\em Journal of Symbolic Computation}, 33(5):627--646, 2002.

\bibitem{Giusti2005}
M.~Giusti, G.~Lecerf, B.~Salvy, and J.-C. Yakoubsohn.
\newblock On location and approximation of clusters of zeros of analytic
  functions.
\newblock {\em Foundations of Computational Mathematics}, 5(3):257--311, 2005.

\bibitem{henrici1974}
P.~Henrici.
\newblock {\em Applied and computational complex analysis: Special functions,
  integral transforms, asymptotics, continued fractions}.
\newblock Pure and applied mathematics. J. Wiley \& Sons, 1974.

\bibitem{householder-graeffe:59}
A.~S. Householder.
\newblock Dandelin, {L}obacevskii, or {G}raeffe.
\newblock {\em The American Mathematical Monthly}, 66(6):464--466, 1959.

\bibitem{DBLP:journals/corr/abs-1104-1362}
M.~Kerber and M.~Sagraloff.
\newblock Root refinement for real polynomials using quadratic interval
  refinement.
\newblock {\em Journal of Computational and Applied Mathematics}, 280:377--395,
  2015.
\newblock A preliminary version appeared in the proceedings of the
  International Symposium on Symbolic and Algebraic Computation (ISSAC), 2011.

\bibitem{Kirrinnis1998378}
P.~Kirrinnis.
\newblock {Partial Fraction Decomposition in C(z) and Simultaneous Newton
  Iteration for Factorization in C[z]}.
\newblock {\em Journal of Complexity}, 14(3):378 -- 444, 1998.

\bibitem{DBLP:conf/issac/KobelRS16}
A.~Kobel, F.~Rouillier, and M.~Sagraloff.
\newblock Computing real roots of real polynomials ... and now for real!
\newblock In {\em Proceedings of the {ACM} on International Symposium on
  Symbolic and Algebraic Computation, {ISSAC} 2016, Waterloo, ON, Canada, July
  19-22, 2016}, pages 303--310, 2016.

\bibitem{DBLP:journals/corr/abs-1304.8069}
A.~Kobel and M.~Sagraloff.
\newblock Fast approximate polynomial multipoint evaluation and applications.
\newblock {\em CORR}, abs/1304.8069, 2013.

\bibitem{McNamee:2002}
J.~M. McNamee.
\newblock A 2002 update of the supplementary bibliography on roots of
  polynomials.
\newblock {\em Journal of Computational and Applied Mathematics},
  142(2):433--434, 2002.

\bibitem{McNamee2007}
J.~M. McNamee.
\newblock {\em Numerical Methods for Roots of Polynomials}.
\newblock Number~1 in Studies in Computational Mathematics. Elsevier Science,
  2007.

\bibitem{McNamee2012239}
J.~M. McNamee and V.~Y. Pan.
\newblock Efficient polynomial root-refiners: A survey and new record
  efficiency estimates.
\newblock {\em Computers \& Mathematics with Applications}, 63(1):239 -- 254,
  2012.

\bibitem{McNamee-Pan2013}
J.~M. McNamee and V.~Y. Pan.
\newblock {\em Numerical Methods for Roots of Polynomials}.
\newblock Number~2 in Studies in Computational Mathematics. Elsevier Science,
  2013.

\bibitem{MSW-rootfinding2013}
K.~Mehlhorn, M.~Sagraloff, and P.~Wang.
\newblock {From approximate factorization to root isolation with application to
  cylindrical algebraic decomposition}.
\newblock {\em Journal of Symbolic Computation}, 66:34--69, 2015.
\newblock A preliminary version appeared in the proceedings of the
  International Symposium on Symbolic and Algebraic Computation (ISSAC), 2013.

\bibitem{mt-mega-2009}
B.~Mourrain and E.~P. Tsigaridas.
\newblock On the complexity of complex root isolation.
\newblock In {\em Proc. 10th Int. Symp. on Effective Methods in Algebraic
  Geometry (MEGA)}, Barcelona, Spain, 2009.

\bibitem{Mourrain2002612}
B.~Mourrain, M.~Vrahatis, and J.~Yakoubsohn.
\newblock On the complexity of isolating real roots and computing with
  certainty the topological degree.
\newblock {\em Journal of Complexity}, 18(2):612 -- 640, 2002.

\bibitem{Neff199681}
C.~Neff and J.~H. Reif.
\newblock An efficient algorithm for the complex roots problem.
\newblock {\em Journal of Complexity}, 12(2):81 -- 115, 1996.

\bibitem{Pan:alg}
V.~Pan.
\newblock {Univariate Polynomials: Nearly Optimal Algorithms for Numerical
  Factorization and Root Finding}.
\newblock {\em Journal of Symbolic Computation}, 33(5):701--733, 2002.

\bibitem{Pan:history}
V.~Y. Pan.
\newblock Solving a polynomial equation: Some history and recent progress.
\newblock {\em SIAM Review}, 39(2):187--220, 1997.

\bibitem{Pan2000213}
V.~Y. Pan.
\newblock {A}pproximating {C}omplex {P}olynomial {Z}eros: {M}odified {W}eyl's
  {Q}uadtree {C}onstruction and {I}mproved {N}ewton's {I}teration.
\newblock {\em Journal of Complexity}, 16(1):213 -- 264, 2000.

\bibitem{pantsi:ISSAC13}
V.~Y. Pan and E.~P. Tsigaridas.
\newblock On the boolean complexity of real root refinement.
\newblock In {\em Proceedings of the 38th International Symposium on Symbolic
  and Algebraic Computation}, ISSAC '13, pages 299--306, New York, NY, USA,
  2013. ACM.

\bibitem{DBLP:conf/snc/PanT14a}
V.~Y. Pan and E.~P. Tsigaridas.
\newblock Accelerated approximation of the complex roots of a univariate
  polynomial.
\newblock In {\em Proceedings of the 2014 Symposium on Symbolic-Numeric
  Computation}, SNC '14, pages 132--134, New York, NY, USA, 2014. ACM.

\bibitem{pawlowski1999}
P.~Pawlowski.
\newblock The location of the zeros of the higher order derivatives of a
  polynomial.
\newblock {\em Proceedings of the American Mathematical Society}, 127(5):pp.
  1493--1497, 1999.

\bibitem{Pinkert:1976}
J.~R. Pinkert.
\newblock An exact method for finding the roots of a complex polynomial.
\newblock {\em ACM Trans. Math. Softw.}, 2(4):351--363, Dec. 1976.

\bibitem{rahman2002analytic}
Q.~Rahman and G.~Schmeisser.
\newblock {\em {Analytic Theory of Polynomials}}.
\newblock London Mathematical Society monographs. Clarendon Press, 2002.

\bibitem{DBLP:journals/jc/Renegar87}
J.~Renegar.
\newblock On the worst-case arithmetic complexity of approximating zeros of
  polynomials.
\newblock {\em Journal of Complexity}, 3(2):90--113, 1987.

\bibitem{rouillier-zimmermann:roots:04}
F.~Rouillier and P.~Zimmermann.
\newblock Efficient isolation of polynomial's real roots.
\newblock {\em Journal of Computational and Applied Mathematics}, 162(1):33 --
  50, 2004.
\newblock Proceedings of the International Conference on Linear Algebra and
  Arithmetic 2001.

\bibitem{Sagraloff:2012:NMD:2442829.2442872}
M.~Sagraloff.
\newblock When {N}ewton meets {D}escartes: A simple and fast algorithm to
  isolate the real roots of a polynomial.
\newblock In {\em Proceedings of the 37th International Symposium on Symbolic
  and Algebraic Computation}, ISSAC '12, pages 297--304, New York, NY, USA,
  2012. ACM.

\bibitem{Sagraloff:2014:NAC:2608628.2608632}
M.~Sagraloff.
\newblock A near-optimal algorithm for computing real roots of sparse
  polynomials.
\newblock In {\em Proceedings of the 39th International Symposium on Symbolic
  and Algebraic Computation}, ISSAC '14, pages 359--366, New York, NY, USA,
  2014. ACM.

\bibitem{Sagraloff2014DSC}
M.~Sagraloff.
\newblock {On the complexity of the Descartes method when using approximate
  arithmetic}.
\newblock {\em Journal of Symbolic Computation}, 65(0):79 -- 110, 2014.

\bibitem{Sagraloff2015}
M.~Sagraloff and K.~Mehlhorn.
\newblock Computing real roots of real polynomials.
\newblock {\em Journal of Symbolic Computation}, 73:46 -- 86, 2016.

\bibitem{Yap:2011:SBE:1993886.1993938}
M.~Sagraloff and C.~K. Yap.
\newblock A simple but exact and efficient algorithm for complex root
  isolation.
\newblock In {\em Proceedings of the 36th International Symposium on Symbolic
  and Algebraic Computation}, ISSAC '11, pages 353--360, New York, NY, USA,
  2011. ACM.

\bibitem{Schoenhage82}
A.~Sch{\"o}nhage.
\newblock Asymptotically fast algorithms for the numerical multiplication and
  division of polynomials with complex coeficients.
\newblock In J.~Calmet, editor, {\em Computer Algebra, {EUROCAM} '82, European
  Computer Algebra Conference, Marseille, France, 5-7 April, 1982,
  Proceedings}, volume 144 of {\em Lecture Notes in Computer Science}, pages
  3--15. Springer, 1982.

\bibitem{schonhage:fundamental}
A.~Sch{\"o}nhage.
\newblock {The Fundamental Theorem of Algebra in Terms of Computational
  Complexity}.
\newblock Technical report, Math. Inst. Univ. T\"{u}bingen, 1982.

\bibitem{Schroder1870}
E.~Schr{\"o}der.
\newblock {\"Uber unendliche viele {A}lgorithmen zur {A}ufl\"osung der
  Gleichungen}.
\newblock {\em Mathematische Annalen}, 2:317--365, 1870.

\bibitem{sharma}
V.~Sharma.
\newblock Complexity of real root isolation using continued fractions.
\newblock {\em Theoretical Computer Science}, 409:292--310, 2008.

\bibitem{Sharma:2015:NOS:2755996.2756656}
V.~Sharma and P.~Batra.
\newblock Near optimal subdivision algorithms for real root isolation.
\newblock In {\em Proceedings of the 2015 ACM on International Symposium on
  Symbolic and Algebraic Computation}, ISSAC '15, pages 331--338, New York, NY,
  USA, 2015. ACM.

\bibitem{Strzebonski2012282}
A.~Strzebo\'nski.
\newblock Real root isolation for exp--log--arctan functions.
\newblock {\em Journal of Symbolic Computation}, 47(3):282 -- 314, 2012.

\bibitem{tsigaridas13}
E.~P. Tsigaridas.
\newblock Improved bounds for the cf algorithm.
\newblock {\em Theoretical Computer Science}, 479:120 -- 126, 2013.

\bibitem{te-cf:08}
E.~P. Tsigaridas and I.~Z. Emiris.
\newblock On the complexity of real root isolation using continued fractions.
\newblock {\em Theoretical Computer Science}, 392(1-3):158--173, 2008.

\bibitem{turan1984new}
P.~Tur{\'a}n, G.~Hal{\^a}asz, and J.~Pintz.
\newblock {\em On a new method of analysis and its applications}.
\newblock A Wiley-Interscience publication. Wiley-Interscience, 1984.

\bibitem{vzGG03}
J.~von~zur Gathen and J.~Gerhard.
\newblock {\em Modern Computer Algebra}.
\newblock Cambridge University Press, Cambridge, UK, 2nd edition edition, 2003.

\bibitem{Weyl}
H.~Weyl.
\newblock {Randbemerkungen zu Hauptproblemen der Mathematik. II.
  Fundamentalsatz der Algebra and Grundlagen der Mathematik}.
\newblock {\em Mathematische Zeitschrift}, 20:131--151, 1924.

\bibitem{Wilf:1978}
H.~S. Wilf.
\newblock {A Global Bisection Algorithm for Computing the Zeros of Polynomials
  in the Complex Plane}.
\newblock {\em J. ACM}, 25(3):415--420, July 1978.

\bibitem{Yakoubsohn2000}
J.-C. Yakoubsohn.
\newblock Finding a cluster of zeros of univariate polynomials.
\newblock {\em Journal of Complexity}, 16(3):603 -- 638, 2000.

\bibitem{Yakoubsohn2005652}
J.-C. Yakoubsohn.
\newblock Numerical analysis of a bisection-exclusion method to find zeros of
  univariate analytic functions.
\newblock {\em Journal of Complexity}, 21(5):652 -- 690, 2005.

\bibitem{DBLP:conf/cie/YapS013}
C.~Yap, M.~Sagraloff, and V.~Sharma.
\newblock {Analytic Root Clustering: {A} Complete Algorithm Using Soft Zero
  Tests}.
\newblock In P.~Bonizzoni, V.~Brattka, and B.~L{\"{o}}we, editors, {\em The
  Nature of Computation. Logic, Algorithms, Applications - 9th Conference on
  Computability in Europe, CiE 2013, Milan, Italy, July 1-5, 2013.
  Proceedings}, volume 7921 of {\em Lecture Notes in Computer Science}, pages
  434--444. Springer, 2013.

\bibitem{yap-fundamental}
C.~K. Yap.
\newblock {\em {Fundamental Problems of Algorithmic Algebra}}.
\newblock Oxford University Press, 2000.

\end{thebibliography}
\bibliographystyle{abbrv}


\section{Appendix}
\subsection{Missing proofs in Section~\ref{sec:pellet}}\label{appendix:Graeffe}
We split the proof of Theorem~\ref{tktest thm} into two technical lemmas:
\begin{lemma}
	Let $\Delta:=\Delta(m,r)$ be a disk that is $(1,4c_2\cdot\MAX(k)\cdot n^3)$-isolating for the roots $z_{1},\ldots,z_k$, then,
	for all $z\in c_2 n^2\cdot\Delta$, it holds that $F^{(k)}(z)\neq 0$. Furthermore,
	\[
		\sum_{i=k+1}^{n} \left| \frac{ F^{(i)}(m) (c_1n\cdot r)^{i-k} k! } { F^{(k)}(m) i!} \right|
		< \frac{1}{2K}.
	\]
\end{lemma}

\begin{proof}
	\begin{enumerate}
		\item For the first part, we may assume that $k\ge 1$. Then, for the $k$'th derivative of $F$ and any complex $z$ that is not a root of $F$, it holds that
		      \begin{align*}
		      	\frac{F^{(k)}(z)}{F(z)}
		      	=
		      	\sum_{J \in \binom{[n]}{k}}
		      	\prod_{j\in J} \frac{1}{z-z_j}
		      	=
		      	\prod_{j=1}^{k}\frac{1}{z-z_j}
		      	+
		      	\sum_{J \in \binom{[n]}{k}, J\neq [k]}
		      	\prod_{j\in J} \frac{1}{z-z_j}.
		      \end{align*}
		      By way of contradiction, assume
		      $F^{(k)}(z)=0$ for some $z\in c_2n^2\cdot\Delta$.  Then,
		      \[
		      	\prod_{j=1}^{k}\frac{1}{|z-z_j|}
		      	\le
		      	\sum_{J \in \binom{[n]}{k}, J\neq [k]}
		      	\prod_{j\in J} \frac{1}{|z-z_j|}.
		      \]
		      Assuming $k\le n/2$, we proceed to show
		      \begin{align*}
		      	1\le\sum_{J \in \binom{[n]}{k}, J\neq [k]}
		      	  & \frac{|z-z_1|\cdots |z-z_k|}{\prod_{j\in J}|z-z_j|} \\
		      	  & =
		      	\sum_{k'=0}^{k-1}
		      	\sum_{J\in\binom{[k]}{k'}}\sum_{J'\in\binom{[n]\setminus [k]}{k-k'}}
		      	\frac{\prod_{i=1}^k|z-z_i|}
		      	{\prod_{i\in J} |z-z_{i}| \cdot \prod_{j\in J'} |z-z_{j}|}\\
		      	  & \le
		      	\sum_{k'=0}^{k-1} \binom{k}{k'}\binom{n-k}{k-k'}
		      	\left(\frac{2c_2n^2r}{4c_2kn^3r - c_2n^2r }\right)^{k-k'}\\
		      	  & \le
		      	\sum_{k'=0}^{k-1} \binom{k}{k'}\binom{n-k}{k-k'}
		      	\left(\frac{1}{2kn}\right)^{k-k'}\\
		      	  & \le
		      	\sum_{k'=0}^{k-1} \frac{k^{k-k'}}{(k-k')!}(n-k)^{k-k'}
		      	\left(\frac{1}{2kn}\right)^{k-k'}\\
		      	  & \le
		      	\sum_{k'=0}^{k-1} (1/2)^{k-k'}/{(k-k')!}\\
		      	  & < e^{1/2}-1 <1,
		      	\text{ a contradiction.}
		      \end{align*}

		      The preceding argument assumes $k\le n/2$ because
		      this allows us to freely choose $J$ and $J'$ as indicated.
		      But suppose $k>n/2$.   We then have
		      \begin{align*}
		      	1\le\sum_{I \in \binom{[n]}{k}\setminus [k]}
		      	  & \frac{|z-z_1|\cdots |z-z_k|}{|z-z_{i_1}|\cdots |z-z_{i_k}|} \\
		      	  & =
		      	\sum_{k'=1}^{n-k}
		      	\sum_{J\in\binom{[k]}{k-k'}}\sum_{J'\in\binom{[n]\setminus [k]}{k'}}
		      	\frac{\prod_{i=1}^k|z-z_i|}
		      	{\prod_{i\in J} |z-z_{i}| \cdot \prod_{j\in J'} |z-z_{j}|}\\
		      	  & \le
		      	\sum_{k'=1}^{n-k} \binom{k}{k-k'}\binom{n-k}{k'}
		      	\left(\frac{2c_2n^2r}{4c_2kn^3r - c_2n^2r }\right)^{k'}\\
		      	  & <
		      	\sum_{k'=1}^{n-k} \binom{k}{k-k'}\binom{n-k}{k'}
		      	\left(\frac{2}{3kn}\right)^{k'}\\
		      	  & \le
		      	\sum_{k'=1}^{n-k} \frac{k^{k'}}{k'!}(n-k)^{k'}
		      	\left(\frac{2}{3kn}\right)^{k'}\\
		      	  & \le
		      	\sum_{k'=1}^{n-k} \frac{1}{k'!}\left(\frac{2}{3}\right)^{k'}\le e^{2/3}-1
		      	< 1,
		      	\text{ again a contradiction.}
		      \end{align*}

		\item Similar as above, with $z^{(k)}_1,\ldots , z^{(k)}_{n-k}$ denoting the roots of $F^{(k)}$, it holds that
		      \begin{align*}
		      	\left| \frac{F^{(k+i)}(m)}{F^{(k)}(m)} \right|
		      	\le
		      	\sum_{J\in\binom{[n-k]}{i}}\prod_{j\in J}
		      	\frac{1}{ |m-z^{(k)}_{j}|}
		      	\le
		      	\frac{\binom{n-k}{i}}{(c_2n^2r)^i},
		      \end{align*}
		      and thus
		      \begin{align*}
		      	\sum_{i=k+1}^{n} \left| \frac{ F^{(i)}(m) (c_1nr)^{i-k} k! } { F^{(k)}(m) i!} \right|
		      	  & \le
		      	\sum_{i=1}^{n-k} \left| \frac{F^{(k+i)}(m)}{F^{(k)}(m)} \right| \frac{(c_1nr)^i}{i!}
		      	\qquad |\text{ since } \frac{k!}{(k+i)!}\le \frac{1}{i!}\\
		      	  & \le
		      	\sum_{i=1}^{n-k} \frac{\binom{n-k}{i}}{c_2^in^{2i}r^i} \frac{(c_1nr)^i}{i!}\\
		      	  & <
		      	\sum_{i=1}^{n-k} \Big( \frac{c_1}{c_2} \Big)^i \frac{1}{i!}
		      	\le
		      	e^{c_1/c_2} - 1
		      	\le \frac{1}{2K}
		      	,
		      \end{align*}
		      where we used (\ref{def:constants}) for the last inequality.\qedhere
	\end{enumerate}
\end{proof}

\begin{lemma}
	Le $\lambda$ be a real value with $\lambda\ge 16K\cdot\MAX(k)^2\cdot n$ and suppose that $\Delta:=\Delta(m,r)$ is a disk that is $(1,\lambda)$-isolating for the roots $z_1,\ldots, z_k$ of $F$, then
	\[
		\sum_{i<k} \frac{| F^{(i)}(m)|}{| F^{(k)}(m) |} \frac{ (c_1n\cdot r)^{i-k}k!}{i!}
		<
		\frac{1}{2K}.
	\]
\end{lemma}
\begin{proof}
	We may assume that $k\ge 1$. Write $F(x)=G(x) H(x)$ with $G(x)=\prod_{i=1}^{k} (x-z_i)$ and
	$H(x)=\prod_{j=k+1}^n (x-z_j)$. By induction, one shows that
	$F^{(i)}(x) = \sum_{j=0}^i \binom{i}{j} G^{(i-j)}(x)H^{(j)}(x)$ and
	$
	F^{(k)}(x) = k! \sum_{I\in\binom{[n]}{k}} \prod_{i\notin I}(x-z_i)=k!\cdot\sum_{J\in\binom{[n]}{n-k}} \prod_{j\in J} (x-z_j)$. It follows that

	\begin{align*}
		|F^{(k)}(m)|
		  & =
		k!\cdot \Big|\sum_{J\in\binom{[n]}{n-k}} \prod_{j\in J} (m-z_j) \Big|
		=\\
		  & k! \cdot|H(m)| \cdot\Big|\sum_{J\in\binom{[n]}{n-k}} \frac{\prod_{j\in J} (m-z_j)}{\prod_{i=k+1}^n |m-z_i|}\Big| \\
		  & \ge
		k!\cdot |H(m)| \cdot \Big(1 - \sum_{J\in\binom{[n]}{n-k}:J\neq \{k+1,\ldots,n\} } \frac{\prod_{j\in J} |m-z_j|}{\prod_{i=k+1}^n |m-z_i|}\Big)\\
		  & \ge
		k! \cdot|H(m)| \cdot \Big(1 - \sum_{j=1}^{\min(k,n-k)} \sum_{J_1,J_2:J_1\subset[k]:|J_1|=j\text{ and } J_2\subset [n]\setminus [k]:|J_2|=n-k-j}\frac{\prod_{j\in J_1} |m-z_j|}{(\lambda r)^j}\Big)\\
		  & \ge
		k! \cdot|H(m)| \cdot \Big(1 - \sum_{j=1}^{\min(k,n-k)} \binom{k}{j}\binom{n-k}{n-k-j} \frac{r^j}{(\lambda r)^j}\Big)\\
		  & \ge
		k! \cdot|H(m)| \cdot \Big(2 - \sum_{j=0}^{n} k^j\binom{n}{j} \frac{1}{\lambda^j}\Big)\\
		  & \ge
		k!\cdot|H(m)| \cdot \Big(2 - \left(1 + \frac{k}{\lambda}\right)^n\Big)\\
		  & \ge
		k! \cdot|H(m)| \cdot \Big(2 - e^{\frac{1}{4}}\Big)\\
		  & \ge
		\frac{k!\cdot |H(m)|}{2}.
	\end{align*}
	For $G$, we have
	$G^{(i)}(x) = i!\sum_{J\in\binom{[k]}{i}} \prod_{j\notin J}(x-z_j)$, and thus $|G^{(i)}(m)|\le i!\binom{k}{i} r^{k-i}$. In addition,
	\[
		\left| \frac{H^{(i)}(m)}{H(m)} \right|
		\le
		\sum_{J\in \binom{[n]}{i}} \prod_{j\in J}\frac{1}{|m-z_{j}|}
		\le
		i!\cdot\binom{n-k}{i}\frac{1}{(\lambda r)^i}
	\]
	and thus
	\begin{align*}
		|G^{(i-j)}(m)H^{(j)}(m)|
		  & \le
		|H(m)|\cdot(i-j)!\binom{k}{i-j}r^{k-(i-j)}
		\cdot j!\binom{n-k}{j} \frac{1}{(\lambda r)^j}\\
		  & =
		|H(m)|\cdot(i-j)!j!\binom{k}{i-j}\binom{n-k}{j}
		\cdot  \frac{1}{\lambda ^j}r^{k-i}.
	\end{align*}
	Hence, using that $\lambda\ge 16Kk^2\cdot n$ and $c_1n\ge \frac{k}{\ln\left(1+\frac{1}{8K}\right)}$, yields
	\begin{align*}
		  & \sum_{i=0}^{k-1} \frac{| F^{(i)}(m)|}{| F^{(k)}(m) |} \frac{ (c_1nr)^{i-k}k!}{i!}
		\le
		\sum_{i=0}^{k-1} \sum_{j=0}^i \frac{| G^{(i-j)}(m)H^{(j)}(m)|}{| F^{(k)}(m) |} \frac{\binom{i}{j} (c_1nr)^{i-k}k!}{i!}\\
		  & \le
		\sum_{i=0}^{k-1} \sum_{j=0}^i \frac{|H(m)|}{| F^{(k)}(m) |}(i-j)!j!\binom{k}{i-j} \binom{n-k}{j}\binom{i}{j}\frac{(c_1n)^{i-k}}{\lambda ^j} \frac{ k!}{i!}\\
		  & \le
		2(c_1n)^{-k} \sum_{i=0}^{k-1} \binom{k}{i}(c_1n)^i
		+ 2\sum_{i=1}^{k-1}\sum_{j=1}^i \binom{k}{i-j} \binom{n-k}{j}\frac{(c_1n)^{i-k}}{\lambda^j}\\
		  & \le
		2(c_1n)^{-k} \sum_{i=0}^{k-1} \binom{k}{i}(c_1n)^i
		+ 2\sum_{i=1}^{k-1}\sum_{j=1}^i  \binom{n-k}{j}\frac{k^j}{\lambda^j} \ln^{k-i}\left(1+\frac{1}{8K}\right)\\
		  & =
		2(c_1n)^{-k} \sum_{i=0}^{k-1} \binom{k}{i}(c_1n)^i
		+ \frac{k-1}{8Kk} + \sum_{i=2}^{k-1}\sum_{j=2}^i \frac{2}{(16Kk)^j} \\
		  & <
		2\left(\sum_{j=0}^{k} \binom{k}{j}(c_1n)^{-j} - 1\right) + \frac{1}{4K}
		\le
		2\left(e^{k/(c_1n)} - 1\right) + \frac{1}{4K}
		\le
		\frac{1}{2K}
		.\qedhere
	\end{align*}
\end{proof}

\subsection{Missing Proofs in Section~\ref{sec:Graeffe}}\label{appendix:graeffe}

\begin{theorem*}[Restatement of Theorem~\ref{graeffe thm}]
	Denote the roots of $F$ by $z_1,\ldots,z_n$, then it holds that $F^{[1]}(x)=\sum_{i=0}^n a_i^{[1]}x^i = a_n^2\cdot \prod_{i=1}^n(x-z_i^2)$. In particular, the roots of the first Graeffe iterate $F^{[1]}$ are the squares of the roots of $F$. In addition, we have
	\[
		n^2\cdot \MAX(\|F\|_\infty)^2\ge\|F^{[1]}\|_\infty\ge \|F\|_\infty^2 \cdot 2^{-4n}.
	\]
\end{theorem*}

\begin{proof}
	Notice that $a_n^{[1]}=a_n^2$ follows
	directly from the definition of $F^{[1]}$. Furthermore, we have
	\begin{align*}
		F^{[1]}(z_i^2)
		  & =(-1)^n \cdot[F_e(z_i^2)^2 - z_i^2\cdot F_o(z_i^2)^2]                                   \\
		  & =(-1)^n \cdot[F_e(z_i^2) - z_i\cdot F_o(z_i^2)]\cdot [F_e(z_i^2) + z_i\cdot F_o(z_i^2)] \\
		  & =(-1)^n \cdot[F_e(z_i^2) - z_i\cdot F_o(z_i^2)]\cdot F(z_i)=0.
	\end{align*}
	Going from $F$ to an arbitrary small perturbation $\tilde{F}$ (which has only simple roots and for which $\tilde{z}_i^2\neq \tilde{z}_j^2$ for all pairs of distinct roots $\tilde{z}_i$ and $\tilde{z}_j$ of $\tilde{F}$), we conclude that each root $z_i^2$ of $F^{[1]}$ has multiplicity $\operatorname{mult}(z_i^2,F^{[1]})=\sum_{j:z_j^2=z_i^2}\operatorname{mult}(z_j,F)$. Hence, the first claim follows. For the second claim, notice that the left inequality follows immediately from the fact that each coefficient of $F^{[1]}$ is the sum of at most $n^2$ many products of the form $\pm a_i\cdot a_j$, and each of these products has absolute value smaller than or equal to $\MAX(\|F\|_\infty)^2$. For the right inequality we have to work harder: W.l.o.g., we may assume that $|z_i|<2$ for $i=1,\ldots,k$, and that $|z_i|\ge 2$ for $i=k+1,\ldots,n$. Let $z_{\max}$ be a point in the closure of the unit disk $\Delta(0,1)$ such that $|F(z_{\max})|=\max_{z:|z|\le 1}|F(z)|$. Since $F$ takes its maximum on the boundary of $\Delta(0,1)$, we must have $|z_{\max}|=1$, and using Cauchy's Integral Theorem to write the coefficients of $F$ in terms of an integral, we conclude that $|F(z_{\max)}|\ge \|F\|_\infty$. In addition, it holds that $|F(z_{\max})|\le \sum_{i=0}^n |a_i|\cdot |z_{\max}|^n=\sum_{i=0}^n |a_i|\le (n+1)\cdot \|F\|_\infty$, and thus
	\begin{align}\label{formula:FinftyDelta}
		\|F\|_\infty\le |F(z_{\max})|=\max_{z:|z|\le 1}|F(z)|\le (n+1)\cdot \|F\|_\infty.
	\end{align}
	Applying the latter result to the polynomial $g(x):=\prod_{i=1}^k (z-z_i^2)$ yields the existence of a point $z'$ with $|z'|=1$ and $|g(z')|\ge 1$. Hence, it follows that
	\begin{align*}
		|F^{[1]}(z')| & =|a_n|^2\prod_{i=1}^k|z'-z_i^2| \prod_{i=k+1}^n|z'-z_i^2|
		\ge |a_n|^2\prod_{i=k+1}^n |(\sqrt{z'}-z_i)\cdot(\sqrt{z'}+z_i)|\\
		              & \ge |a_n|^2\cdot \prod_{i=1}^k \frac{|z_{\max}-z_i|^2}{9} \cdot \prod_{i=k+1}^n \frac{|z_{\max}-z_i|^2}{9}
		\ge \frac{|F(z_{\max})|^2}{9^n},
	\end{align*}
	where we used that $|x-y|<3$ for arbitrary complex points $x,y$ with $|x|=1$ and $|y|<2$, and that $|x-z|\ge \frac{|y-z|}{3}$ for arbitrary complex points $x,y,z$ with $|x|=|y|=1$ and $|z|\ge 2$. We conclude that
	\[
		\|F^{[1]}\|_\infty\ge \frac{|F^{[1]} (z')|}{n+1}\ge \frac{|F(z_{\max})|^2}{(n+1)\cdot 9^n}\ge\|F\|_\infty^2\cdot 2^{-4n}. \qedhere
	\]
\end{proof}

\end{document}